\documentclass[a4paper]{amsart}
\usepackage[left=2.7cm,right=2.7cm,top=3.5cm,bottom=3cm]{geometry}
\usepackage{amsthm,amssymb,amsmath,amsfonts,mathrsfs,amscd,dsfont,verbatim,tikz-cd,graphicx,latexsym,longtable,mathtools,footnote,footmisc,manyfoot,enumitem,fancyhdr,upgreek,latexsym,longtable,mathtools}
\usepackage[T1]{fontenc}
\usepackage[all,cmtip]{xy}
\usepackage{hyperref}
\usepackage{graphicx}
\usepackage[scr=boondox]{mathalpha}

\numberwithin{equation}{section}
\mathtoolsset{showonlyrefs}

\newfont{\cyr}{wncyr10 scaled 1100}
\newfont{\cyrr}{wncyr9 scaled 1000}

\makeatletter
\newcommand{\smallbullet}{}
\DeclareRobustCommand\smallbullet{
  \mathord{\mathpalette\smallbullet@{0.5}}
}
\newcommand{\smallbullet@}[2]{
  \vcenter{\hbox{\scalebox{#2}{$\m@th#1\bullet$}}}
}
\makeatother

\theoremstyle{definition}
\newtheorem{theorem}{Theorem}[section]
\newtheorem*{theorem*}{Theorem}
\newtheorem{proposition}[theorem]{Proposition}
\newtheorem{lemma}[theorem]{Lemma}
\newtheorem{corollary}[theorem]{Corollary}

\newtheorem{definition}[theorem]{Definition}

\newtheorem{remark}[theorem]{Remark}

\newcommand{\Q}{\mathbb{Q}}

\newcommand{\Z}{\mathbb{Z}}
\newcommand{\R}{\mathbb{R}}
\newcommand{\C}{\mathbb{C}}

\newcommand{\F}{\mathbb{F}}

\newcommand{\PP}{\mathbb{P}}
\newcommand{\G}{\mathbb{G}}
\newcommand{\A}{\mathbb{A}}

\newcommand{\X}{\mathbb{X}}

\newcommand{\defeq}{\vcentcolon=}

\newcommand{\pwseries}[1]{[\![ #1]\!]}

\DeclareMathOperator{\Pic}{Pic}
\DeclareMathOperator{\End}{End}

\DeclareMathOperator{\Hom}{Hom}

\DeclareMathOperator{\Gal}{Gal}
\DeclareMathOperator{\GL}{GL}

\DeclareMathOperator{\M}{M}
\DeclareMathOperator{\CH}{CH}

\DeclareMathOperator{\id}{id}

\DeclareMathOperator{\Ta}{Ta}

\DeclareMathOperator{\et}{\text{\'et}}

\DeclareMathOperator{\TSym}{TSym}
\DeclareMathOperator{\Sym}{Sym}
\DeclareMathOperator{\mom}{mom}
\DeclareMathOperator{\Res}{Res}
\DeclareMathOperator{\Corr}{Corr}
\DeclareMathOperator{\CHM}{CHM}
\DeclareMathOperator{\Rep}{Rep}
\DeclareMathOperator{\VHS}{VHS}
\DeclareMathOperator{\Anc}{Anc}

\newcommand{\cores}{\mathrm{cores}}

\newcommand{\ord}{\mathrm{ord}}

\newcommand{\norm}{\mathrm{norm}}

\newcommand{\cont}{\mathrm{cont}}
\newcommand{\mot}{\mathrm{mot}}

\newcommand{\Sh}{\mathrm{Sh}}

\newcommand{\rec}{\mathrm{rec}}

\newcommand{\boldG}{\boldsymbol{\mathrm{G}}}
\newcommand{\boldH}{\boldsymbol{\mathrm{H}}}

\newcommand{\Et}{\mathrm{\acute{E}t}}
\renewcommand{\et}{\mathrm{\acute{e}t}}

\newcommand{\longmono}{\mbox{\;$\lhook\joinrel\longrightarrow$\;}}

\newcommand{\longepi}{\mbox{\;$\relbar\joinrel\twoheadrightarrow$\;}}

\newcommand{\smallmat}[4]{\bigl(\begin{smallmatrix}#1&#2\\#3&#4\end{smallmatrix}\bigr)}

\begin{document}

\title{Interpolation of generalized Heegner classes along quaternionic Coleman families}
\date{}
\author{E. Rocha Walchek}
\thanks{}

\begin{abstract}
We construct big generalized Heegner classes by interpolating $p$-adically the generalized Heegner classes associated to quaternionic modular forms along a Coleman (finite slope) family, following the approach introduced by Jetchev--Loeffler--Zerbes.
\end{abstract}

\address{Facultad de Ciencias, Universidad de la República, Iguá 4225, 11400 Montevideo, Uruguay}
\email{erocha@cmat.edu.uy}

\subjclass[2020]{11F11 (primary), 14G35 (secondary)}

\keywords{quaternionic multiplication abelian varieties, Shimura curves, quaternionic Coleman families, generalized Heegner classes, big generalized Heegner classes}
\maketitle

\section{Introduction}\label{sec.Introduction}

Let $K$ be a number field. The \textit{Bloch--Kato conjecture} for $K$-representations $V$ coming from geometry states that the Bloch--Kato Selmer group of $V$ dictates the behavior of the $L$-function attached to $V$ at the special point $s=0$, minus slight correction: 
\begin{equation}\label{eq.BKC}
\ord_{s=0}L(s,V)=\dim H^1_f(\Gal(\overline{K}/K),V^*(1))-\dim H^0(\Gal(\overline{K}/K),V^*(1)).
\end{equation}
If $V$ is the $p$-adic representation attached to an elliptic curve $E$ over $\Q$ with finite Shafarevich--Tate group, \eqref{eq.BKC} becomes the \textit{Birch and Swinnerton-Dyer (BSD) conjecture} which is known to hold when $\ord_{s=1}L(s,E)\in\{0,1\}$, due to Kolyvagin's work on Euler systems \cite{Kolyvagin} -- the use of norm-compatible system of classes arising from a tower of \textit{Heegner points} -- combined with the Gross--Zagier formula \cite{GrossZagier} -- relating the height of said Heegner points to the derivative of the $L$-function -- which is the key to bound the size of the Selmer group.

Since then, Euler systems of Abel--Jacobi images of cycles generalizing Heegner points have become the most promising tool for studying the Bloch--Kato conjecture for representations attached to modular forms. Neková\v{r} \cite{Nekovar.Kolyvagin, Nekovar.HeegnerCycles} introduced \textit{Heegner cycles} for higher weight modular forms, defined over $\mathcal{E}^k$, the $k$-th fiber product power of the universal elliptic curve $\mathcal{E}$ over the modular curve $Y_0(N)$. Bertolini--Darmon--Prasanna \cite{BertoliniDarmonPrasanna} later introduced \textit{generalized Heegner cycles} by adding a ``fixed'' elliptic curve $E$ with CM over an imaginary quadratic field factor to the Kuga--Sato variety, now consisting of the canonical desingularization of $\mathcal{E}^k\times_{Y_1(N)} E^k$, from which twists by characters arise. These cycles can be directly related to certain values of $p$-adic analogues of complex $L$-functions, called BDP $p$-adic $L$-functions; notably, these values are obtained by $p$-adic limit processes at points outside the range of $p$-adic interpolation of complex $L$-values. The relation between Abel--Jacobi images of generalized Heegner cycles (classes) and values of the BDP $p$-adic $L$-function outside the range of interpolation is known as \textit{explicit reciprocity law}, and plays a fundamental role in all applications of generalized Heegner cycles to the Bloch--Kato conjecture.

By interpolating classes along a $p$-adic family of modular forms,  geometric constructions that are available only for infinitely many classical forms in a $p$-adic family can be used to generate, via $p$-adic approximation, algebraic objects at points where at first no such constructions are available, a major advantage over the single modular form approach. The zero slope (Hida family) case is due to Howard \cite{Howard}: after establishing a relation between the generalized Heegner classes and Heegner points in a certain tower of modular curves, a big Heegner point, playing the role of $p$-adic family of Heegner points, allows for the $p$-adic interpolation along a Hida family of classes associated to elliptic modular forms. Büyükboduk--Lei \cite{BuyukbodukLei} and Jetchev--Loeffler--Zerbes \cite{JetchevLoefflerZerbes} have independently defined \textit{big generalized Heegner classes} interpolating the generalized Heegner classes of \cite{BertoliniDarmonPrasanna} as they vary over a finite slope (Coleman) family of elliptic modular forms. After $p$-adically interpolating the special values of the complex $L$-functions of each modular form in the family into a \textit{$p$-adic $L$-function}, the big generalized Heegner classes are found to be related to the $p$-adic $L$-function via an \textit{explicit reciprocity law} extending to families the main result of \cite{BertoliniDarmonPrasanna}.

The present document concerns the definition of the \textit{big generalized Heegner classes} in the indefinite quaternionic setting, adapting the approach of \cite{JetchevLoefflerZerbes}. Future work in preparation will address the interpolation of $p$-adic $L$-functions along Coleman families of quaternionic modular forms and the correspondent explicit reciprocity law, an expected extension to families of the main result of \cite{Brooks}. What follows is a brief account of the contents of this article.

\subsection{The indefinite quaternionic setting} In the \textit{elliptic} setting above, generalized Heegner cycles attached to a $p$-stabilized eigenform of level $pN$, with $p\nmid N$, are defined over (ring class fields of) an imaginary quadratic field $K$ of discriminant $-D_K$ in which all prime factors of $N$ are split. However, this restrictive hypothesis leave out many fields, which can be covered by weakening the hypothesis on $N$, allowing $N=N^+N^-$ to have both split and inert primes (the factors of $N^+$ and $N^-$ respectively). This replaces the modular curves in which the Heegner points were defined with \textit{Shimura curves} over the quaternionic algebra $B$ of discriminant $N^->1$ classifying \textit{QM abelian surfaces}. Since the modular forms over $B$ and $\GL_2(\Q)$ are essentialy the same via the Jacquet--Langlands lift, one needs to ``halve'' the dimension of every object coming from such abelian surfaces so they behave like objects coming from elliptic curves; this is done by introducing an action by a suitable global idempotent $e\in B$. Furthermore, it is imperative for the quaternion algebra $B$ to split at $\infty$ as $\GL_2(\Q)$ does, in which case we say that $B$ is \textit{indefinite}.

\subsection{Generalized Heegner cycles} The definitions and results of \cite{BertoliniDarmonPrasanna} find their quaternionic counterpart in \cite{Brooks}. For each $m\ge 0$, there is a Shimura curve $\widetilde{X}_m$ over $B$ playing the role of ``quaternionic $Y_1(Np^m)$''. Let $\pi_{\mathcal{A}}\colon\widetilde{\mathcal{A}}_m\to \widetilde{X}_m$ be the universal abelian surface over $\widetilde{X}_m$, $\mathcal{F}$ be a quaternionic eigenform over $\widetilde{X}_m$ of tame level $N^+$ and weight $k=2r+2\ge 2$, $A\defeq (\C/\mathcal{O}_K)^{\oplus 2}$ an abelian surface with an action by a maximal order of $B$ (quaternionic multiplication), $\mathcal{O}_{cp^n}$ be the order in $K$ of conductor $cp^n$ for some $c\ge 0$ prime to $pND_K$ and consider the natural isogeny $\phi\colon A\to A_{cp^n}\defeq (\C/\mathcal{O}_{cp^n})^{\oplus 2}$. Then \[\operatorname{graph}(\phi)^r\in\CH^{k-1}(W_{k,m}\otimes_{H}F_{cp^n})_{\Q}\] defines a cycle over the \textit{generalized Kuga--Sato variety} $W_{k,m}\defeq \widetilde{\mathcal{A}}_m^r\times_{\widetilde{X}_m}A^r$ over a suitably large finite extension $F_{cp^n}$ of $H_{cp^n}$, the ring class field of $K$ of conductor $cp^n$. Averaging over all permutations of the fiber product via a projector $\epsilon_W$ defines the \textit{generalized Heegner cycle} \[\Delta_{cp^n,m}^{[k]}\defeq \epsilon_W\operatorname{graph}(\phi)^r\in\epsilon_W\CH^{k-1}(W_{k,m}\otimes_{H}F_{cp^n})_{\Q},\] which can be mapped into a class in $\epsilon_{W}H_{\et}^{2k-2}(W_{k,m}\otimes_HF_{cp^n},\Q_p(k-1))$ via the étale realization map. Lieberman's trick allows us to work over the simpler Shimura curve $\widetilde{X}_m$ instead of $W_{k,m}$, for the price of a slightly more complicated coefficient sheaf:
\begin{equation*}
\epsilon_W H_{\et}^{2k-2}(W_{k,m},\Q_p(k-1))\longrightarrow H_{\et}^2\left(\widetilde{X}_m,\TSym^{2r}(eR^1\pi_{\mathcal{A},*}\Q_{p})(2r+1)\otimes \TSym^{2r}(eR^1\pi_{A,*}\Q_{p})\right),
\end{equation*}
where $e$ is a ``dimension halver'' idempotent which splits $R^1\pi_{\mathcal{A},*}\Q_{p}$ into two isomorphic factors, each of the size of the correspondent object arising from an elliptic curve. The second factor is a sum of $2r$ characters, while the first factor, after projecting to the $j$-th among the $2r$ characters and then applying the Hochschild--Lyndon--Serre isomorphism, can be projected onto the dual $p$-adic Galois representation $V_{\mathcal{F}}^*$, all maps composing into the $p$-adic Abel--Jacobi map, under which $\Delta_{cp^n,m}^{[k]}$ becomes a \textit{generalized Heegner class} \[z_{cp^n,m}^{[\mathcal{F},j]}\in H^1\left(F_{cp^n},H^1_{\et}(\widetilde{X}_m\otimes_\Q\overline{\Q},V_{\mathcal{F}}^*\otimes\sigma_{\et}^{2r-j}\bar{\sigma}_{\et}^j)\right),\] where $\sigma_{\et}^{2r-j}\bar{\sigma}_{\et}^j$ is a character arising from a fixed embedding $\sigma\colon K\hookrightarrow L$ into a suitable $p$-adic field over which families of modular forms will be defined. If $\chi$ is a Hecke character that restricts to $\sigma_{\et}^{2r-j}\bar{\sigma}_{\et}^j$ in $\Gal(F_{cp^n}/H_{cp^n})$, we can then see $z_{cp^n,m}^{[\mathcal{F},j]}$ as a class in $H^1(H_{cp^n}, V_\mathcal{F}^*\otimes\chi)$.

\subsection{Interpolation into big Heegner classes} Let $\mathcal{F}$ now be defined over the Shimura curve $\widetilde{X}_1$ and $p$-stabilized with respect to a root $\upalpha$ of its Hecke polynomial; let $\mathscr{F}$ be a Coleman family of slope $v_p(\upalpha)<k-1$ passing through $\mathcal{F}$ and defined over a suitably small affinoid $\mathscr{U}$ of the weight space $\mathscr{W}$. To $p$-adically interpolate along $\mathscr{F}$, we follow the approach of \cite{JetchevLoefflerZerbes} and transfer this problem to the $p$-adic interpolation of basis vectors of the representation associated, under Ancona's functor \cite{Ancona}, to the motive whose étale realization is $\TSym^{2r}(eR^1\pi_{\mathcal{A},*}\Q_{p})(2r+1)\otimes\sigma_{\et}^{2r-j}\bar{\sigma}_{\et}^j$ (we take all étale sheaves to be defined over a suitable $p$-adic field $L$ over which $\mathscr{F}$ is defined). Such basis vectors are, in turn, defined after \textit{Heegner points} in Shimura curves.

A Heegner point $\vartheta_{cp^n}\in \widetilde{X}_1$ associated to the optimal embedding $\iota_K^B\colon K\hookrightarrow B$ of $\mathcal{O}_{cp^n}$ into the standard Eichler order of level $p$ gives a \textit{basis vector} \[e_{cp^n}^{[k,j]}\in\TSym^{2r}((L^2)^\vee)\otimes \sigma_{\et}^{2r-j}\bar\sigma_{\et}^j,\] which can be seen as a section over $S_{cp^n}$, the zero dimensional Shimura variety over $K$ classifying elliptic curves defined over $F_{cp^n}$ with CM by $\mathcal{O}_K$. Under the \textit{Gysin (``wrong way'') map} induced by the embedding $S_{cp^n}\stackrel{\iota_K^B\times\id}{\hookrightarrow} \widetilde{X}_1\times S_{cp^n}$, $e_{cp^n}^{[k,j]}$ maps into $z_{cp^n,1}^{[\mathcal{F},j]}$. One then interpolate basis vectors over $\mathscr{U}$ into \textit{big basis vectors} $\mathbf{e}^{[\mathscr{U},j]}_{cp^n}$, using the overconvergent projector of Loeffler--Zerbes \cite{LoefflerZerbes.Rankin}. Under the interpolated Gysin map, the big basis vector map into a big class $\mathbf{z}^{[\mathscr{U}, j]}_{cp^n,m}$, not quite the big generalized Heegner class over $\mathscr{U}$, but its $j$-component. The classes obtained varying $n$ and $m$ already enjoy Euler system properties:
\begin{equation}
\norm_{F_{cp^{n+1}}/F_{cp^{n}}}\left(\mathbf{z}^{[\mathscr{U},j]}_{cp^{n+1},1}\right)=\mathbf{a}_p\cdot\mathbf{z}^{[\mathscr{U},j]}_{cp^n,1}.
\end{equation}
The interpolated classes above lie in cohomologies with coefficients in a distributions module $D_\mathscr{U}$ that interpolates the single form case coefficient sheaves $\TSym^{2r}(eR^1\pi_{\mathcal{A},\ast}\Q_p(1))$ in $2r$ over $\mathscr{U}$. As in the single form case, $D_\mathscr{U}$ projects onto a \textit{big dual Galois representation} $\mathbf{V}_{\mathscr{F}}^*$ which specializes at weight $k$ to $V_{\mathcal{F}_k}^*$, where $\mathcal{F}_k$ denotes the specialization of $\mathscr{F}$ at weight $k\in\Z\cap\mathscr{U}$. This projection defines a class $\mathbf{z}^{[\mathscr{F},j]}_{cp^n,1}$, which specializes at weight $k\ge j$ to $\binom{2r}{j}^{-1}z^{[\mathcal{F}_k,j]}_{cp^n,1}$, the binomial factor needed to ``control'' the slope of $\mathscr{F}$. Finally, interpolate $\sigma_{\et}^{-j}\bar{\sigma}_{\et}^j$ in $j$ to $\Lambda(\bar{\sigma}_{\et}/\sigma_{\et})$, where $\Lambda\defeq\Z_p[\![\operatorname{Gal}(\bigcup_{i=1}^n F_{cp^i}/F_{cp})]\!]$, giving the \textit{big generalized Heegner class} associated to $\mathscr{F}$, denoted $\mathbf{z}^{[\mathscr{F},\mathbf{j}]}_{cp}$ and defined over $F_{cp}$.

\subsection{Main result}
As our main result, we describe the specializations of the generalized big Heegner class at classical weights:

\begin{theorem*}[Theorem \ref{teo.Specialization}]
Let $\mathcal{F}$ be a $p$-stabilized quaternionic modular form of signature $(k,\psi)$ and $\mathscr{F}$ be a Coleman family passing through $\mathcal{F}$, defined over an affinoid $\mathscr{U}$ of the weight space $\mathscr{W}$, and with coefficients in a $p$-adic field $L\hookleftarrow K$. The specialization of $\mathbf{z}_{cp}^{[\mathscr{F},\mathbf{j}]}\in H^1(H_{cp},(\mathbf{V}_\mathscr{F}^*\otimes\kappa^\mathrm{univ}_\mathscr{U})\widehat{\otimes}\Lambda(\bar{\sigma}/\sigma))$ at $0\le j\le 2r$, weight $k=2r+2\in\Z\cap\mathscr{U}$ and Hecke character of $\chi$ of infinity type $(2r-j,j)$ is the class \[a_p(\mathcal{F}_k)^{-n}\binom{2r}{j}^{-1}z_{cp^n,1}^{[\mathcal{F}_{k},\chi]}\in H^1(H_{cp^n},V_{\mathscr{F}}^\ast\otimes\chi),\] where $a_p(\mathcal{F}_k)$ denotes the eigenvalue of $\mathcal{F}_k$ with respect to the Hecke operator $U_p$ at $p$.
\end{theorem*}

This result extends \cite[Theorem 5.4.1]{JetchevLoefflerZerbes} to the quaternionic setting, and the proof is an adaptation of the main techniques employed in \textit{op. cit.}, consisting of checking the conditions -- most notably the Euler system relations (\textit{cf.} Proposition \ref{prop.EulerRelationsBig}) and the boundness of the growth of the big classes (\textit{cf.} Lemma \ref{lem.vartheta} and the proof of Proposition \ref{prop.BigHeegnerClass}) -- under which a general result of Loeffler--Zerbes \cite[Proposition 2.3.3]{LoefflerZerbes.Rankin} apply.

\subsubsection*{Acknowledgements}
I would like to profoundly thank my mentor Matteo Longo for his patient guidance and constant support during my Ph.D. candidacy at Padova, which were essential to the completion of my thesis and of this article. I also thank Paola Magrone for helpful discussions of her work, and for indirect discussions through the updates of a previous joint project with her and Matteo Longo. I extend my gratitude to the kind referee for their valuable comments.

The content of this article was developed as part of my Ph.D. thesis at Università degli Studi di Padova, and funded by the doctoral scholarship. 

\section{Notation}\label{sec.Notation}

Fix algebraic closures $\overline{\Q}$ for $\Q$ with an embedding $\iota_\infty\colon\overline{\Q}\hookrightarrow\C$, and, for any rational prime $\ell$, $\overline{\Q}_\ell$ for $\Q_\ell$ with an embedding $\iota_\ell\colon\overline{\Q}\hookrightarrow\overline{\Q}_\ell$. For any prime $\ell$, $v_\ell$ denotes the $\ell$-adic valuation normalized so that $v_\ell(\ell)=1$.

\subsection{Generalized Heegner hypothesis}\label{subsec.GeneralizedHeegnerHypothesis}
Fix $p$ a rational prime and $f_0^\sharp$ a new at $p$ elliptic eigenform of even weight $k_0\ge 2$ and level $\Gamma_1(N)$, with $N$ coprime with $p$. Next, we consider a quadratic imaginary field $K$, of discriminant $-D_K<0$ and ring of integers $\mathcal{O}_K$, satisfying the \textit{generalized Heegner hypothesis} with respect to $N$:
\begin{enumerate}[label=\textnormal{(Heg\arabic*)}]
\item\label{cond.Heg1} A prime $\ell\mid N$ either splits or remains inert in $K$;

\item\label{cond.Heg2} Being $N^{+}$ the product of all primes $\ell\mid N$ that split in $K$, $N^{+}\ge 4$;

\item\label{cond.Heg3} Being $N^{-}$ the product of all primes $\ell\mid N$ inert in $K$, $N^{-}$ is square-free and has an even and non-zero number of prime divisors.
\end{enumerate}
From \ref{cond.Heg1}, we get a decomposition $N=N^{+}N^{-}$. For the prime $p$, write $\mathfrak{p}$ for the prime ideal in $K$ corresponding to the embedding $\iota_p$, so that in $K$ we have $(p)=\mathfrak{p}\bar{\mathfrak{p}}$ if $p$ splits, $(p)=\mathfrak{p}$ if $p$ is inert, and $(p)=\mathfrak{p}^2$ if $p$ is ramified.

\subsection{Class field theory}\label{subsec.CFT}
Denote by $\widehat{\Z}\defeq\varprojlim_{n\in\Z_{\ge 1}}\Z/n\Z$ the profinite integers, and for any $\Z$-module $M$, write $\widehat{M}\defeq M\otimes_{\Z}\widehat{\Z}$. For any number field $F$, $\A_F$ denotes its ring of adèles, being $\A_{F,\mathrm{fin}}$ and $\A_{F,\infty}$ the finite and infinite adèles, respectively.

\begin{remark}\label{rem.Reciprocity}
The Artin reciprocity map \[\operatorname{rec}_F\colon F^\times\backslash\A_F^\times\longrightarrow\Gal(F^{\mathrm{ab}}/F),\] where $F^\mathrm{ab}$ denotes the maximal abelian extension of $F$, will be always taken to be geometrically normalized, that is, uniformizers map to \textit{geometric} Frobenii instead of the arithmetic Frobenii, by composing the original map with the inversion map $\sigma\mapsto\sigma^{-1}$.
\end{remark}

For $c\in\Z_{\ge 1}$ coprime with $pND_K$ and $n\in\Z_{\ge 0}$, we denote by $\mathcal{O}_{cp^n}\defeq\Z +cp^n\mathcal{O}_K$ the order of $\mathcal{O}_K$ of conductor $cp^n$. We write $H_{cp^n}$ for the ring class field of $K$ of conductor $cp^n$; put $H\defeq H_{1}$ and $H_{cp^\infty}\defeq\bigcup_{n\ge 1}H_{cp^n}$.
Since $\gcd(c,p)=1$, we have $H_c\cap H_{p^\infty}=H$, and so \[\Gal(H_{cp^\infty}/K)\cong\Gal(H_c/K)\times\Gal(H_{p^\infty}/K).\] It follows from the Artin reciprocity law that $\Gal(H_{cp^n}/K)\cong\Pic(\mathcal{O}_{cp^n}).$

Since all primes dividing $N^+$ are split in $K$, there is a decomposition $N^+=\mathfrak{N}^+\bar{\mathfrak{N}}^+$ into ideals of $\mathcal{O}_K$ by which quotients are cyclic of order $N^+$. Let $F_{\mathfrak{N}^+}$ be the ray class field of $K$ modulo $\mathfrak{N}^+$ and $F_{\mathfrak{N}^+,cp^n}$ be the smallest abelian extension of $K$ containing both $F_{\mathfrak{N}^+}$ and $H_{cp^n}$. Under the Artin reciprocity map, 
$F_{\mathfrak{N}^+,cp^n}$ corresponds to the group 
\[U_{\mathfrak{N}^+,cp^n}\defeq\{x\in \widehat{\mathcal{O}}_{cp^n}^\times:x\equiv 1\mod \mathfrak{N}^{+}\}.\] We do not carry the dependence to $\mathfrak{N}^+$ in the notation and just write $F_{cp^n}\defeq F_{\mathfrak{N}^+,cp^n}$.

\subsection{Quaternion algebras}\label{subsec.QuaternionAlgebras}
Condition \ref{cond.Heg3} implies the existence of an indefinite\footnote{In this article, whenever we refer to the \textit{quaternionic} setting as opposed to the elliptic ($\mathrm{GL}_2$) setting, it is always understood that we are talking about the \textit{indefinite quaternionic} setting.}. rational quaternion division algebra of discriminant $N^{-}$. Fix $B$ one such algebra, denote by $\bar{\smallbullet}$ its main involution, and fix in $B$ a maximal order $\mathcal{O}_B$. Define $\delta\defeq\sqrt{-D_K}$ and
\begin{equation}\label{eq.vartheta}
\vartheta\defeq\dfrac{\delta+D'}{2},\hspace{10pt}\text{where }D'=\begin{cases} D_K, & \text{if }2\nmid D_K\\ D_K/2, & \text{if }2\mid D_K.\end{cases}
\end{equation}

As a $\Q$-algebra, $K$ embeds into $\mathrm{M}_2(\Q)$ by a $\Q$-linear map $\iota_K^{\M_2(\Q)}\colon K\to\mathrm{M}_2(\Q)$ defined in the $\Q$-base $\{1,\vartheta\}$ of $K$ by setting
\begin{equation}\label{eq.ImageVartheta}
\vartheta\longmapsto\begin{pmatrix}\operatorname{trace}_{K/\Q}(\vartheta) & -\operatorname{norm}_{K/\Q}(\vartheta) \\ 1 & 0\end{pmatrix},
\end{equation}
and we denote the composition of $\iota_K^{\M_2(\Q)}$ with the canonical inclusion $\iota_{\M_2(\Q)}^{\M_2(K)}\colon\mathrm{M}_2(\Q)\hookrightarrow\mathrm{M}_2(K)$ by $\iota_K^{\M_2(K)}$. Since all prime factors of $N^{-}$ are non-split in $K$, $B$ \textit{splits} over $K$, that is, there is an embedding $\iota_K^B\colon K\to B$, or equivalently, an isomorphism $I_B\colon B\otimes_{\Q} K\stackrel{\sim}{\longrightarrow}\mathrm{M}_2(K)$, and composing the natural inclusion $B\hookrightarrow B\otimes_{\Q} K$ with $I_B$ gives a map $\iota_B^{\M_2(K)}\colon B\hookrightarrow\mathrm{M}_2(K)$. We can choose $I_B$ to be such that $\iota_K^{\M_2(K)}=\iota_B^{\M_2(K)}\circ\iota_K^B$: all options for $I_B$ are conjugated to each other, so one just have to match the image of $\iota_K^B(\vartheta)$ under $\iota_B^{\M_2(K)}$ with $\iota_K^{\M_2(K)}(\vartheta)$.

For any finite place $\ell\nmid N^{-}$, where $B$ splits, we fix $i_\ell\colon B_\ell\defeq B\otimes_{\Q}\Q_\ell\stackrel{\sim}{\longrightarrow}\mathrm{M}_2(\Q_\ell)$ such that $i_\ell(\mathcal{O}_K\otimes_{\Z}\Z_\ell)\subseteq\mathrm{M}_2(\Z_\ell)$. In particular, $i_\ell$ is an isomorphism between $\mathcal{O}_{B,\ell}\defeq\mathcal{O}_{B}\otimes_{\Z}\Z_\ell$ and $\mathrm{M}_2(\Z_\ell)$, which reduces modulo $\ell^m$ for all $m\ge 1$ to an isomorphism \[\bar{\imath}_{\ell^m}\colon\mathcal{O}_B\otimes_{\Z}(\Z/\ell^m\Z)\stackrel{\sim}{\longrightarrow}\mathrm{M}_2(\Z/\ell^m\Z).\]
If $M$ is an integer coprime with $N^{-}$, the Chinese Remainder Theorem gives an isomorphism \[\bar{\imath}_M\colon \mathcal{O}_B\otimes_{\Z}(\Z/M\Z)\stackrel{\sim}{\longrightarrow}\mathrm{M}_2(\Z/M\Z).\]
Since $B$ is indefinite, $B$ splits at $\infty$, so there is an isomorphism $i_\infty\colon B_\infty\defeq B\otimes_{\Q}\R\to\M_2(\R)$. One can normalize all the isomorphisms $i_v$ defined at all places $v$ splitting $B$ to satisfy \eqref{eq.ImageVartheta}.

\subsection{QM abelian surfaces}\label{subsec.QMAbelianSurfaces}
Let $M$ be a positive integer coprime with $N^{-}$ and $S$ be a $\Z[1/N^{-}]$-scheme. An \textit{abelian surface with quaternionic multiplication} by $\mathcal{O}_B$ (or \textit{QM abelian surface}, for short) is a pair $(A,\iota)$ consisting of an abelian scheme $A\to S$ of relative dimension 2 and an injective algebra morphism $\iota\colon\mathcal{O}_B\hookrightarrow\End_S(A)$. An \emph{isogeny} (resp. an \emph{isomorphism}) of QM abelian surfaces is an isogeny (resp. an isomorphism) of abelian schemes commuting with the $\mathcal{O}_B$-action.

Since $K$ splits $B$, one can take $t\in\mathcal{O}_B$ such that $t^2=-D_K<0$ and define the involution
\begin{equation}\label{eq.Involution}
\smallbullet^\dagger\colon B\longrightarrow B,\hspace{5pt} b^\dagger\defeq t^{-1}\bar{b}t
\end{equation}
Then there is a unique principal polarization $\lambda\colon A\to A^\vee$ such that, for a geometric point $s$ of $S$, the restriction of the Rosati involution of $\End(A_s)$ to $\mathcal{O}_B$ coincides with $\dagger$ (see \cite[\S III.1.5]{BoutotCarayol}).

\subsection{Shimura curves}\label{subsec.ShimuraCurves}
The Shimura datum $(\boldG,\mathcal{H}^\pm)$ for $\boldG$ the algebraic group defined by
\begin{equation}\label{eq.boldG}
\boldG(F)\defeq (B\otimes_{\Q} F)^\times,\text{ for any extension }F/\Q,
\end{equation}
and $\mathcal{H}^\pm=\mathcal{H}^+\cup\mathcal{H}^-$, where $\mathcal{H}^\bullet\defeq\{\tau\in\C:\ \operatorname{sign}(\operatorname{Im}(\tau))=\bullet\}$ for $\bullet\in\{-,+\}$, is of PEL type.

For each integer $m\geq 0$, let \[R_m=\{b\in\mathcal{O}_B:\ i_\ell(b)\text{ is upper triangular mod }\ell^{v_{\ell}(N^+p^m)},\ \forall\ell\mid N^+p^m\text{ prime}\}\] be the \textit{Eichler order} of $B$ of level $N^+p^m$ with respect to the chosen isomorphisms $i_\ell$ for all finite places $\ell\nmid N^-$. Let $U_m\defeq\widehat{R}^\times_{m}=(R_{m}\otimes_\Z\widehat\Z)^\times$ and let $\widetilde U_m$ be the subgroup of $U_m$ consisting of the elements whose $p$-component  is congruent to a matrix of the form $\smallmat 1b0d\mod{p^m}$. For the same Shimura datum $(\boldG,\mathcal{H}^\pm)$ and varying levels $U_m$ and $\widetilde{U}_m$, we get two families of Shimura varieties respectively:
\begin{equation}
X_m(\C)\defeq B^\times\backslash(\mathcal{H}^\pm\times\widehat{B}^\times)/U_m\hspace{10pt}\text{and}\hspace{10pt}\widetilde X_m(\C)\defeq B^\times\backslash(\mathcal{H}^\pm\times\widehat{B}^\times)/\widetilde U_m,
\end{equation}
where $B^\times$ acts by Möbius transformations on $\mathcal{H}^\pm$ under $i_\infty$ and by left multiplication on $\widehat{B}^\times$, and $U_m$ (resp. $\widetilde{U}_m$) acts trivially on $\mathcal{H}^{\pm}$ and by right multiplication on $\widehat{B}^\times$. The double coset represented by a point $(x,g)\in\mathcal{H}^\pm\times\widehat{B}^\times$ is written $[(x,g)]$. The space $\mathcal{H}^{\pm}\cong\GL_2(\R)\cdot\vartheta$ can be identified with $\PP\defeq\Hom_{\R}(\C,B_\infty)$, over which $B^\times$ acts by conjugation, as follows: since $\iota_K^B\in\Hom_{\Q}(K,B)\otimes_{\Q}\R\subseteq\PP$, a point $x=\gamma_x\cdot\vartheta$ is sent to $\iota_x\defeq\gamma_x^{-1}\iota_K^B \gamma_x\in\PP$. 

For any $m\ge 0$, both double coset spaces $B^\times\backslash\widehat{B}^\times/U_m$ and $B^\times\backslash\widehat{B}^\times/\widetilde{U}_m$ are singletons, so the Strong Approximation Theorem implies
\begin{equation}\label{eq.RiemannSurfaces}
X_m(\C)\cong\Gamma_m\backslash\mathcal{H}^{+}\hspace{10pt}\text{and}\hspace{10pt}\widetilde{X}_m(\C)\cong\widetilde{\Gamma}_m\backslash\mathcal{H}^{+},
\end{equation}
where $\Gamma_m$ (resp. $\widetilde\Gamma_m$) is the subgroup of norm $1$ elements in $B^\times\cap U_m$ (resp. $B^\times\cap\widetilde U_m$). This means that both $X_m(\C)$ and $\widetilde{X}_m(\C)$ are Riemann surfaces, thus being algebraic curves defined over $\C$. Furthermore, there are algebraic curves $X_m$ and $\widetilde{X}_m$ defined over $\Q$ whose $\C$-rational points match $X_m(\C)$ and $\widetilde{X}_m(\C)$ respectively, tying the notation together.

If $\tau\in K\cap\mathcal{H}^{\pm}$ is a CM point, then $\iota_\tau\in\Hom_{\Q}(K,B)$ fixes $\tau$, so that $i_\infty(\iota_\tau(z))\binom{\tau}{1}=z\binom{\tau}{1}$ for all $z\in K$. Explicitly, a CM class $[(\iota_\tau,g)]$ maps under strong approximation to $[b^{-1}\cdot\tau]$, where $g=bu$ with $b\in B^\times$ and $u\in U_m$ (resp. $\widetilde{U}_m$).

\subsection{Heegner points}\label{subsec.HeegnerPoints}
A \textit{Heegner point} of conductor $cp^n$ in $X_m$ is a class $[(\iota_x,g)]\in X_m$ where $g\in\widehat{B}^\times$ and $\iota_x$ is an optimal embedding of $\mathcal{O}_{cp^n}$ into the Eichler order $g^{-1}U_mg\cap B$, that is, $\iota_x^{-1}(g^{-1}U_mg\cap B)=\mathcal{O}_{cp^n}$. The finite index inclusion $U_m\supseteq\widetilde{U}_m$ induces a finite covering $\varsigma_m\colon\widetilde{X}_m(\C)\twoheadrightarrow X_m(\C)$. A \textit{Heegner point} of conductor $cp^n$ in $\widetilde{X}_m$ is a class $[(\iota_x,g)]\in\widetilde{X}_m$ such that $\varsigma_m([(\iota_x,g)])$ is a Heegner point of conductor $cp^n$ in $X_m$ and an additional local condition at $p$ is satisfied (see \cite[\S3]{LongoVigni.Manuscripta} for details).

We construct explicit Heegner points of conductor $cp^n$ in $\widetilde{X}_m$ (for $m\leq n$, condition needed only when $p$ is not split) associated to certain CM points by slightly modifying \cite{ChidaHsieh.2015, CastellaHsieh, ChidaHsieh.2018, CastellaLongo}. 
For each rational prime $\ell$ we define:
\begin{itemize}
\item $\xi_{cp^n}^{(\ell)}\defeq1$ if $\ell\nmid N^+cp$;
\end{itemize}
For primes dividing $N^+c$, let $s\defeq v_\ell(c)$ and set:
\begin{itemize}
\item $\xi_{cp^n}^{(\ell)}\defeq\delta^{-1}\left(\begin{smallmatrix}\vartheta & \bar{\vartheta} \\ 1 & 1\end{smallmatrix}\right)\left(\begin{smallmatrix}\ell^s & 1 \\ 0 & 1\end{smallmatrix}\right)\in \GL_2(K_\mathfrak{l})=\GL_2(\Q_\ell)$ if $\ell\mid N^+c$ splits as $\ell=\mathfrak{l}\bar{\mathfrak{l}}$ in $K$;
\item $\xi_{cp^n}^{(\ell)}=\left(\begin{smallmatrix}0 & -1 \\ 1 & 0\end{smallmatrix}\right)\left(\begin{smallmatrix}\ell^s & 0 \\ 0 & 1\end{smallmatrix}\right)\in \GL_2(\Q_\ell)$ if $\ell\mid c$ does not split in $K$;
\end{itemize}
Finally, for $p$, we set:
\begin{itemize}
\item $\xi_{cp^n}^{(p)}=\left(\begin{smallmatrix}\vartheta & -1 \\ 1 & 0\end{smallmatrix}\right)\left(\begin{smallmatrix}p^n & 0 \\ 0 & 1\end{smallmatrix}\right)\in \GL_2(K_\mathfrak{p})=\GL_2(\Q_p)$ if $p=\mathfrak{p}\bar{\mathfrak{p}}$ splits in $K$;
\item $\xi_{cp^n}^{(p)}=\left(\begin{smallmatrix}0 & -1 \\ 1 & 0\end{smallmatrix}\right)\left(\begin{smallmatrix}p^n & 0 \\ 0 & 1\end{smallmatrix}\right)\in \GL_2(\Q_p)$ if $p$ does not split in $K$.
\end{itemize}

All those elements can be seen inside $\widehat{B}^\times$ under the isomorphisms $i_\ell^{-1}$ defined in \S\ref{subsec.QuaternionAlgebras}. Denote by $\xi_{cp^n}\in \widehat{B}^\times$ the element whose $\ell$-component is $\xi_{cp^n}^{(\ell)}$ for each prime $\ell$.

Define $x_{cp^n,m}\colon\Pic(\mathcal{O}_{cp^n})\rightarrow {X}_m(\C)$ by $[\mathfrak{a}]\mapsto[(\iota_K^B,a\xi_{cp^n})]$, where
for each class $[\mathfrak{a}]$ we take an element $\mathfrak{a}$ and write it as $\mathfrak{a}=a\widehat{\mathcal{O}}_{cp^n}\cap K$ for some $a\in \widehat{K}^\times$, seen as an element of $\widehat{B}^\times$ via the adelization of $\iota_K^B$. This definition is independent of the choice of $\mathfrak{a}$ and $a$, so one can replace $[\mathfrak{a}]$ in $x_{cp^n,m}([\mathfrak{a}])$ by just $\mathfrak{a}$ or $a$. More generally, the previous map can be lifted to
$\tilde{x}_{cp^n,m}\colon K^\times\backslash\widehat{K}^\times\rightarrow \widetilde{X}_m(\C)$ by setting $\tilde{x}_{cp^n,m}(a)\defeq[(\iota_K^B,a\xi_{cp^n})]$.

\begin{proposition}\label{prop.HeegnerPoint}
For any $[\mathfrak{a}]\in\Pic(\mathcal{O}_{cp^n})$, the point $x_{cp^n,m}([\mathfrak{a}])\in X_m(H_{cp^n})$ is a Heegner point of conductor $cp^n$.
\end{proposition}
\begin{proof}
We check for each prime $\ell$ if, writing $\xi\defeq \xi_{cp^n}^{(\ell)}$ and $\iota_\xi\defeq \xi^{-1}\iota_K^B\xi$, we have \[\iota_\xi(K_\ell)\cap (U_m\otimes\Z_\ell) = \iota_\xi(\mathcal{O}_{cp^n}\otimes \Z_\ell).\]
For $\ell\nmid N^+cp$, both orders $U_m\otimes\Z_\ell$ and $\mathcal{O}_{cp^n}\otimes \Z_\ell$ are maximal, thus the image of the latter is contained in the earlier. Next we consider $\ell\mid N^+cp$;

For a split prime $\ell=\mathfrak{l}\bar{\mathfrak{l}}\neq p$ (case $\ell=p$ is similar), where $\mathfrak{l}$ is the prime corresponding to the fixed embedding $\iota_\ell\colon\overline{\Q}\hookrightarrow\overline{\Q}_\ell$ from the beginning of the chapter, we have that $K_\ell\defeq K\otimes_{\Q}\Q_\ell$ splits as $K_\mathfrak{l}\oplus K_{\bar{\mathfrak{l}}}\cong\Q_\ell e_{\mathfrak{l}}\oplus\Q_\ell e_{\bar{\mathfrak{l}}}$, where \[ e_\mathfrak{l}=\frac{1\otimes\vartheta-\vartheta\otimes1}{(\vartheta-\bar{\vartheta})\otimes 1}\hspace{15pt}\text{and}\hspace{15pt}e_{\bar{\mathfrak{l}}}=\frac{\vartheta\otimes 1-1\otimes\bar{\vartheta}}{(\vartheta-\bar{\vartheta})\otimes1}.\] One then computes that for each $(a+b\vartheta)\otimes x\in K\otimes_\Q\Q_\ell$,
\[\iota_\xi((a+b\vartheta)\otimes x)=\begin{pmatrix}xa+xb\vartheta & \ell^{-s}xb\delta \\ 0 & xa+xb\bar{\vartheta}\end{pmatrix}=(xa+xb\vartheta)\iota_\xi(e_\mathfrak{l})+(xa+xb\bar{\vartheta})\iota_\xi(e_{\bar{\mathfrak{l}}}),\]
showing that $\iota_\xi(K_\ell)\cap (U_m\otimes\Z_\ell)$ consists of elements of the form \[(xa+xb\vartheta)e_\mathfrak{l}+(xa+xb\bar{\vartheta})e_{\bar{\mathfrak{l}}}\in\Z_\ell e_\mathfrak{l}\oplus\Z_\ell e_{\bar{\mathfrak{l}}}\] with $xb\in\ell^s\Z_\ell$, thus of elements in $\mathcal{O}_{cp^n}\otimes_\Z\Z_p$;

For a non-split prime $\ell$, one computes that for each $(a+b\vartheta)\otimes x\in K_\ell=K\otimes_\Q\Q_\ell$, \[\iota_\xi((a+b\vartheta)\otimes x)=\begin{pmatrix}xa & \ell^{-s}xb \\ xb\vartheta\bar{\vartheta}\ell^s & x(a+b(\vartheta+\bar{\vartheta}))\end{pmatrix},\] so elements $xa+xb\vartheta$ in $\iota_\xi(K_\ell)\cap (U_m\otimes\Z_\ell)$ are of the form $xa\in\Z_\ell$ and $xb\in\ell^s\Z_\ell$. This shows that, for all $s'\leq 2s$, $\iota_\xi$ is an optimal embedding into an Eichler order of level $\ell^{s'}$, showing both the cases inert ($s'=s$, when $\ell\mid c^{-}$ or $\ell=p$ inert) and ramified ($s'=2s$, when $\ell=p$ is ramified).\qedhere
\end{proof}
Direct verification of the local condition at $p$ shows that $\widetilde{x}_{cp^n,m}(a)$ is also defined over $H_{cp^n}$.

\subsection{Idempotents}\label{subsec.Idempotents}
The proof of Proposition \ref{prop.HeegnerPoint} motivates the choice of the following idempotents in $K\otimes_\Q K$:
\[e=\frac{1\otimes\vartheta-\vartheta\otimes1}{(\vartheta-\bar{\vartheta})\otimes 1}\hspace{15pt}\text{and}\hspace{15pt}\bar{e}=\frac{\vartheta\otimes 1-1\otimes\bar{\vartheta}}{(\vartheta-\bar{\vartheta})\otimes1}.\]  Via direct computation, one can show that those elements are orthogonal idempotents: \[e^2=e,\hspace{5pt} \bar{e}^2=\bar{e},\hspace{5pt} e\bar{e}=0\hspace{5pt}\text{and}\hspace{5pt}e+\bar{e}=1.\]
Locally at all split primes $\ell\mid N^+p$ (including $\ell=p$ in case $p$ splits), their local realizations at splitting places are the idempotent matrices $\left(\begin{smallmatrix}1 & 0 \\ 0 & 0\end{smallmatrix}\right)$ and $\left(\begin{smallmatrix}0 & 0 \\ 0 & 1\end{smallmatrix}\right)$: if we again consider the splitting $\ell=\mathfrak{l}\bar{\mathfrak{l}}$, we have that $K\otimes_{\Q} K$ embeds into $K\otimes_{\Q}K_\mathfrak{l}\cong K\otimes_{\Q}\Q_\ell$ thus obtaining a map \[j_\ell\colon K\otimes_{\Q}K\longmono K\otimes_{\Q}\Q_\ell\longmono B_\ell\stackrel{i_\ell}{\longrightarrow}\mathrm{M}_2(\Q_\ell)\] satisfying $j_\ell(e)=\big(\begin{smallmatrix}1 & 0 \\ 0 & 0\end{smallmatrix}\big)$ and $j_\ell(\bar{e})=\big(\begin{smallmatrix}0 & 0 \\ 0 & 1\end{smallmatrix}\big)$. We also have a global map
\[j\colon K\otimes_\Q K \longmono B\otimes_\Q K\stackrel{I_B}{\longrightarrow} \M_2(K)\] under which  $j(e)=\big(\begin{smallmatrix}1 & 0 \\ 0 & 0\end{smallmatrix}\big)$ and $j(\bar{e})=\big(\begin{smallmatrix}0 & 0 \\ 0 & 1\end{smallmatrix}\big)$.

If $p$ is inert or ramified, the images of $e$ and $\bar{e}$ are not yet $\left(\begin{smallmatrix}1 & 0 \\ 0 & 0\end{smallmatrix}\right)$ and $\left(\begin{smallmatrix}0 & 0 \\ 0 & 1\end{smallmatrix}\right)$ -- to achieve this we conjugate the image of the map above by $\mathcal{M}\defeq\left(\begin{smallmatrix}\bar{\vartheta} & \vartheta \\ 1 & 1\end{smallmatrix}\right)$ via an isomorphism $i_\mathcal{M}$, obtaining \[j_p\colon K\otimes_{\Q}K\longmono K\otimes_{\Q} K_\mathfrak{p}\longmono B_p\stackrel{i_p}{\longrightarrow}\mathrm{M}_2(K_\mathfrak{p})\stackrel{i_\mathcal{M}}{\longrightarrow}\mathrm{M}_2(K_\mathfrak{p})\]
again satisfying $j_p(e)=\big(\begin{smallmatrix}1 & 0 \\ 0 & 0\end{smallmatrix}\big)$ and $j_p(\bar{e})=\big(\begin{smallmatrix}0 & 0 \\ 0 & 1\end{smallmatrix}\big)$; notice that the notation $K_\mathfrak{p}$ correspond to different objects in each case. Likewise, we have a global map \[j\colon K\otimes_\Q K \longmono B\otimes_\Q K\stackrel{I_B}{\longrightarrow} \M_2(K)\stackrel{i_\mathcal{M}}{\longrightarrow} \M_2(K),\] and again we find that $j(e)=\big(\begin{smallmatrix}1 & 0 \\ 0 & 0\end{smallmatrix}\big)$ and $j(\bar{e})=\big(\begin{smallmatrix}0 & 0 \\ 0 & 1\end{smallmatrix}\big)$.
\begin{remark}
Some references, such as \cite{Brooks}, fix a \textit{Hashimoto model} for $B$, that is, the choice of a certain $\Q$-basis $\{1,s_1,s_2,s_1s_2\}$ and a totally real quadratic field $M$ defined after it which splits $B$. This allows the definition of a global idempotent $e\in K\otimes_{\Q} B$ that is fixed by the involution of $B$ given by $b\mapsto s_1^{-1}\bar{b} s_1$.
\end{remark}

\subsection{Moduli of QM abelian surfaces}\label{subsec.Moduli}
We recall from \cite[\S1--4]{Buzzard} some results on the representability of moduli problems of QM abelian surfaces. Let $M$ be a positive integer coprime with $N^-$, $S$ be a $\Z[1/N^-]$-scheme and $(A,\iota)$ a QM abelian surface over $S$.

\subsubsection{Naïve level structures}\label{subsec.NaiveLevelStructures}
A \textit{naïve full level $M$ structure} on $A$ is an isomorphism \[\alpha\colon\mathcal{O}_B\otimes_\Z(\Z/{M}\Z)\stackrel{\sim}{\longrightarrow} A[{M}]\] of $S$-group schemes which commutes with the left actions of $\mathcal{O}_B$ given by $\iota$ on $A[M]$, and the multiplication from the left of $\mathcal{O}_B$ on the constant $S$-group scheme $\mathcal{O}_B\otimes_\Z(\Z/{M}\Z)$. 

Denote by $r_g$ the right multiplication action of $g\in(\mathcal{O}_B\otimes(\Z/M\Z))^\times$ over $(\mathcal{O}_B\otimes(\Z/M\Z))^\times$. If $\alpha$ is a naïve full level $M$ structure on $A$, then so is $\alpha_g\defeq\alpha\circ r_g$, giving a \textit{left} action $g\cdot\alpha\defeq\alpha_g$ of $(\mathcal{O}_B\otimes_\Z(\Z/M\Z))^\times$ on the set of na\"{\i}ve full level $M$ structures on $(A,\iota)$. For any subgroup $U$ composing the action of $(\mathcal{O}_B\otimes_\Z(\Z/M\Z))^\times$ with the canonical projection \[\widehat{\pi}_M\colon\widehat{\mathcal{O}}_B^\times\longepi(\mathcal{O}_B\otimes_\Z(\Z/M\Z))^\times\] gives a left action of $U$ on the set of naïve full level $M$ structures. An equivalence class of the action of $U$ is a \emph{naïve level $U$ structure} $\alpha$ on $A$. We write $(A,\iota,\alpha)$ when equipping $(A,\iota)$ with a naïve level $U$ structure $\alpha$. An \textit{isogeny} (resp. \textit{isomorphism}) between two such triples $(A,\iota,\alpha)$ and $(A',\iota',\alpha')$ is an isogeny (resp. isomorphism) of QM abelian surfaces $\varphi\colon A\rightarrow A'$ such that $\varphi\circ\alpha=\alpha'$.

\begin{proposition}\label{prop.ModuliNaive}
The functor taking a $\Z[1/(MN^-)]$-scheme $S$ to the set of isomorphism classes of such triples $(A,\iota,\alpha)$ over $S$ is representable 
by a  projective, smooth, of relative dimension $1$ and geometrically connected $\Z[1/(MN^-)]$-scheme $\mathcal{X}_U$ (\cite[Theorem 2.1]{Buzzard}).
\end{proposition}

\begin{remark}\label{rem.Representability}
The representability result above is due to Morita \cite[Main Theorem 1]{Morita} in the case of na\"{\i}ve full level $M$ structures. The proof of the general case in \cite[\S2]{Buzzard} invokes the representability result of \cite[Theorem \S14, Exposé III]{Boutot}. Both the original and the general case require $N^{+}\ge 4$, justifying condition \ref{cond.Heg2} in \S\ref{subsec.GeneralizedHeegnerHypothesis}; see also \cite[\S4]{DiamondTaylor} and \cite[Theorem 2.2]{Brooks}. 
\end{remark}

\begin{remark}
We consider specially the naïve level $U$ structures given by
\[U_0(M)\defeq\{g\in\widehat{\mathcal{O}}_B^\times:\ \exists a,b,d\text{ such that } \bar{\imath}_M\circ\widehat{\pi}_M(g)\equiv\left(\begin{smallmatrix}a & b \\ 0 & d\end{smallmatrix}\right)\mod M\},\] and its subgroup \[U_1(M)\defeq\{g\in\widehat{\mathcal{O}}_B^\times:\ \exists b,d\text{ such that } \bar{\imath}_M\circ\widehat{\pi}_M(g)\equiv\left(\begin{smallmatrix}1 & b \\ 0 & d\end{smallmatrix}\right)\mod M\}.\]
The levels of the Shimura curves defined in \S\ref{subsec.ShimuraCurves} relate to the levels above as follows: \[U_m= U_0(N^{+}p^m)\hspace{10pt}\text{and}\hspace{10pt} \widetilde{U}_m= U_0(N^{+})\cap U_1(p^m).\] The curves $X_m$ and $\widetilde{X}_m$ are the generic fibers of $\mathcal{X}_m=\mathcal{X}_{U_m}$ and $\widetilde{\mathcal{X}}_m=\mathcal{X}_{\widetilde{U}_m}$, respectively.
\end{remark}
 
\subsubsection{Drinfeld level structures}\label{subsubsec.Drinfeld}
Let $m\ge 1$ and $\ell$ be a prime such that $\ell\nmid N^{-}$. Consider idempotents $e_\ell=\smallmat{1}{0}{0}{0}$ and $\bar{e}_\ell=\smallmat{0}{0}{0}{1}$ in $\M_2(\Z/\ell\Z)$ (for $\ell\mid N^+p$, these can be taken to be $e$ and $\bar{e}$ from \S\ref{subsec.Idempotents} via $j_\ell$). The kernel of the action of $e_\ell$ on $A[\ell^m]$ is isomorphic as a group scheme to the image of the action of $\bar{e}_\ell$ on $A[\ell^m]$ and vice-versa. Conjugation by $w_\ell=\left(\begin{smallmatrix}0 & 1 \\ 1 & 0\end{smallmatrix}\right)$ gives an isomorphism of group schemes between $\ker(e_\ell)$ and $\ker(\bar{e}_\ell)$. Therefore, we have a decomposition into isomorphic factors
\begin{equation}\label{eq.DecompositionApm}
A[\ell^m]=\ker(e_\ell)\oplus \ker(\bar{e}_\ell).
\end{equation}
We endow $(A,\iota)$ with a \emph{level $\Gamma_1(\ell^m)$ structure} consisting of a cyclic finite flat $S$-subgroup scheme $H$ of $\ker(e_\ell)$ which is locally free of rank $\ell^m$ and $P$ a generator of $H$ ($P\in H$ such that $\{P,\dots,\ell^m P\}$ is a full set of sections, see \cite[\S 3]{Buzzard}, \cite[\S 1.8]{KatzMazur}). We denote by $(A,\iota,\alpha,(H,P))$ the QM abelian variety $(A,\iota)$ equipped with a naïve level $U$ structure $\alpha$ and a level $\Gamma_1(\ell^m)$ structure $(H,P)$. An \textit{isogeny} (resp. \textit{isomorphism}) between two such quadruples $(A,\iota,\alpha,(H,P))$ and $(A',\iota',\alpha',(H',P'))$ is an isogeny (resp. isomorphism) of QM abelian surfaces $\varphi\colon A\to A'$ such that $\varphi\circ\alpha=\alpha'$, $\varphi(H)=H'$ and $\varphi(P)=P'$. A \emph{level $\Gamma_0(\ell^m)$ structure} can be defined by dropping the choice of a generator $P$.

\begin{proposition}\label{prop.ModuliDrinfeld}
The functor which takes a $\Z_{(\ell)}$-scheme $S$ to the set of isomorphism classes of such quadruplets $(A,\iota,\alpha,(H,P))$ over $S$ is representable by a $\Z_{(\ell)}$-scheme $\mathcal{X}_{U,\Gamma_1(\ell^m)}$, which is proper and finite over $\mathcal{X}_U$, here viewed as a $\Z_{(\ell)}$-scheme. 
Moreover, there is a canonical isomorphism of $\Q$-schemes between the generic fiber of 
$\mathcal{X}_{U,\Gamma_1(\ell^m)}$ and the generic fiber of 
$\mathcal{X}_{U\cap U_1(\ell^m)}$ (for $m=1$, this is \cite[Proposition 4.1]{Buzzard}. The same proof works for $m\ge 1$.) 
\end{proposition}

If $M=\ell_1^{d_1}\dots\ell_s^{d_s}$, one can make sense of a \textit{level $\Gamma_1(M)$ structure} in $(A,\iota)$ via the Chinese Remainder Theorem: a $S$-group scheme $H=\prod_{j=1}^sH_j$ such that $H_j$ is a cyclic finite flat $S$-subgroup scheme $H$ of $\ker(e_{\ell_j})$ which is locally free of rank $\ell_j^{d_j}$, and $P=(P_1,\dots,P_s)$ such that $P_j$ is a generator of $H_j$.

\subsubsection{Correspondece between naïve and Drinfeld level structures}
A simple extension of \cite[Lemma 4.4]{Buzzard} to higher powers of primes followed by the Chinese Remainder Theorem gives a bijection between the naïve level $U_0(M)$ structures (resp. $U_1(M)$) and level $\Gamma_0(M)$ structures (resp. $\Gamma_1(M)$) in $(A,\iota)$. Explicitely, a bijection between the latter and structures of level \[V_0(M)\defeq\{g\in\widehat{\mathcal{O}}_B^\times:\ \exists a,b,d\text{ such that } \bar{\imath}_M\circ\widehat{\pi}_M(g)\equiv\left(\begin{smallmatrix}a & b \\ 0 & d\end{smallmatrix}\right)\mod M\}\]\[\text{(resp. }V_1(M)\defeq\{g\in\widehat{\mathcal{O}}_B^\times:\ \exists a,b\text{ such that } \bar{\imath}_M\circ\widehat{\pi}_M(g)\equiv\left(\begin{smallmatrix}a & b \\ 0 & 1\end{smallmatrix}\right)\mod M\}\text{)}\] is described as follows: if $\alpha$ is a naïve level $V_1(M)$ structure, lift it to a naïve full level $M$ structure $\widetilde{\alpha}\colon\mathcal{O}_B\otimes (\Z/M\Z)\to A[M]$; then $\left(\widetilde{\alpha}\left(\smallmat{\ast}{0}{0}{0}\right),\widetilde{\alpha}\left(\smallmat{1}{0}{0}{0}\right)\right)$ is a level $\Gamma_1(M)$ structure. Conversely, if $(H,P)$ is a level $\Gamma_1(M)$ structure, replace $S$ by an étale surjective cover $\widetilde{S}/S$, using the fact that the functor taking $S$-schemes $\widetilde{S}/S$ to the set of naïve level $V_1(M)$ structures over $\widetilde{S}$ is étale (see \textit{loc. cit.}), over which an isomorphism $\widetilde{\alpha}\colon\mathcal{O}_B\otimes (\Z/M\Z)\stackrel{\sim}{\to}\M_2(\Z/M\Z)\stackrel{\sim}{\to} A[M]$ exists. Precomposing $\alpha$ with an automorphism of $\mathcal{O}_B\otimes (\Z/M\Z)$ mapping $(\widetilde{\alpha}^{-1}(H),\widetilde{\alpha}^{-1}(P))$ to $\left(\smallmat{\ast}{0}{0}{0},\smallmat{1}{0}{0}{0}\right)$ gives a naïve full level $M$ structure defined over $\widetilde{S}$, that can be projected onto a naïve full level $M$ structure $\alpha$ defined over $S$, and thus a level $V_1(M)$ structure induced by $\alpha$ (for $V_0(M)$ and $\Gamma_0(M)$ structures it is the same by forgetting the marked point). Finally, $g\mapsto\norm(g)g^{-1}$ (``transposing about the secondary diagonal'', up to signs) gives an anti-automorphism of $V_0(M)=U_0(M)$ that sends $V_1(M)$ to $U_1(M)$.

\subsubsection{Moduli classes corresponding to CM points}\label{subsec.CMpoints}
Associate to each $\tau\in\mathcal{H}^+$ the QM abelian variety \[A_\tau\defeq\dfrac{i_\infty(B_\infty)\binom{\tau}{1}}{i_\infty(\mathcal{O}_B)\binom{\tau}{1}},\] consisting of the four-dimensional real torus $\M_2(\R)/\iota_\infty(\mathcal{O}_B)$ with the complex structure induced by right multiplication action by the unique $J_\tau\in\M_2(\R)$ such that $J_\tau\binom{\tau}{1}=i\binom{\tau}{1}$. A point $\tau\in\mathcal{H}^+\cap K$ is said to be a \textit{CM point} if there exists an embedding $\iota_\tau\colon K\hookrightarrow B$ such that \[i_\infty\circ\iota_\tau(x)\cdot\binom{\tau}{1}=\binom{\tau}{1},\text{ for all }x\in K^\times,\] where the action is by Möbius transformations. Embedding the order $\mathcal{O}_\tau\defeq\Z\tau\oplus\Z$ into $\mathcal{O}_B$ via $\iota_\tau$, the Main Theorem of Complex Multiplication gives that the following isomorphisms are defined over the Hilbert class field $H_\tau$ of $\mathcal{O}_\tau$: \[A_\tau(H_\tau)=\dfrac{\iota_\tau(K)\otimes_{\Q}\R\binom{\tau}{1}}{\iota_\tau(\mathcal{O}_\tau)\binom{\tau}{1}}\cong\dfrac{\C\tau\oplus\C}{\mathcal{O}_\tau\tau\oplus\mathcal{O}_\tau}\cong\left(\dfrac{\C}{\mathcal{O}_\tau}\right)^{\oplus 2}.\]In particular, the CM point $\vartheta$ for the embedding $\iota_\vartheta=\iota_K^B$ fixed in \S\ref{subsec.QuaternionAlgebras} gives $A_\vartheta\cong(\C/\mathcal{O}_K)^{\oplus 2}$ over $H$.

Now consider the Heegner point $x_{cp^n,m}(1)=[(\iota_K^B,\xi_{cp^n})]\in\widetilde{X}_m$ from \S\ref{subsec.HeegnerPoints}. Via the decomposition $\xi_{cp^n}=b_{cp^n,m}u_{cp^n,m}$ into a product of $b_{cp^n,m}\in B^\times$ and $u_{cp^n,m}\in \widetilde{U}_m$, the Heegner point $[(\iota_K^B,\xi_{cp^n})]=[b_{cp^n,m}^{-1}\cdot(\iota_K^B,\xi_{cp^n})]=[(i_{cp^n,m},u_{cp^n,m})]$, with $i_{cp^n,m}\defeq b_{cp^n,m}^{-1}\iota_K^Bb_{cp^n,m}$, is associated to the CM point $\vartheta_{cp^n,m}\in\mathcal{H}^+\cap K$ fixed by $i_{cp^n,m}$, which is given by $\vartheta_{cp^n,m}=b_{cp^n,m}^{-1}\vartheta$ and is defined up to left multiplication by $B^\times$ (which comes from choosing another decomposition for $\xi_{cp^n}$).
\begin{remark}\label{rem.Factorization}
We are being overcautious in writing dependency on $m$ in the paragraph above: in fact, the CM point $\vartheta_{cp^n,m}$ corresponds in each moduli space $X_m$ (or $\widetilde{X}_m$) to the QM abelian variety $A_{{cp^n}}\defeq (\C/\mathcal{O}_{cp^n})^2$, related to $A_{\vartheta}$ through the cyclic isogeny $\phi_{cp^n}\colon A_{\vartheta}\to A_{{cp^n}}$ of degree $cp^n$: explicitely, if $(A_{\vartheta},\iota,\alpha,(H,P))$ is the moduli class in $\widetilde{X}_m$ corresponding to the CM point $\vartheta$, the moduli class corresponding to $\vartheta_{cp^n,m}$ in $\widetilde{X}_m$ is $(A_{{cp^n}},\iota\circ\phi_{cp^n}^\ast,\phi_{cp^n}\circ\alpha,(\phi_{cp^n}(H),\phi_{cp^n}(P)))$. This means that there is a CM point $\vartheta_{cp^n}$ that is an independent of $m$ common representative of all classes $[b^{-1}_{cp^n,m}\vartheta]$ as $m$ varies, and so there exists $b_{cp^n}$, and therefore $u_{cp^n}$ and $i_{cp^n}$ as above all independent of $m$. By dropping $m$ from the notation above, we are leaving understood that we are choosing such an independent of $m$ factorization of $\xi_{cp^n}$. In this case, $A_{\vartheta_{cp^n}}=A_{cp^n}$.
\end{remark}

\subsubsection{Arithmetic trivializations}
Let $(A,\iota)$ be a QM abelian surface over a $\Z_p$-scheme $S$ and $\mathbf{m}$ denote either an integer $m\in\Z_{\ge0}$ or the symbol $\infty$. Recall that $\mathcal{O}_{B,p}\defeq\mathcal{O}_B\otimes_{\Z}\Z_p$ acts on $A[p^{\mathbf{m}}]$ by quaternionic multiplication, and on $\boldsymbol{\mu}_{p^\mathbf{m}}\times\boldsymbol{\mu}_{p^\mathbf{m}}$, where $\boldsymbol{\mu}_{p^\mathbf{m}}$ is the $S$-group scheme of $p^\mathbf{m}$-th roots of unity, by left matrix multiplication via $i_p$. An \emph{arithmetic trivialization} on $A[p^{\mathbf{m}}]$ is an isomorphism of finite flat group schemes over $S$ \[\beta\colon\boldsymbol{\mu}_{p^{\mathbf{m}}}\times \boldsymbol{\mu}_{p^{\mathbf{m}}}\stackrel{\sim}{\longrightarrow} A[p^{\mathbf{m}}]^0\]
which is equivariant with respect to the action of $\mathcal{O}_{B,p}$, where $\smallbullet^0$ denotes the connected component of the identity of an $S$-group scheme. Via the decomposition into isomorphic factors \eqref{eq.DecompositionApm} induced by $e_p$ and $\bar{e}_p$ in the notation of \S\ref{subsubsec.Drinfeld} (which commutes with inverse limits and thus extends to ${\mathbf{m}}=\infty$), giving an arithmetic trivialization of $A[p^{\mathbf{m}}]$ is equivalent to giving an isomorphism 
$\beta\colon\boldsymbol{\mu}_{p^{\mathbf{m}}} \stackrel{\sim}{\rightarrow} e_pA[p^{\mathbf{m}}]^0$ of finite flat connected group schemes over $S$, equivariant for the action of $e_p\mathcal{O}_{B,p}$.

\begin{remark}\label{rem.OrdinaryReduction}
The existence of an arithmetic trivialization on $A[p^\infty]$ implies that $A$ is an ordinary abelian scheme over $S$.
\end{remark}

An arithmetic trivialization $\beta$ of $A[p^\infty]$ induces for each integer $m\geq 1$ an arithmetic trivialization
$\bar\beta^{(m)} \colon\boldsymbol{\mu}_{p^m}\stackrel{\sim}{\rightarrow} e_pA[p^m]^0.$ Furthermore, $\beta_m$, an arithmetic trivialization of $A[p^m]$, is said to be \emph{compatible} with a given arithmetic trivialization $\beta$ of $A[p^\infty]$ if the composition 
\[\boldsymbol{\mu}_{p^m}\overset{\beta_{m} }\longrightarrow 
e_pA[p^m]^0\overset{(\bar\beta^{(m)})^{-1}}\longrightarrow\boldsymbol{\mu}_{p^m}\] is the identity.

\subsubsection{Igusa towers}\label{subsec.IgusaTowers}
Let us consider $\mathcal{X}_1=\mathcal{X}_{U_1(N^{+})}$ as a $\Z_{(p)}$-scheme and let $\X_1$ denote the special fiber of $\mathcal{X}_1$ (the same constructions can be performed over $\mathcal{X}_0$). 
Denote by $\mathbf{Ha}$ the Hasse invariant of $\X_1$ (\textit{cf.} \cite[\S6]{Kassaei}) and by $\widetilde{\mathbf{Ha}}$ a lift of $\mathbf{Ha}$ to $\mathcal{X}_1$ (\textit{cf. ibid.} \S7). Let $\mathcal{X}_1^\ord\defeq\mathcal{X}_1[1/\widetilde{\mathbf{Ha}}]$, an affine open $\Z_{(p)}$-subscheme of $\mathcal{X}_1$ which, as in Proposition \ref{prop.ModuliNaive}, represents the moduli problem that associates to any $\Z_{(p)}$-scheme $S$ the isomorphism classes of triplets $(A,\iota,\alpha)$ where $(A,\iota)$ is an ordinary QM abelian surface over $S$ and $\alpha$ a na\"{\i}ve level $U_1(N^+)$ structure. Let $\X_1^\mathrm{ord}$ denote the special fiber of $\mathcal{X}_1^\ord$. Let $\mathcal{A}^\ord$ be the universal ordinary abelian variety over $\mathcal{X}_1^\ord$. For any $\Z_{(p)}$-algebra $R$, denote $\mathcal{A}^\ord_R\defeq \mathcal{A}^\ord\otimes_{\Z_{(p)}}R$, and $\mathcal{A}^\ord_n\defeq \mathcal{A}^\ord_{\Z/p^n\Z}$.

Given two group schemes $G$ and $H$ equipped with a left $\mathcal{O}_{B,p}$-action, 
$\mathrm{Isom}_{\mathcal{O}_{B,p}}(G,H)$ denotes the set of isomorphisms of groups schemes $G\rightarrow H$ which are equivariant for the action of $\mathcal{O}_{B,p}$. Let $m\geq 1$ and $n\geq 1$ be integers. Consider 
\[\mathcal{P}_{m,n}(S)=\mathrm{Isom}_{\mathcal{O}_{B,p}}\left(\boldsymbol{\mu}_{p^m}\times\boldsymbol{\mu}_{p^m},\mathcal{A}_n^\ord[p^m]^0\right),\]
the set of arithmetic trivializations on $\mathcal{A}^\ord_{n}[p^m]$. 

\begin{proposition}\label{prop.ModuliIgusa}
The moduli problem $\mathcal{P}_{m,n}$ is represented by a $\Z/p^n\Z$-scheme $\mathrm{Ig}_{m,n}$, the \emph{$p^m$-layer of the Igusa tower over $\Z/p^n\Z$}, which is finite étale over $\mathrm{Ig}_{0,n}$.
\end{proposition}

By the universality of $\mathcal{A}^\ord_n$, 
the $\Z/p^n\Z$-scheme $\mathrm{Ig}_{m,n}$ represents the moduli problem which associates 
to any $\Z/p^n\Z$-scheme $S$ the set of isomorphism classes of quadruplets $(A,\iota,\alpha,\beta)$ consisting of a point $(A,\iota,\alpha)\in\X^\mathrm{ord}_1$ and an arithmetic trivialization $\beta$ of $A[p^m]$. 

For integers $m\geq 0$ and $n\geq 1$ and a $\Z/p^n\Z$-scheme $S$, the canonical monomorphism $\boldsymbol{\mu}_{p^m}\hookrightarrow \boldsymbol{\mu}_{p^{m+1}}$ of $S$-group schemes induces a canonical map $\mathrm{Ig}_{m+1,n}\rightarrow\mathrm{Ig}_{m,n}$. We can therefore consider the $\Z/p^n\Z$-scheme \[\widehat{\mathrm{Ig}}_n=\varprojlim_m\mathrm{Ig}_{m,n},\] called the \textit{Igusa tower over $\Z/p^n\Z$}.

\begin{proposition}
The $\Z/p^n\Z$-scheme $\widehat{\mathrm{Ig}}_n$ represents the moduli problem over $\Z/p^n\Z$ \[\mathcal{P}_n(S)=\mathrm{Isom}_{\mathcal{O}_{B,p}}\left(\boldsymbol{\mu}_{p^\infty}\times\boldsymbol{\mu}_{p^\infty},\mathcal{A}_n^\ord[p^\infty]^0\right)\]
classifying the set of arithmetic trivializations of $\mathcal{A}^\ord_n[p^\infty]$, or, equivalently, the moduli problem which associates to a $\Z/p^n\Z$-scheme $S$ the set of isomorphism classes of quadruplets $(A,\iota,\alpha,\beta)$ for each integer $m\geq 1$ consisting of a point $(A,\iota,\alpha)\in\X^\mathrm{ord}_1$ over $S$ equipped with a level $U_1(N^+)$ structure $\alpha$ and a family of arithmetic trivializations $\beta_{m}$ of $\mathcal{A}_n^\ord[p^m]$, one for each integer $m\geq 1$, such that there is a trivialization $\beta$ of $\mathcal{A}_n^\ord[p^\infty]$ for which $\beta_{m} $ is compatible with $\beta$, for all $m\geq1$.
\end{proposition}

Define finally the \textit{Igusa tower over $\Z_p$} to be the $\Z_p$-formal scheme 
\[\widehat{\text{Ig}}=\varprojlim_n\widehat{\mathrm{Ig}}_n=\varprojlim_n\varprojlim_m\mathrm{Ig}_{m,n}\]
where the inverse limit is computed with respect to the canonical maps
induced by the canonical projection maps 
$\Z/p^{n+1}\Z\twoheadrightarrow\Z/p^{n}\Z$ for each $n\geq 1$. Refer to \cite[Chapter 8]{Hida.book}, \cite[\S2.1]{Hida.Control} and \cite[\S2.5]{Burungale} for details on Igusa towers.

\subsection{Relative Chow motives}
We recall some general results on representations of algebraic groups obtained in \cite{Ancona}.
For this subsection, let ${G}$ be an algebraic group, $({G},D)$ be a PEL Shimura datum, $U$ a compact open subset of ${G}(\A_{\Q,\mathrm{fin}})$, $X$ the canonical model of the Shimura variety $\mathrm{Sh}_U({G},D)$ of level $U$ over the reflex field $F$ and $\pi\colon\mathcal{A}\to X$ be the universal PEL abelian variety. If $H^\bullet$ is a Weil cohomology, there is a realization functor \[R\colon\CHM_F(X)\to\operatorname{Vec}_F^\pm\] from the category $\operatorname{CHM}_F(S)$ of relative Chow motives $(X,p,n)$ (where $X\to S$ is a smooth projective scheme, $p\in\operatorname{CH}^{\dim(X)}(X\times_S X)_F$ satisfies $p^2=p$ and $n\in\Z$), to the category of finite dimensional $\Z$-graded $F$-vector spaces compatible with the ``varieties-to-motives functor''
\begin{center}
\begin{tikzcd}[column sep = large]
\CHM_F(S)\arrow[r,"R"] & \operatorname{Vec}_F^\pm \\
\operatorname{Var}_S \arrow[u,"M"]\arrow[ur,swap, "H^\bullet"] &
\end{tikzcd}
\end{center} 
and being such that $R(h^i(\mathcal{A}))=H^i(\mathcal{A})$ for all $0\le i\le 2g$ (\textit{ibid.}, Proposition 3.5). Furthermore, every decomposition into direct summands of $H^1(\mathcal{A})^{\otimes n}$ lifts canonically to a decomposition into direct summands of $h^1(\mathcal{A})^{\otimes n}$ (\textit{ibid.} Théorème 6.1). In particular, we have
\begin{itemize}
\item A \textit{Hodge realization} $\operatorname{Hod}_{X/F}\colon \CHM_F(X)\to\VHS_F(X(\C))$ into variations of Hodge structures, when \[H^\bullet\colon (f\colon V\to X)\mapsto \bigoplus_i (R^if_*F_V)_x,\] for $F\hookrightarrow\C$ a number field and $x\in X(F)$ a base point;

\item An \textit{étale realization} $\Et_{X/F}\colon\CHM_F(X)\to\CHM_{\et,F}(X)$ into lisse étale sheaves, when \[H^\bullet\colon (f\colon V\to X)\mapsto \bigoplus_i \left(R^if_*F_V\right)_{\bar{x}},\] for $\ell$ a prime, $F$ an $\ell$-adic field and a base geometric point $\bar{x}\in X(\overline{\Q}_\ell)$.
\end{itemize}

By \textit{ibid.} Théorème 8.6, the canonical construction functor (\cite[\S 1.18]{Pink}) lifts through $\operatorname{Hod}_{X/F}$ to a functor
\[\Anc_{X/F}\colon\operatorname{Rep}_F({G})\longrightarrow\operatorname{CHM}_F(X),\]where $\operatorname{Rep}_F({G})$ is the category of $F$-representations of ${G}$, with the following properties:
\begin{itemize}
\item $\Anc_{X/F}$ is $F$-linear, preserves duals and tensor products;
\item If $V_{G}$ is the standard algebraic representation of ${G}$, $\Anc_{X/F}(V_{G}(F))=h^1(\mathcal{A})$ (following the normalization in \cite[Remark 6.2.3]{LoefflerSkinnerZerbes}, see also \cite[\S 8]{Torzewski}).
\end{itemize}
See \cite{CortiHanamura}, \cite{DeningerMurre} and \cite[\S2]{Ancona} for properties of the decomposition of Chow motives.

\subsection{Symmetric tensors}\label{subsec.TSym}
Let $n\ge 1$ be an integer, $\mathfrak{S}_n$ be the group of permutations on $n$ letters and $H$ be a free abelian module of finite rank over a ring $R$ of characteristic 0. The group $\mathfrak{S}_n$ acts on the $n$-th tensor power $H^{\otimes n}$ by permutation of the factors of each elementary tensor. The set of $\mathfrak{S}_n$-invariant elements are the \textit{symmetric tensors}:\[\TSym^{n}(H)\defeq \{\eta\in H^{\otimes n}:\ \sigma(\eta)=\eta\}.\]
Similarly to the symmetric algebra $\Sym^\bullet(H)$, $\TSym^{\bullet}(H)\defeq\bigoplus_{n\ge 0}\TSym^{n}(H)$ is a graded algebra with the usual sum and with the symmetrization of the tensor product $\odot$. Identifying $H$ with $\TSym^1(H)$, $H$ can be embedded in $\TSym^{\bullet}(H)$, so the universal property of the symmetric algebra gives a graded algebra morphism $\Sym^{\bullet}(H)\to\TSym^{\bullet}(H)$, which specializes in degree $n$ to \[\Sym^{n}(H)\longrightarrow\TSym^{n}(H)\colon h_1^{\cdot e_i}\cdots h_d^{\cdot e_d} \longmapsto e_1!\cdots e_d! h_1^{\odot e_i}\odot\dots\odot h_d^{\odot e_d}.\]
If $n!$ is invertible in $R$, the map above can be inverted and so becomes an isomorphism. There is another natural link between $\Sym$ and $\TSym$: if $\smallbullet\,^\vee\defeq\Hom(\,\smallbullet\,,R)$ denotes the linear dual of an $R$-module, we have a canonical isomorphism
\begin{equation}\label{eq.IsomSymTSym}
\TSym^{n}(H^\vee)\cong\left(\Sym^{n}(H)\right)^\vee.
\end{equation}

If $\mathscr{F}$ is a locally free sheaf on a variety $X$ over a field of characteristic 0, $\Sym^{\bullet}(\mathscr{F})$ makes sense as a locally free sheaf on $X$, but that is not always the case with $\TSym$. However, since $H$ is a free module, $\TSym^{\bullet}(H)$ coincides with $\Gamma^{\bullet}(H)$, the divided power algebra of $H$ (see \cite{Lundkvist} for more details on the relation between these two objects) which does sheafify well: as pointed out in \cite[\S2.2]{KingsLoefflerZerbes.ERL}, if $\mathscr{F}$ is a locally free sheaf on a variety $X$ over a field of characteristic 0, $\Gamma^n(\mathscr{F})$ and therefore $\TSym^n(\mathscr{F})$ define a locally free sheaf on $X$ and, in particular, one can talk about symmetric tensors of coefficient sheaves of étale cohomology. See also \cite[\S 12.2.2]{Kings.TSym}.

Let $M=(V,p,m)$ be a relative Chow motive as in the previous subsection. The group $\mathfrak{S}_n$ acts on $M^{\otimes n}=(V^n,p,m)$, where $V^n$ denotes the $n$-th fiber product power of $V$ over $X$. The image of $M^{\otimes n}$ under the projector $\pi_{\mathfrak{S}_n}=\sum_{\sigma\in\mathfrak{S}_n}\sigma$ is the $n$-th symmetric power of $M$, denoted $\Sym^n(M)=(V^n,\pi_{\mathfrak{S}_n}\circ p,m)$. We then define $\TSym^n(M)\defeq\left(\Sym^n(M^\vee)\right)^\vee$.

\subsection{Hecke Characters}\label{subsec.HeckeCharacters}
An \textit{algebraic Hecke character} over a finite Galois extension $F/\Q$ is a 1-dimensional $F$-representation $\chi\in\Hom_{\cont}(\A_{F}^\times,\overline{\Q}^\times)$, which is said to have \textit{infinity type $(i_1,i_2)$} if, for each $x=x_{\mathrm{fin}}x_\infty\in\A^\times_F$ decomposed in its finite $x_{\mathrm{fin}}\in\A^\times_{F,\mathrm{fin}}$ and infinity $x_\infty\in\A^\times_{F,\infty}$ components, \[\chi(x)=\chi(x_{\mathrm{fin}})x_\infty^{i_1}\bar{x}_\infty^{i_2}.\]
We decompose $\chi$ into a finite and a infinite part given respectively by: \[\chi_{\mathrm{fin}}(x)\defeq\chi(x_\mathrm{fin})\hspace{10pt}\text{and}\hspace{10pt}\chi_\infty(x)\defeq x_\infty^{i_1}\bar{x}_\infty^{i_2}.\] Through the Artin reciprocity map, $\chi$ can be seen as a \textit{Galois character} over $F$, that is, a character of $\Gal(\overline{F}/F)$. All characters we consider are algebraic, so we drop the adjective for the sake of brevity.

Now let $F=K$ and $\widetilde{c}\in\Z_{\ge0}$ coprime with $N$ (but not necessarily with $p$). We say that $\chi$ has \textit{finite type $(\widetilde{c},\mathfrak{N}^+,\psi)$} if the conductor $\mathfrak{c}$ of $\chi$ is divisible by $\widetilde{c}$ and the restriction of $\chi$ to $(\mathcal{O}_{\widetilde{c}}/\mathfrak{N}^+\cap\mathcal{O}_{\widetilde{c}})^\times$ (by restricting to $\widehat{\mathcal{O}}_{\widetilde{c}}^\times$ and projecting to the $\mathfrak{N}^+$-component) is $\psi^{-1}$, a finite order character.

\section{Generalized Heegner classes}\label{sec.GeneralizedHeegnerClasses}
For the whole chapter fix integers $k,r,n,m,c\in\Z_{\ge0}$ such that $n\ge m$, $k\defeq 2r+2\ge 2$ and $c\geq 1$ is coprime with $ND_Kp$. Recall the Shimura curve $\widetilde{X}_m$ of level $\widetilde{U}_m$ introduced in \S\ref{subsec.ShimuraCurves} and let $\widetilde{\mathcal{A}}_m$ be the the correspondent universal abelian variety in view of \S\ref{subsec.Moduli}. Via the main theorem of complex multiplication, the elliptic curve $E\defeq \C/\mathcal{O}_K$ has a model defined over $H$. We shall denote by $(A,\iota,\alpha)$ a fixed QM abelian variety, consisting of $A\defeq A_{\vartheta}= E\oplus E$ (\textit{cf.} \S\ref{subsec.CMpoints}), the embedding $\iota\colon\mathcal{O}_B\hookrightarrow\End(A)=\mathcal{O}_K\oplus\mathcal{O}_K$ induced by $\iota_K^B$, and equipped with a level $\widetilde{U}_m$ structure $\alpha$. Recall that the idempotents $e$ and $\bar{e}$ from \S\ref{subsec.Idempotents} act under $j$ as projections into each factor. Thus we have a decomposition $A=eA\oplus\bar{e}A$ into isomorphic factors $eA\cong\bar{e}A\cong E$.

\subsection{Generalized Kuga--Sato variety}
We recall basic definitions following \cite[\S 2.6]{Brooks}.

\begin{definition}
The \textit{Kuga--Sato variety} of weight $k$ over $\widetilde{X}_m$ is $\widetilde{\mathcal{A}}_m^r$, the $r$-fold fiber product of $\widetilde{\mathcal{A}}_m$ over $\widetilde{X}_m$. The \emph{generalized Kuga--Sato variety} of weight $k$ over $\widetilde{X}_m$ is \[W_{k,m}\defeq\widetilde{\mathcal{A}}_m^{r}\times_{\widetilde{X}_m} A^{r}.\]
\end{definition}

Since $A$ is defined over $H$ and both $\widetilde{\mathcal{A}}_m$ and $\widetilde{X}_m$ have models over $\Q$, $W_{k,m}$ has a model over $H$, which we fix henceforth. The generalized Kuga--Sato variety $W_{k,m}$ has relative dimension $4r+1=2k-1$ over $\widetilde{X}_m$, the first factor contributing with $2r+1$ and the second with $2r$.

Denote by $\pi_{\mathcal{A}}\colon\widetilde{\mathcal{A}}_m\to \widetilde{X}_m$ the canonical projection. Since the complex structure commutes with the action of $B$, the idempotents $e$ and $\bar{e}$ from \ref{subsec.Idempotents} induce a decomposition of variations of Hodge structures \[R^1\pi_{\mathcal{A},*}\Q=eR^1\pi_{\mathcal{A},*}\Q\oplus\bar{e}R^1\pi_{\mathcal{A},*}\Q\] into isomorphic factors. The projector $\epsilon_{\mathcal{A}}\in\Corr_{\widetilde{X}_m}^0(\widetilde{\mathcal{A}}_m^{r},\widetilde{\mathcal{A}}_m^{r})$ from \cite[Theorem 5.8.iii]{Besser} induces a projection from $R^{2r}\pi_{\mathcal{A},*}\Q$ to $\Sym^{2r}(eR^1\pi_{\mathcal{A},*}\Q)$: after finding $\Sym^{2}(eR^1\pi_{\mathcal{A},*}\Q)$ inside $R^2\pi_{\mathcal{A},*}\Q$ (\textit{ibid.}, Theorem 5.8.ii), we have
\begin{align*}
\Sym^{2r}(eR^1\pi_{\mathcal{A},*}\Q)&\longmono\Sym^{r}(\Sym^2(eR^1\pi_{\mathcal{A},*}\Q))&\text{(\textit{ibid.}, Proposition 5.6.vii)}\\
&\longmono (R^2\pi_{\mathcal{A},*}\Q)^{\otimes r}\longmono R^{2r}\pi_{\mathcal{A},*}\Q&\text{(\textit{ibid.}, Theorem 5.8.ii)}\\
\end{align*}
the last step following from the Künneth formula. Essentially, $\epsilon_\mathcal{A}$ kills all factors in the Künneth formula apart from $(R^2\pi_{\mathcal{A},*}\Q)^{\otimes r}$, which gets projected into $\Sym^{2r}(eR^1\pi_{\mathcal{A},*}\Q)$. We let $\epsilon_A\in\Corr^{2r}(A^r,A^r)$ to be the projector obtained by specializing each factor or $\widetilde{\mathcal{A}}_m^r$ to $A$. We also define $\epsilon_W\defeq\epsilon_{\mathcal{A}}\epsilon_A\in\Corr^r_{\widetilde{X}_m}(W_{k,m},W_{k,m})$. See also \cite[(1.4.4), (2.1.2)]{BertoliniDarmonPrasanna} for the elliptic counterpart.

\begin{remark}\label{rem.TSym}
We identify the sheaf $\Sym^{2r}\left((eR^1\pi_{\mathcal{A},*}\Q)^\vee\right)$ with $\TSym^{2r}\left((eR^1\pi_{\mathcal{A},*}\Q)^\vee\right)$ via the isomorphism \eqref{eq.IsomSymTSym}. This makes the \textit{dual} of the image of $\epsilon_W$ to lie in $\Sym^{2r}(eR^1\pi_{\mathcal{A},*}\Q)$, from which representations of modular forms will arise (see \S\ref{subsec.ClassesModularForms}).
\end{remark}

\subsection{Generalized Heegner cycles}
The extension of Bertolini--Darmon--Prasanna's definition \cite[\S 2.3]{BertoliniDarmonPrasanna} to the indefinite quaternionic case is due to Brooks \cite[\S 6.2]{Brooks}.

Let $F$ be a finite extension of $F_{cp^n}$, $(A',\iota')$ be a QM abelian surface and $\phi\colon A\to A'$ be an isogeny defined over $F$ whose kernel intersects trivially the image of $\alpha$. Under these conditions, $\alpha'\defeq\phi\circ\alpha$ gives a level $\widetilde{U}_m$ structure on $A'$, and therefore a point $(A',\iota',\alpha')$ in $\widetilde{\mathcal{A}}_m$ defined over $F$ (\textit{cf.} \cite[Theorem 3.2]{Shimura.CCFZFAC}), which allows the graph of $\phi$ to be embedded into $A\times_{\widetilde{X}_m}\widetilde{\mathcal{A}}_m$ via the ``universality'' inclusion $A'\hookrightarrow\widetilde{\mathcal{A}}_m$, giving an embedding $\operatorname{graph}(\phi)^r\hookrightarrow W_{k,m}$. The graph of $\phi$ has dimension $2$ so, as a rational cocycle in $W_{k,m}$ defined over $F$, $\operatorname{graph}(\phi)^r$ has codimension $2r+1=k-1$.

\begin{definition}
The \emph{generalized Heegner cycle} associated to the isogeny $\phi$ is \[\Delta_{\phi,m}^{[k]}\defeq \epsilon_W\operatorname{graph}(\phi)^r\in\epsilon_W\CH^{k-1}(W_{k,m}\otimes_{H}F)_{\Q}.\]
\end{definition}

Chow groups can be identified with motivic cohomology groups (\textit{cf.} \cite[Corollary 19.2]{MazzaVoevodskyWeibel}), so there is an isomorphism
\begin{equation}
r_\mot\colon\epsilon_W\CH^{k-1}(W_{k,m}\otimes_{H}F)_{\Q}\stackrel{\sim}{\longrightarrow}\epsilon_W H^{2k-2}_{\mot}(W_{k,m}\otimes_{H}F,\Q(k-1)).
\end{equation}
Fix for the rest of the section $L$ a $p$-adic field with an embedding $\sigma_L\colon K\hookrightarrow L$ (if $p$ splits in $K$, one can simply take $L=\Q_p$). Composing the map $r_{\mathrm{mot}}$ with the realization map into the étale cohomology with coefficients in $L$ (that is, the cycle map from étale cohomology),
\begin{equation}
r_{{\et},L}\colon\epsilon_{W} H_\mathrm{mot}^{2k-2}(W_{k,m}\otimes_HF,\Q(k-1))\longrightarrow\epsilon_{W} H_{\et}^{2k-2}(W_{k,m}\otimes_HF,L(k-1))
\end{equation}
gives an étale class $\Delta_{\phi,m,\et}^{[k]}\defeq r_{\et,L}\circ r_{\mot}(\Delta_{\phi,m}^{[k]})$ associated to $\Delta_{\phi,m}^{[k]}$.

\begin{remark}
If $\phi$ is the isogeny $\phi_{cp^n}\colon A_{\vartheta}\to A_{\vartheta_{cp^n}}$ from \S\ref{subsec.CMpoints}, we write $cp^n$ instead of $\phi_{cp^n}$ when it appears as an index in all notation to follow.
\end{remark}

\subsection{Lieberman trick}\label{subsec.Lieberman}
By a technique known as Lieberman trick, one can replace the base scheme $W_{k,m}$ with the simpler variety $\widetilde{X}_m$, to the cost of having a slightly more complicated coefficient system. By the Künneth decomposition theorem in étale cohomology (see \cite[Theorem 8.21]{Milne.EC}), we have \[H_{\et}^{2k-2}(W_{k,m},L)=\bigoplus_{i+j=2k-2}H_{\et}^i(\widetilde{\mathcal{A}}_m^r,L)\otimes H_{\et}^j(A^r,L).\]
Let $\pi_{\mathcal{A}}\colon\widetilde{\mathcal{A}}_m\to \widetilde{X}_m$ and $\pi_A\colon A\hookrightarrow\widetilde{\mathcal{A}}_m\stackrel{\pi_\mathcal{A}}{\longrightarrow} \widetilde{X}_m$ be the canonical projections. Since the Leray spectral sequence degenerates at page 2 (\emph{cf.} \cite[\S 2.4]{Deligne.Leray}), the groups in the right-hand side decompose as \[H_{\et}^i(\widetilde{\mathcal{A}}_m^r,L)=\bigoplus_{a+b=i} H^a_{\et}(\widetilde{X}_m,R^b\pi_{\mathcal{A},*}L)\text{ and }H_{\et}^j(A^r,L)=\bigoplus_{a+b=j}H^{a}_{\et}(\widetilde{X}_m,R^{b}\pi_{A,*}L)\] and the image of the projectors $\epsilon_{\mathcal{A}}$ and $\epsilon_A$ are motives whose Betti realizations are of type\linebreak $((k-1,0),(0,k-1)),$ as in \cite{Besser} (see the proof of Theorem 5.8 and the paragraph after the proof of Proposition 5.9 in \emph{op. cit.}). Therefore, the only summand remaining after applying the projectors corresponds to the indexes $i=j=k-1$, so
\begin{center}
\begin{tikzcd}[column sep = tiny]
\epsilon_W H_{\et}^{2k-2}(W_{k,m},L)\arrow[r,equal]&H_{\et}^1\left(\widetilde{X}_m,\TSym^{2r}(eR^1\pi_{\mathcal{A},*}L)\right)\otimes H_{\et}^1\left(\widetilde{X}_m,\TSym^{2r}(eR^1\pi_{A,*}L)\right)\arrow[d,"\mathrm{PD}"]\\ & H_{\et}^2\left(\widetilde{X}_m,\TSym^{2r}(eR^1\pi_{\mathcal{A},*}L)\otimes \TSym^{2r}(eR^1\pi_{A,*}L)\right),
\end{tikzcd}
\end{center}
where $\mathrm{PD}$ is the Poincaré duality map (see for example \cite[Corollary 11.2]{Milne.EC}).

\subsection{Generalized Heegner classes}\label{subsec.GeneralizedHeegnerClasses}

Twisting the cohomology groups above by $k-1=2r+1$ gives a map
\begin{equation}\label{eq.LTtwisted}
\epsilon_W H_{\et}^{2k-2}(W_{k,m},L(k-1))\longrightarrow H_{\et}^2\left(\widetilde{X}_m,\TSym^{2r}(eR^1\pi_{\mathcal{A},*}L)\otimes \TSym^{2r}(eR^1\pi_{A,*}L)(2r+1)\right).
\end{equation}
It will be convenient to have the twists distributed in the following way: denoting \[\mathscr{M}^{2r}_{\et}\defeq \TSym^{2r}(eR^1\pi_{\mathcal{A},*}L(1))\otimes \TSym^{2r}(eR^1\pi_{A,*}L),\] we have
\begin{equation}\label{eq.distributingTwists}
\TSym^{2r}(eR^1\pi_{\mathcal{A},*}L)\otimes \TSym^{2r}(eR^1\pi_{A,*}L)(2r+1)\cong\mathscr{M}^{2r}_{\et}(1).
\end{equation}
The composition of $r_\mot$, $r_{\et,L}$, \eqref{eq.LTtwisted} and the isomorphism induced by \eqref{eq.distributingTwists} gives a map

\begin{equation}\label{eq.LTfullmap}
\epsilon_{W} H_\mathrm{mot}^{2k-2}(W_{k,m}\otimes_HF,\Q(k-1))\longrightarrow H_{\et}^2\left(\widetilde{X}_m\otimes_\Q F, \mathscr{M}^{2r}_{\et}(1)\right).
\end{equation}

\begin{definition} The image of $\Delta_{\phi,m}^{[k]}$ by \eqref{eq.LTfullmap} is the 
\textit{generalized Heegner class} $z_{\phi,m}^{[k]}$.
\end{definition}

\subsection{Representations associated to motives}\label{subsec.RepresentationsAssociatedToMotives}
The sheaf $\mathscr{V}^{2r}_{\et}\defeq\TSym^{2r}(eR^1\pi_{\mathcal{A},*}L(1))$ is the étale realization of
\begin{equation}\label{eq.MotiveV}
\mathscr{V}^{2r}\defeq\TSym^{2r}(eh^1(\widetilde{\mathcal{A}}_m)(1))\in\CHM_{K}(\widetilde{X}_m),
\end{equation}
so $\mathscr{V}^{2r}$ comes from a $K$-representation of $\boldG$ defined in \S\ref{subsec.ShimuraCurves}. Since $\mathbf{G}(K)=(B\otimes_\Q K)^\times$, the the standard representation of $\mathbf{G}$ over $K$ is $B\otimes_\Q K\cong\mathrm{M}_2(K)$ (via $I_B$ when $p$ is split and $i_\mathcal{M}\circ I_B$ from \S\ref{subsec.Idempotents}), which is 4-dimensional over $K$.
By \cite[Corollaire 2.6]{Ancona}, $h^1(\widetilde{\mathcal{A}}_m)^\vee\cong h^1(\widetilde{\mathcal{A}}_m)(1)$ and, since $e(h^1(\widetilde{\mathcal{A}}_m))^\vee\cong (e^\dagger h^1(\widetilde{\mathcal{A}}_m))^\vee$, we have \[eh^1(\widetilde{\mathcal{A}}_m)(1)\cong (e^\dagger h^1(\widetilde{\mathcal{A}}_m))^\vee.\] 
The idempotents $e^\dagger$ and $\bar e^\dagger$ split $\mathrm{M}_2(K)$ into two 2-dimensional components isomorphic to $K^2$, thus \[\mathscr{V}^{2r}\cong\TSym^{2r}\left((e^\dagger h^1(\widetilde{\mathcal{A}}_m))^\vee\right)\cong\Anc_{\widetilde{X}_m/K}\left(\TSym^{2r}((K^2)^\vee)\right)\cong\Anc_{\widetilde{X}_m/K}\left((\Sym^{2r}(K^2))^\vee\right).\]

Similarly, $\TSym^{2r}(eR^1\pi_{A,*}L)$, is the étale realization of $\TSym^{2r}(eh^1(A))$, but since $A$ is not a universal object over any Shimura variety, it is not imediately obvious what representation would correspond to this motive. However, since $eA=E$, we have $eh^1(A)=h^1(E)$. The zero-dimensional PEL Shimura variety associated to the torus $\boldH\defeq\Res_{K/\Q}(\G_m)$ with level $U_{\mathfrak{N}^+,cp^n}$ has a canonical model $S_{cp^n}$ defined over $K$ (see the paragraph ``CM-tori'' at \cite[p. 347]{Milne.SV}), and its points correspond to all $K$-isomorphism classes of elliptic curves defined over $F_{cp^n}$ with CM by $\mathcal{O}_K$ (in other words, the $\Gal(F_{cp^n}/K)$-orbit of $E$). Denote by $\mathcal{E}_{cp^n}$ the universal elliptic curve over $S_{cp^n}$. Then $E\hookrightarrow\mathcal{E}_{cp^n}$ is a scheme over $S_{cp^n}$, so $h^1(E)\in\CHM_K(S_{cp^n})$. Furthermore, since there are only finitely many moduli classes, $h^1(E)$ is a direct summand of $h^1(\mathcal{E}_{cp^n})$. Then there is some direct summand $W$, 1-dimensional over $K$, of the standard $K$-representation of $\boldH$ (which is $K$, as $\boldH(K)=\GL(K)$) such that $\Anc_{S_{cp^n}/K}(W)=h^1(E)=eh^1(A)$, and therefore \[\TSym^{2r}(eh^1(A))=\Anc_{S_{cp^n}/K}\left(\TSym^{2r}(W)\otimes_{\Q}K\right)\in\CHM_{K}(S_{cp^n})\](since $\TSym^{2r}(eh^1(A))$ is a Chow motive defined over $\Q$, seeing it as a motive over $K$ prompts a base change in the associated representation.)

\subsection{$j$-components}\label{subsec.jComponents}
The representation $W(\Q)$ is 2-dimensional over $\Q$, so $\TSym^{2r}(W)(\Q)$ is afforded by the (dual) space of degree $2r$ polynomials in 2 variables and with coefficients in $\Q$, thus being $(2r+1)$-dimensional over $\Q$. Therefore $\TSym^{2r}(W)(K)=\TSym^{2r}(W)(\Q)\otimes_{\Q}K$ is a $(2r+1)$-dimensional $K$-representation that splits into a direct sum of 1-dimensional representations (characters) $W^{[j]}$ for $0\le j\le 2r$ in which complex multiplication by $x\in\mathcal{O}_K$ acts as multiplication by $x^{2r-j}\bar{x}^j$, by which we mean the correspondences induced by both maps act the same way over the motive $eh^1(A)$. This is because a $K$-representation of $\boldH$ is the direct sum of representations of the form $\sigma^{i_1}\otimes\bar{\sigma}^{i_2}$, where $i_1,i_2\in\Z_{\ge 0}$ and $\sigma\colon K\hookrightarrow K$ is a fixed embedding seen as the standard representation of $\boldH$ (see \cite[Remarque 4.8]{Ancona.Thesis}), and since there are $2r+1$ of such factors for $i_1+i_2=2r$, those are all of them. Define \[h^{(2r-j,j)}(A)\defeq\Anc_{S_{cp^n}/K}(W^{[j]}).\]
More generally, an embedding $\sigma_F\colon K\hookrightarrow F$ can be seen as an $F$-representation of $\boldH$. Consider the algebraic representation $V^{(2r-j,j)}$ of $\boldG\times\boldH$ which over $F$ is given by \[V^{(2r-j,j)}(F)=\left(\Sym^{2r}(F^2)\right)^\vee\boxtimes\left(\sigma_F^{2r-j}\otimes\bar{\sigma}_F^{j}\right),\]
where $\boxtimes$ denotes the external tensor product (we often simplify notation by suppressing the lower level tensor product and denote this one by $\otimes$ instead). Define also the motive \[\mathscr{M}^{(2r-j,j)}\defeq \mathscr{V}^{2r}\otimes h^{(2r-j,j)}(A)\in\CHM_K(\widetilde{X}_m\times S_{cp^n}).\] Combining all of the above, it follows from the tensoriality of the Ancona functor that \[\mathscr{M}^{(2r-j,j)}=\Anc_{\widetilde{X}_m\times S_{cp^n}/K}\left(V^{(2r-j,j)}(K)\right).\]

The projection $\operatorname{TSym}^{2r}(eh^1(A))\twoheadrightarrow h^{(2r-j,j)}(A)$ gives a correspondence $\varpi^{[j]}\colon\mathscr{M}^{2r}\to\mathscr{M}^{(2r-j,j)}$ in $\Corr^0_{\widetilde{X}_m}(W_{k,m})$ which induces a map $\varpi^{[j]}\colon\mathscr{M}_{\et}^{2r}\to\mathscr{M}^{(2r-j,j)}_{\et}$ under the étale realization of motives and therefore a pushforward map in the étale cohomology
\begin{equation}\label{eq.jComponentMap}
\varpi^{[j]}\colon H^2_{\et}\left(\widetilde{X}_m\otimes_\Q F_{cp^n},\mathscr{M}_{\et}^{2r}(1)\right)\longrightarrow H^2_{\et}\left(\widetilde{X}_m\otimes_\Q F_{cp^n},\mathscr{M}^{(2r-j,j)}_{\et}(1)\right),
\end{equation}

\begin{definition}
The image of $z_{\phi,m}^{[k]}$ under \eqref{eq.jComponentMap} is a class $z_{\phi,m}^{[k,j]}$, the \emph{$j$-component} of the generalized Heegner class $z_{\phi,m}^{[k]}$.
\end{definition}

\subsection{Basis vectors from Heegner points}\label{subsec.BasisVectors}
Let $i\colon\mathbf{H}\hookrightarrow\mathbf{G}$ (resp. $i_{cp^n}\colon \mathbf{H}\hookrightarrow\mathbf{G})$ be the embedding induced by $\iota_K^B$ (resp. $i_{cp^n}$, in the notation of \S\ref{subsec.CMpoints}).

Let $\sigma_F\colon K\hookrightarrow F$ be an embedding of $K$ into a field $F$, and let $\bar\sigma_F\colon K\hookrightarrow F$ denote its conjugate, in the sense that $\bar{\sigma}_F(\vartheta)\defeq\sigma_F(\bar{\vartheta})$. Associated to each Heegner point $x_{cp^n,m}(1)=[(\iota_K^B,\xi_{cp^n})]\in \widetilde{X}_m$, we define vectors in $F^2$ \[v_{cp^n}\defeq b_{cp^n}^{-1}\binom{\sigma_F(\vartheta)}{1}\hspace{10pt}\text{and}\hspace{10pt}\bar{v}_{cp^n}\defeq b_{cp^n}^{-1}\binom{\bar{\sigma}_F({\vartheta})}{1}\] after the decomposition $\xi_{cp^n}=b_{cp^n}u_{cp^n}$ with $b_{cp^n}\in B^\times$ and $u_{cp^n}\in \widetilde{U}_m$ as in \S\ref{subsec.CMpoints}). By definition, $v_{cp^n}$ is an eigenvector for the action of $\boldH(K)$ through $i_{cp^n}$, in the sense that $i_{cp^n}(x)v_{cp^n}=\sigma_F(x)v_{cp^n}$ for any $x\in\boldH(K)=K^\times$, where $i_{cp^n}(x)\in\boldG(K)=B^\times$ is understood as an element in $\M_2(F)$ through $\iota_B^{\M_2(K)}\otimes_K F$. By seeing $eV_{\boldG}(F)\cong F^2$ as an $F$-representation of $\boldG$ as in \S\ref{subsec.RepresentationsAssociatedToMotives}, the set of vectors \[v_{cp^n}^{[k,j]}\defeq v_{cp^n}^{\cdot(2r-j)}\cdot \bar{v}_{cp^n}^{\cdot j}\]for varying $0\le j\le 2r$ is a base for the $F$-representation $\Sym^{2r}(F^2)$ of $\boldG$.

Dually, set $\vartheta^\ast\defeq -1/\bar{\vartheta}$ and define the vectors in $(F^2)^\vee$ \[e_{cp^n}\defeq  b_{cp^n}^\mathsf{T}\binom{\sigma_F(\vartheta^{\ast})}{1}\hspace{10pt}\text{and}\hspace{10pt}\bar{e}_{cp^n}\defeq b_{cp^n}^\mathsf{T}\binom{\bar{\sigma}_F(\vartheta^\ast)}{1}.\] Similarly, $e_{cp^n}$ is an eigenvector for the ``dual'' action of $\boldH(K)$, $(i_{cp^n}(x)^{-1})^\mathsf{T}e_{cp^n}=\sigma_F(x)^{-1}e_{cp^n}$ for any $x\in\boldH(K)=K^\times$, being dual in the sense that it corresponds to the dual $\sigma_F^{-1}$ of the standard representation $\sigma_F\colon K^\times\to V_{\boldH}(F)=F^\times$. The set of vectors \[e_{cp^n}^{[k,j]}\defeq e_{cp^n}^{\odot(2r-j)}\odot \bar{e}_{cp^n}^{\odot j}\]for varying $0\le j\le 2r$ is a base for the $F$-representation $(\Sym^{2r}(F^2))^\vee\cong\TSym^{2r}\left((F^2)^\vee\right)$.

\begin{lemma}\label{lem.vartheta} 
If $p$ is split in $K$, then $e_{cp^n}=u\cdot\binom{p^n\sigma_F(\delta)}{1}$ and $\bar{e}_{cp^n}=u\cdot\binom{0}{1}$, with $u\in\GL_2(\mathcal{O}_F)$. If $p$ is not split in $K$, then $e_{cp^n}=u\cdot\binom{p^n\sigma_F(\bar{\vartheta})}{1}$ and $\bar{e}_{cp^n}=u\cdot\binom{p^n\sigma_F(\vartheta)}{1}$, with $u\in\GL_2(\mathcal{O}_F)$. In every case, $e_{cp^n}-\bar{e}_{cp^n}=u\cdot\binom{\pm p^n\sigma_F(\delta)}{0}\in p^n(\mathcal{O}_F^\vee)^2$.
\end{lemma}
\begin{proof}
The question is local at $p$. Recall the decomposition $\xi_{cp^n}=b_{cp^n}u_{cp^n}$ from \S\ref{subsec.CMpoints}, so $b_{cp^n}=b_0\cdot\big(\begin{smallmatrix}p^n & 0 \\ 0 & 1\end{smallmatrix}\big)\cdot u_{cp^n}^{-1}$, with $b_0=\big(\begin{smallmatrix}\bigstar & -1 \\ 1 & 0\end{smallmatrix}\big)$ for some $\bigstar\in K$ depending on whether $p$ splits or not. Therefore \[b_0^\mathsf{T}\cdot\binom{\sigma_F(\vartheta^\ast)}{1}=\dfrac{\bigstar\cdot\vartheta^\ast+1}{-\vartheta^\ast}=-\bigstar+\bar{\vartheta}\hspace{5pt}\text{and}\hspace{5pt}b_0^\mathsf{T}\cdot\binom{\sigma_F(\bar{\vartheta}^\ast)}{1}=\dfrac{\bigstar\cdot\bar{\vartheta}^\ast+1}{-\bar{\vartheta}^\ast}=-\bigstar+{\vartheta}.\]
If $p$ is split, then $\bigstar=\sigma_F(\vartheta)$, otherwise $\bigstar=0$. In every case, we take $u=(u_{cp^n}^{-1})^\mathsf{T}$.
\end{proof}

The representation $i^\ast\TSym^{2r}((F^2)^\vee)$ decomposes as a sum of $2r+1$ one-dimensional representations of the form $\sigma_F^{i_1}\otimes\bar\sigma_F^{i_2}$ with $i_1+i_2=2r$, among which we have the summand $\sigma_F^{2r-j}\otimes\bar\sigma_F^j$, for which $i^\ast e^{[k,j]}_{cp^n}$ is a base vector. Then $e^{[k,j]}_{cp^n}\otimes i^\ast e^{[k,j]}_{cp^n}$ 
defines an element in the $F$-representation 
\[V^{(2r-j,j)}(F)=\TSym^{2r}((F^2)^\vee)\boxtimes(\sigma_F^{(2r-j)}\otimes\bar\sigma_F^{j})\]
of $\mathbf{G}\times\mathbf{H}$, 
which is invariant under the diagonal action of $K$ by $i_{cp^n}\times\id$. To simplify the notation, we write just $e^{[k,j]}_{cp^n}$ for all instances of this vector.

\subsection{From basis vectors to generalized Heegner classes}\label{subsec.BasisVectorsToGHCs}
Since $S_{cp^n}$ is a model over $K$ of the Shimura variety $\Sh(\boldH,\{\ast\})=K^\times\backslash\A_{K,\mathrm{fin}}^\times/U_{cp^n,\mathfrak{N}^+}$, base changing to $F_{cp^n}$ gives an isomorphism
\begin{equation}\label{eq.Varrho}
\varrho\colon ((X_m\otimes_\Q K)\times S_{cp^n})\otimes_KF_{cp^n}\stackrel{\sim}{\longrightarrow} \widetilde{X}_m\otimes_K F_{cp^n}.
\end{equation}
To simplify the notation, write $X_m\times S_{cp^n}\defeq (X_m\otimes_\Q K)\times S_{cp^n}$.

In order to associate the basis vectors from the previous subsection to the generalized Heegner classes defined in \S\ref{subsec.GeneralizedHeegnerClasses}, start by considering $e^{[k,j]}_{cp^n}$, with the same symbol, as an element in the twisted representation $V^{(2r-j,j)}_{F_{cp^n}}(1)$. Since this representation is finite-dimensional,
\[e^{[k,j]}_{cp^n}\in\Hom_{F_{cp^n}}\left(V^{(2r-j,j)}_{F_{cp^n}}(1)^\vee,F_{cp^n}\right)\subseteq\Hom_{F_{cp^n}}\left(\mathscr{M}^{2r}_{\et}(1)^\vee,\mathds{1}\right),\] where the inclusion comes from $\Anc_{\widetilde{X}_m\times S_{cp^n}/F_{cp^n}}$ and we further use $\varrho$ to see this morphism defined in $\CHM_{F_{cp^n}}(\widetilde{X}_m)$; here, $\mathds{1}=(\widetilde{X}_m,\id,0)$ denotes the trivial motive. In view of \S\ref{subsec.Lieberman} and \S\ref{subsec.jComponents}, we have the explicit description \[\mathscr{M}_{\et}^{2r}(1)^\vee=(W_{k,m}\otimes_H F_{cp^n},\varpi^{[j]}\circ\epsilon_W,2r+1)^\vee=(W_{k,m}\otimes_H F_{cp^n},(\varpi^{[j]}\circ\epsilon_W)^\mathsf{T},2r),\] so, by the definition of the morphisms in the category of Chow motives as correspondences,
\[e_{cp^n}^{[k,j]}\in\Corr^{-2r}_{\widetilde{X}_m}(W_{k,m}\otimes_H F_{cp^n},\widetilde{X}_m\otimes_K F_{cp^n})\circ(\varpi^{[j]}\circ\epsilon_W)^\mathsf{T}\cong\varpi^{[j]}\circ\epsilon_W\CH^{2r+1}(W_{k,m}\otimes_H F_{cp^n})_{\Q},\] where we further use the homotopy invariance property of relative Chow groups. As a cycle (in the notation of \S\ref{subsec.CMpoints}), $e_{cp^n}^{[k,j]}$ is the graph of the map $\varphi\colon A^r\times_{\widetilde{X}_m}\widetilde{\mathcal{A}}_m^r\to\widetilde{X}_m$ given by $(P^r,\phi_{cp^n}(P)^r)\mapsto\vartheta_{cp^n}$, where $(\langle P\rangle,P)$ is the level $\Gamma_1(p^m)$ structure attached to $A$ and $(\langle\phi_{cp^n}(P)\rangle,\phi_{cp^n}(P))$ is the correspondent level $\Gamma_1(p^m)$ structure attached to the section of $\widetilde{\mathcal{A}}_m$ lying over $\vartheta_{cp^n}$, that is, $\phi_{cp^n}(A)=A_{\vartheta_{cp^n}}$. In other words, $e_{cp^n}^{[k,j]}$ corresponds as a cycle to \[\varpi^{[j]}\circ\epsilon_W\operatorname{graph}(\phi_{cp^n})^r=\varpi^{[j]}(\Delta_{cp^n,m}^{[k]}),\] the $j$-component of the generalized Heegner cycle attached to $\phi_{cp^n}$. As an element of the group $\varpi^{[j]}\epsilon_W\CH^{2r+1}(W_{k,m}\otimes_H F_{cp^n})_{\Q}$, $\varpi^{[j]}(\Delta_{cp^n,m}^{[k]})$ maps under \eqref{eq.LTfullmap} to $z_{cp^n,m}^{[k,j]}$.

\subsection{Gysin map}\label{subsec.GysinMap}
Each of the embeddings $i_{cp^n}$ induces an embedding of $\delta_{cp^n}\colon \mathbf{H}\hookrightarrow\mathbf{G}\times\mathbf{H}$, defined by $\delta_{cp^n}=(i_{cp^n},\id)$, where $\id$ is the identity map. This map induces a pullback at the level of Shimura varieties defined over $K$ given by \[\delta_{cp^n}^\ast\colon\widetilde{X}_m\times S_{cp^n}\longrightarrow S_{cp^n},\] which further induce maps at the level of representations, Chow motives and lisse étale sheaves, the former two being compatible via Ancona functors and the latter two being compatible via étale realization functors (using \cite[Theorem 9.7]{Torzewski}):
\begin{center}
\begin{tikzcd}[column sep = huge]
\Rep_{K}(\boldG\times\boldH)\arrow{r}{\Anc_{\widetilde{X}_m\times S_{cp^n}/K}}\arrow{d}{\delta_{cp^n}^\ast}&\CHM_K(\widetilde{X}_m\times S_{cp^n})\arrow{r}{\Et_{\widetilde{X}_m\times S_{cp^n}/K}}\arrow{d}{\delta_{cp^n}^\ast}&\CHM_{\et,L}(\widetilde{X}_m\times S_{cp^n})\arrow{d}{\delta_{cp^n}^\ast}\\
\Rep_{K}(\boldH)\arrow{r}{\Anc_{S_{cp^n}/K}}&\CHM_K(S_{cp^n})\arrow{r}{\Et_{S_{cp^n}/K}}&\CHM_{\et,L}(S_{cp^n})
\end{tikzcd}
\end{center}

The embedding $\delta_{cp^n}$ also induces a pushforward at the level of Shimura varieties: \[\delta_{cp^n,\ast}\colon S_{cp^n}\longmono \widetilde{X}_m\times S_{cp^n},\] which in turn induces a Gysin (``wrong way'') map at the cohomological level, \[H^0_{\et}\left(S_{cp^n},\delta_{cp^n}^\ast(\mathscr{M}_{\et}^{(2r-j,j)})\right)\longrightarrow H^2_{\et}\left({X}_m\times S_{cp^n},\mathscr{M}_{\et}^{(2r-j,j)}(1)\right),\]which happens to have precisely the index and twist to land in \[H^2_{\et}(\widetilde{X}_m\otimes_\Q F_{cp^n},\mathscr{M}^{2r}_{\et}(1))\stackrel{\eqref{eq.Varrho}}{\cong} H^2_{\et}(\widetilde{X}_m\times S_{cp^n},\mathscr{M}^{2r}_{\et}(1)).\]

As observed in \cite[p. 132]{JetchevLoefflerZerbes} (see also \cite[Definition 3.1.2, \S 5.2]{KingsLoefflerZerbes.MF}), the Gysin map gives an equivalent and more practical way to express the relation between $e_{cp^n}^{[k,j]}$ to $z_{cp^n,m}^{[k,j]}$. Compose the two maps above into
\begin{equation}\label{eq.Gysin}
\delta_{cp^n,\ast}\colon H^0_{\et}\left(S_{cp^n},\delta_{cp^n}^\ast(\mathscr{M}_{\et}^{(2r-j,j)})\right)\longrightarrow 
H^2_{\et}\left({X}_m\otimes_KF_{cp^n},\mathscr{M}_{\et}^{(2r-j,j)}(1)\right),
\end{equation}
which we will refer to as the \textit{Gysin map} induced by $\delta_{cp^n}$. In view of \S\ref{subsec.BasisVectorsToGHCs}, we summarize the relation between basis vectors and generalized Heegner classes as follows:

\begin{proposition}\label{prop.Gysin}
The image under $\delta_{cp^n,\ast}$ of the basis vector $e_{cp^n}^{[k,j]}$ is $z_{cp^n,m}^{[k,j]}$.
\end{proposition}

\subsection{The Abel--Jacobi map}\label{subsec.AbelJacobi}
The degeneration at page 2 of the Hochschild--Lyndon--Serre spectral sequence \cite[\S 1.2]{Nekovar.Banff} yields an isomorphism
\begin{equation}\label{eq.degenerationHLS}
H^2_{\et}\left({X}_m\otimes_\Q F_{cp^n},\mathscr{M}^{(2r-j,j)}_{\et}(1)\right)\stackrel{\sim}{\longrightarrow}H^1\left(F_{cp^n},H^1_{\et}({X}_m\otimes_\Q\overline{\Q},\mathscr{M}^{(2r-j,j)}_{\et}(1))\right).\end{equation}
The embeddings $\sigma_L,\bar{\sigma}_L\colon K\hookrightarrow L$, seen as Galois characters, induce $p$-adic characters
\[\sigma_{\et},\bar\sigma_{\et}\colon \Gal(K^\mathrm{ab}/F_{\mathfrak{N}^+})\longrightarrow L^\times\] defined by $x\mapsto \sigma_L^{-1}(x^{(p)})$ 
and $x\mapsto\bar\sigma_L^{-1}(x^{(p)})$, respectively, where $x^{(p)}$ is the $p$-component of $x$ as an element of $(1+\mathfrak{N}^+\widehat{\mathcal{O}}_K)^\times$, which is isomorphic to $\Gal(K^\mathrm{ab}/F_{\mathfrak{N^+}})$ under the geometrically normalized Artin reciprocity map $\rec_K$ (see Remark \ref{rem.Reciprocity}).

The right-hand side of \eqref{eq.degenerationHLS} can then be further rewritten as
\begin{equation}\label{eq.degenerationHLS2}
H^1\left(F_{cp^n},H^1_{\et}({X}_m\otimes_\Q\overline{\Q},\mathscr{M}^{(2r-j,j)}_{\et}(1))\right)\cong H^1\left(F_{cp^n},H^1_{\et}({X}_m\otimes_\Q\overline{\Q},\mathscr{V}^{2r}_{\et}(1)\otimes\sigma_{\et}^{2r-j}\bar{\sigma}_{\et}^j)\right).
\end{equation}

\begin{definition}
The composition of the maps \eqref{eq.LTfullmap}, \eqref{eq.jComponentMap}, \eqref{eq.degenerationHLS} and \eqref{eq.degenerationHLS2} is the \emph{$p$-adic Abel--Jacobi map} 
\begin{equation}\label{eq.AbelJacobi}
\Phi_m^{[k,j]}\colon \epsilon_{W}\CH^{k-1}(W_{k,m}\otimes_HF_{cp^n})_\Q\longrightarrow H^1\left(F_{cp^n},H^1_{\et}({X}_m\otimes_\Q\overline{\Q},\mathscr{V}^{2r}_{\et}(1)\otimes\sigma_{\et}^{2r-j}\bar{\sigma}^j_{\et})\right).
\end{equation}
In particular, $z_{\phi,m}^{[k,j]}=\Phi_m^{[k,j]}(\Delta_{\phi,m}^{[k]}).$ 
\end{definition}

\section{Modular forms}

Let $R$ be a $p$-adic ring, that is, a $\Z_{p}$-algebra which is complete and Hausdorff with respect to the $p$-adic topology, so that $R \cong \varprojlim_n R/p^nR$. In this section, if $S$ is a $\Z_p$-scheme, $S_R\defeq S\otimes_{\Z_p}R$ indicates the coefficient extension to $R$.

\subsection{$p$-adic modular forms}
Quaternionic $p$-adic modular forms were defined for general weight in \cite{Brasca.QuaternionicMF}, see also \cite{Kassaei.Thesis}.

A \textit{$p$-adic modular form of tame level $N^{+}$ on $B$} is a global section of $\widehat{\mathrm{Ig}}_R$ (see \S\ref{subsec.IgusaTowers}). Let $V_p(N^+,R)$ denote the space of quaternionic modular forms of tame level $N^+$. In terms of the layers of the Igusa tower $\widehat{\mathrm{Ig}}$, 
\[V_p(N^+,R)=H^0(\widehat{\mathrm{Ig}},\mathcal{O}_{\widehat{\mathrm{Ig}}})\otimes_{\Z_p}R\cong\varprojlim_n\varinjlim_m H^0(\mathrm{Ig}_{m,n},\mathcal{O}_{\mathrm{Ig}_{m,n}})\otimes_{\Z_p}R,\]
where, as usual, $\mathcal{O}_S$ denotes the structure sheaf of a scheme $S$. In view of Proposition \ref{prop.ModuliIgusa}, a quaternionic $p$-adic modular form as above is a rule $\mathcal{F}$ that, for each $n,m\ge1$, takes a quadruple $(A,\iota,\alpha, \beta)$ consisting of a QM abelian surface $(A,\iota)$ over an $R$-algebra $\mathcal{R}$ with level $U_1(N^+)$ structure $\alpha$ and an arithmetic trivialization $\beta_m$ of $A[p^m]$ over $\mathcal{R}/p^n\mathcal{R}$, and assigns a value $\mathcal{F}(A,\iota,\alpha, \beta)\in \mathcal{R}/p^n \mathcal{R}$ which is compatible with respect to the canonical maps used to compute the direct and inverse limit, depends only on the isomorphism class of the quadruplet and is compatible under base change given by continuous morphisms between $R$-algebras. 

An element $u\in \Gamma\defeq 1+p\Z_p$ acts on $\boldsymbol{\mu}_{p^\infty}$ by left multiplication. Precomposing an arithmetic trivialization $\beta\in\mathcal{P}_{m,n}$ of $\mathcal{A}^\ord_n[p^\infty]$ with that action $u$ gives a new arithmetic trivialization $u\cdot\beta=(u\times u)\circ\beta\in\mathcal{P}_{m,n}$ which specializes to an arithmetic trivialization of $A[p^\infty]$ for any QM abelian surface $A$ over $R$. This gives a $R[\![\Gamma]\!]$-module structure over $V_p(N^+,R)$ denoted by $\mathcal{F}\mapsto \mathcal{F}|\langle \lambda\rangle$ for $\mathcal{F}\in V_p(N^+,R)$ and $\lambda\in\Lambda_R$, defined for $u\in\Gamma$ by 
\[ (\mathcal{F}|\langle u\rangle)(A,\iota,\alpha, \beta )=\mathcal{F} (A,\iota,\alpha, u\cdot \beta)\]
and for any other $\lambda\in R[\![\Gamma]\!]$ by extending $R$-linearly.
 
\begin{definition}
Let $\psi\colon \Gamma\rightarrow \boldsymbol{\mu}_{p^{\infty}}(\overline\Q_p)$ be a finite order character and $k \in \Z_p$ a $p$-adic integer. We say that a $p$-adic modular form $\mathcal{F}$ is of \emph{signature} $(k, \psi)$ if for every $u\in\Gamma$, we have $\mathcal{F}|\langle u\rangle=u^{k}\psi(u)\mathcal{F}$. 
\end{definition}

\subsection{Geometric modular forms}

As explained in \cite{Kassaei.Thesis}, the classical quaternionic $p$-adic modular forms are those coming from geometry. Let $\pi\colon A\to S$ be a QM abelian surface over $R$ and $\Omega_{A/R}$ denote the bundle of relative differentials. Then quaternionic multiplication by $\mathcal{O}_B$, after extension of scalars to $\Z_p$, gives an action on $\pi_* \Omega_{A/R}$. In particular, the local idempotent $e_p$ from \S\ref{subsubsec.Drinfeld} (which coincides with the global idempotent $e$ from \S\ref{subsec.Idempotents} if $p$ is split) acts over $\pi_* \Omega_{A/R}$ allowing us to define the invertible sheaf $\underline{\omega}_{A/R}\defeq e_p\pi_* \Omega_{A/R}$. 

A \emph{test object} over an $R$-algebra $\mathcal{R}$ is a quintuplet $T=(A, \iota,\alpha,(H,P),\omega)$ consisting of a QM abelian surface $(A,\iota)$ over $\mathcal{R}$, a level $U_1(N^+)$ structure $\alpha$ on $A$, a level $\Gamma_1(p^m)$ structure $(H,P)$ on $A$ and a non-vanishing global section of the line bundle $\underline{\omega}_{A/\mathcal{R}}$ of relative differentials. Two test objects are \emph{isomorphic} if there is an isomorphism of QM abelian surfaces which induces isomorphisms of level $U_1(N^+)$ and $\Gamma_1(p^m)$ structures and pulls back the generator of the differentials of $A'$ to that of $A$. 

\begin{definition}
A \emph{$R$-valued geometric modular form on $\widetilde{X}_m$} is a rule $\mathcal{F}$ that assigns to each isomorphism class of test objects $T=(A,\iota, \alpha,(H,P),\omega)$ over an $R$-algebra $\mathcal{R}$ a value $\mathcal{F}(T)\in\mathcal{R}$ which 
\begin{itemize} 
\item is compatible with base changes: if $\varphi\colon\mathcal{R}\rightarrow\mathcal{R}'$ is a morphism of $R$-algebras, \[\mathcal{F}(A' ,\iota',\alpha',(H',P'),\omega') = 
\varphi\big(\mathcal{F}(A,\iota, \alpha,(H,P),\varphi^*(\omega')\big),\] where
$A'=A \otimes_{R,\varphi} \mathcal{R}'$, and $\iota'$, $\alpha'$ and $(H',P')$ are obtained by base change from $\iota$, $\alpha$ and $(H,P)$, respectively;

\item satisfies a weight condition: for any $\lambda\in \mathcal{R}^\times$ \[\mathcal{F}(A,\iota,\alpha,(H,P),\lambda\omega) =
\lambda^{-k}\mathcal{F}(A,\iota, \alpha,(H,P),\omega).\]
\end{itemize}
We denote $M_k(N^+,p^m,R)$ the $R$-module of $R$-valued weight $k$ modular forms on $\widetilde{X}_m$.
\end{definition}

An element $u\in(\Z/p^m\Z)^\times$ acts on a test object $T=(A,\iota, \alpha,(H,P),\omega)$ by left multiplication on $P$ and trivially everywhere else. This action extends to $M_k(N^+,p^m,R)$ by setting \[(\mathcal{F}|\langle u\rangle)(A,\iota, \alpha,(H,P),\omega)=\mathcal{F}(A,\iota, \alpha,(H,u\cdot P),\omega).\]

\begin{definition}
Let $\psi\colon(\Z/p^m\Z)^\times\rightarrow R$ be a finite order character. We say that a modular form 
$\mathcal{F}\in M_k(N^+,p^m,R)$ has \emph{character $\psi$} if $\mathcal{F}|\langle u\rangle=\psi(u)\mathcal{F}$ for all 
$u\in (\Z/p^m\Z)^\times$. 
\end{definition}

Denote by $M_k(N^+p^m,R)$ the $R$-submodule of $M_k(N^+,p^m,R)$ of $R$-valued weight $k$ modular forms on $X_m$ consisting of modular forms with trivial character.

\subsection{From geometric to classical $p$-adic modular forms}\label{subsec.Cartier}
In order to define a $p$-adic modular form from a geometric modular form, one needs to establish a relation between the test objects above with the quadruplets from before, which boils down to defining a level $\Gamma_1(p^m)$ structure $(H,P)$ after an arithmetic trivialization $\beta$ for each $(A,\iota,\alpha)$. In view of \S\ref{subsec.QMAbelianSurfaces}, there is a unique principal polarization $\lambda\colon A\stackrel{\sim}{\to}A^\vee$ whose correspondent Rosati involution of $\End(A)$ coincides with the restriction to $\mathcal{O}_B$ of the involution $\dagger$ of $B$ defined in \S\ref{subsec.QMAbelianSurfaces}.

As explained in \cite[\S 3.1]{Magrone}, an arithmetic trivialization  
$\beta$ determines by Cartier duality a point $x_{\beta}^\vee$ in 
$e^\dagger{\Ta_p(A^\vee)({\overline{\F}_p})}$, where $\Ta_p$ denotes the $p$-adic Tate module. By \cite[p. 150]{Katz}, we have an isomorphism of formal groups over $\overline{\F}_p$
\begin{align*}
\widehat{A}\cong\Hom_{\Z_p}(\Ta_p(A^\vee)(\overline{\F}_p),\widehat{\G}_m) &\implies e\widehat{A}\cong\Hom_{\Z_p}(\Ta_p((eA)^\vee)(\overline{\F}_p),\widehat{\G}_m)\\
&\implies e^\dagger{\Ta_p(A^\vee)({\overline{\F}_p})}\cong\Hom_{\Z_p}(e\widehat{A},\widehat{\G}_m).
\end{align*}
Under the isomorphism above $x_\beta^\vee$ corresponds to some $\varphi_{\beta}\in \Hom_{\Z_p}(e\widehat{A},\widehat{\G}_m)$. Conversely, one such morphism corresponds to a point in $e^\dagger{\Ta_p(A^\vee)({\overline{\F}_p})}$ and, again by Cartier duality, to an arithmetic trivialization on $A$. The pullback of the standard differential $\mathrm{d}T/T$ of $\widehat{\G}_m$ under $\varphi_\beta$, denoted $\omega_{\beta} = \varphi_{\beta }^*(\mathrm{d}T/T)$, defines a differential in $\underline{\omega}_{A/R}$.
 
Fix a generator of $\boldsymbol{\mu}_{p^\infty}$, which induces a generator $\zeta_{p^m}$ of $\boldsymbol{\mu}_{p^m}$ for all $m$. Then $\beta (\zeta_{p^m})$ gives a point $P$ in $A[p^m]$ of exact order $p^m$ and, taking $H$ to be the subgroup scheme generated by $P$, $(H,P)$ gives a level $\Gamma_1(p^m)$ structure on $(A,\iota,\alpha)$. Define \[\widehat{\mathcal{F}}(A,\iota,\alpha, \beta  ) \defeq \mathcal{F}(A,\iota,\alpha, (H,P),\omega_{\beta  }).\] The map $\mathcal{F}\mapsto \widehat{\mathcal{F}}$ is compatible with base change, only depends on the isomorphism class of $(A,\iota,\alpha, \beta )$ and it is compatible with the maps occurring in the direct and inverse limit in the definition of $p$-adic modular forms; thus $\mathcal{F}\mapsto \widehat{\mathcal{F}}$ establishes a map 
\begin{equation}\label{eq.ClassicMF}
M_k(N^+,p^m,R)\longrightarrow V_p(N^+,R)
\end{equation}
sending a geometric modular form $\mathcal{F}$ is of weight $k$ and character $\psi$ to a $p$-adic modular form $\widehat{\mathcal{F}}$ of signature $(k,\psi)$.

\begin{definition}
A $p$-adic modular form is said to be \textit{classic} if it is in the image of \eqref{eq.ClassicMF}.
\end{definition}

\subsection{Hecke operators}\label{subsec.HeckeOperators}
We follow \cite[\S II.2, II.3]{Gouvea} and \cite[\S 3.6]{Brooks}. If $(A,\iota,\alpha,\beta)$ is a quadruplet like in the previous subsections, since $A$ has ordinary reduction at $p$ (see Remark \ref{rem.OrdinaryReduction}), the $p$-torsion subgroup $H$ defined after $\beta$ in the previous subsection is the \textit{canonical subgroup}, that is, the one (and only one) which reduces modulo $p$ to the kernel of the Frobenius morphism (see also \cite[Theorem 1.11]{Kassaei.Thesis}). Denoting by $\phi_0\colon A\twoheadrightarrow A_0\defeq A/H$ the canonical projection, we can define another quadruplet $(A_0,\iota_0,\alpha_0,\beta_0)$ as follows:
\begin{itemize}
\item $\iota_0$ is the pullback of $\iota$ by $\phi_0$;

\item Since $\phi_0$ has degree $p$ which is coprime with $N^+$, $\phi_0$ reduces to an isomorphism in the $N^+$-torsion, so $\alpha_0\defeq\phi_0\circ\alpha$ is a well defined level $U_1(N^+)$ structure on $A_0$;

\item As before, $\beta$ corresponds to a morphism $\varphi_\beta\in\Hom_{\Z_p}(e\widehat{A},\widehat{\G}_m)$. Since $\phi_0$ is étale, it induces an isomorphism on formal completions $\widehat\phi_0\colon\widehat{A}_0\rightarrow\widehat{A}$. The morphism $\varphi\defeq\varphi_\beta\circ(e\widehat{\phi}_0)\in\Hom_{\Z_p}(e\widehat{A}_0,\widehat{\G}_m)$ induces an arithmetic trivialization $\beta_0$ on $A_0$.
\end{itemize}

\begin{definition}
The operator $V_p\colon V_p(N^+,R)\rightarrow V_p(N^+,R)$ is defined
for any $p$-adic modular form $\mathcal{F}$ over $R$ by the equation
$V_p\cdot\mathcal{F}(A,\iota,\alpha,\beta)\defeq \mathcal{F}(A_0,\iota_0,\alpha_0,\beta_0)$.
\end{definition}

Let $H_i$, for $1\le i\le p$, denote all the other $p$-torsion subgroups of $A$. An equivalent construction as the one performed for the canonical subgroup gives quadruplets $(A_i,\iota_i,\alpha_i,\beta_i)$ where $\phi\colon A\twoheadrightarrow A_i\defeq A/H_i$ is the quotient of $A$ by $H_i$.

\begin{definition}
The operator $U_p\colon V_p(N^+,R) \rightarrow V_p(N^+,R)$ is defined for any $p$-adic modular form $\mathcal{F}$ over $R$ by the equation $\displaystyle U_p\cdot\mathcal{F}\defeq\frac{1}{p} \sum_{i=1}^p\mathcal{F}(A_i,\iota_i,\alpha_i,\beta_{i})$.
\end{definition}

\begin{remark}
It follows that $V_p\circ U_p$ is just the identity map $\id$, which is not the case for $U_p\circ V_p$. In the elliptic modular forms case, the second composition kills all terms of index prime to $p$ in the Fourier expansion. Motivated by this classical case, we define the \emph{$p$-depletion} of a $p$-adic modular form $\mathcal{F}$ in $V_p(N^+,R)$ to be
$\mathcal{F}^{[p]} \defeq (\id - U_pV_p)\cdot \mathcal{F}$.
\end{remark}

If $\ell$ is any other prime, the same construction as above gives for each of the $\ell$-torsion groups $H_0,\dots,H_\ell$ and their quotients $\phi_i\colon A\twoheadrightarrow A_i\defeq A/H_i$ a quadruplet $(A_i,\iota_i,\alpha_i,\beta_i)$.

\begin{definition}
The operator $T_\ell\colon V_p(N^+,R) \rightarrow V_p(N^+,R)$ is defined for any $p$-adic modular form $\mathcal{F}$ over $R$ by the equation $\displaystyle T_\ell\cdot \mathcal{F}\defeq\frac{1}{\ell} \sum_{i=0}^\ell\mathcal{F}(A_i,\iota_i,\alpha_i,\beta_{i})$.
\end{definition}

\subsection{Jacquet--Langlands correspondence}

Fix an isomorphism between $\C$ and $\C_p$, the $p$-adic completion of $\overline{\Q}_p$, compatible with the embedding $\iota_p\colon\overline{\Q}\hookrightarrow\overline{\Q}_p$, and denote by $\mathcal{O}_{\C_p}$ the valuation ring in $\C_p$ arising from this choice. The reason why classical quaternionic modular forms receive such name is that they come from actual classical modular forms via the Jacquet--Langlands correspondence:

\begin{proposition}[{\cite[Theorem 16.1]{JacquetLanglands}}]\label{prop.JacquetLanglands}
There is an isomorphism \[\mathrm{JL}\colon S_k(\Gamma_1(Np^m),\C)^{N^{-}\text{-new}}\stackrel{\sim}{\longrightarrow}M_k(N^+,p^m,\C_p),\] where $S_k(\Gamma_1(Np^m),\C)^{N^{-}\text{-new}}$ denotes the $\C$-vector space of classical cuspidal modular forms of level $\Gamma_1(Np^m)$ and that are new at all primes dividing $N^-$. If $f\in S_k(\Gamma_1(Np^m),\C)^{N^{-}\text{-new}}$ has character $\psi$ then so does $\mathcal{F}=\operatorname{JL}(f)$, and the eigenvalues of $f$ and $\mathcal{F}$ with respect to $U_p$ and all Hecke operators $T_\ell$ away from $N=N^+N^-$ coincide.
\end{proposition}

\subsection{Galois representations}\label{subsec.GaloisRepresentations}
Let $L=\operatorname{Frac}(R)$. Consider $\mathcal{F}\in M_k(N^+,p^m,R)$ an eigenform for all Hecke operators and let $L_\mathcal{F}\subseteq\overline{\Q}_p$ be its Hecke field. Here we take the \textit{$p$-adic Galois representation} of $\mathcal{F}$ to be the Deligne $p$-adic representation \[V_\mathcal{F}\subseteq 
H^{1}_{\et}\left(\widetilde{X}_m\otimes_{\Q}\overline{\Q},(\mathscr{V}_{\et}^k)^\vee\right)\otimes_{\Q_p}L_\mathcal{F}\] characterized by requiring that the trace of the geometric Frobenius element at a prime $\ell\nmid Np$ is equal to $a_\ell(\mathcal{F})$, the Hecke eigenvalue of $T_\ell$. The dual, or contragredient, Galois representation $V_\mathcal{F}^*$ of $V_\mathcal{F}$ is then the maximal quotient  
\[H^{1}_{\et}\left(\widetilde{X}_m\otimes_{\Q}\overline{\Q},\mathscr{V}^{k}_{\et}(1)\right)\otimes_{\Q_p}L_\mathcal{F}\longepi
V_\mathcal{F}^*,\] 
where dual Hecke operators $T_\ell'$ act by $a_\ell(\mathcal{F})$, for all $\ell\nmid Np$, and $U_p$ acts by $a_p(\mathcal{F})$; see \cite[\S 2.8]{KingsLoefflerZerbes.ERL}. 

Let $f$ be a classical eigenform such that $\mathcal{F}=\operatorname{JL}(f)$ and denote by $V_f$ its classical $p$-adic Galois representation with coefficients in $L_\mathcal{F}\cong K_f\otimes_\Q\Q_\ell$. Since the eigenvalues with respect to $T_\ell$ of $f$ and $\mathcal{F}$ coincide in all but finitely many primes $\ell$, their Hecke polynomial (that is, the characteristic polynomial of the image of the Frobenius element) at all but finitely many primes $\ell$ coincide. Since $V_f$ and $V_\mathcal{F}$ are irreducible, they are then isomorphic.

\subsection{$p$-stabilization of modular forms}\label{subsec.Stabilization}

Let $\mathcal{F}=\operatorname{JL}(f)\in M_k(N^+,1,R)$ be a quaternionic modular form over $\widetilde{X}_0$. The \textit{$p$-stabilization} of $\mathcal{F}$ with respect to a root $\upalpha$ of its Hecke polynomial is simply the image under $\mathrm{JL}$ of the $p$-stabilization of $f$ with respect to $\upalpha$.

Consider the canonical degeneration maps
\begin{itemize}
\item $\mathrm{pr}_1\colon \widetilde{X}_1\to \widetilde{X}_0$, which takes a quadruplet $(A,\iota,\alpha,\beta)$ and maps into $(A,\iota,\alpha_p,\beta_p)$, where $\alpha_p$ is (the equivalence class of) $\psi_p\circ\alpha$ with $\psi_p$ the map $(x,y)\mapsto (px,py)$; and $\beta_p=\beta\circ\varphi_p$ with $\varphi_p\colon\boldsymbol{\mu}_{p}\to\boldsymbol{\mu}_{1}$ being the map $x\mapsto x^p$, thus trivial;

\item $\mathrm{pr}_2\colon \widetilde{X}_1\to \widetilde{X}_0$, which takes a quadruplet $(A,\iota,\alpha,\beta)$ and outputs $(A_0,\iota_0,\alpha_0,\beta_0)$, the quadruplet associated to the canonical subgroup of $A$ as explained in \S\ref{subsec.HeckeOperators}.
\end{itemize}

\begin{proposition}\label{prop.ProjStabIsom}
The map
\begin{equation}\label{eq.ProjSlab}
(\mathrm{pr}_\upalpha)_*\defeq(\mathrm{pr}_1)_*-\frac{(\mathrm{pr}_2)_*}{\upalpha}\colon V_{\mathcal{F}_\upalpha}\longrightarrow V_{\mathcal{F}}
\end{equation}
is an isomorphism, where $(\mathrm{pr}_1)_*,(\mathrm{pr}_2)_*\colon V_{\mathcal{F}_\upalpha}\rightarrow V_{\mathcal{F}}$ are induced by the canonical degeneration maps $\mathrm{pr}_1,\mathrm{pr}_2\colon \widetilde{X}_1\to \widetilde{X}_0$.
\end{proposition}
\begin{proof}
One can use the Jacquet--Langlands correspondence to reduce it to the elliptic version of this statement, which is \cite[Proposition 7.3.1]{KingsLoefflerZerbes.ERL}.
\end{proof}

\subsection{Eigencurves}
Recall that the weight space $\mathscr{W}\defeq\Hom_{\mathrm{cont},\,\Z_p}(\Z_p^\times,\Z_p^\times)$ is the rigid analytic variety associated to the Iwasawa algebra $\Lambda\defeq\Z_p[\![\Z_p^\times]\!]$. For $N\in\Z_{\ge 2}$, we also define the level $N$ Iwasawa algebra to be $\Lambda_N\defeq\Z_p[\![(\Z/N\Z)^\times\times\Z_p^\times]\!]$, whose associated level $N$ weight space is $\mathscr{W}_N\defeq\Hom_{\mathrm{cont},\,\Z_p}((\Z/N\Z)^\times\times\Z_p^\times,\Z_p^\times)$, and there is a natural projection $\mathscr{W}_N\to\mathscr{W}$. We embed $\Z$ into those weight spaces by sending $k$ into the character $x\mapsto x^{k-2}$. This normalization is so that $k=2r+2$ maps into $2r$.

Let $\mathcal{E}(N)$ denote the Coleman--Mazur eigencurve of tame level $N$ defined in \cite{ColemanMazur}, a rigid analytic variety parametrizing finite slope overconvergent $p$-adic eigenforms of tame level $N$, and in particular, the $p$-stabilized eigenforms $f\in S_k(\Gamma_1(Np),\C)^{N^{-}\text{-new}}$ from Proposition \ref{prop.JacquetLanglands}. There exists a projection of the eigencurve $\mathcal{E}(N)$ into the level $N$ weight space $\mathscr{W}_N$.

Using Buzzard's eigenvariety machine from \cite{Buzzard.Eigen}, Brasca (in \cite{Brasca.Eigenvarieties}) has constructed the quaternionic analogue of the eigencurve, parametrizing cuspidal quaternionic modular forms over certain PEL Shimura varieties. Since the Shimura curves we consider here are PEL and compact, there are no cusps exist and therefore any modular form defined over them is automatically cuspidal, so Brasca's construction apply. We denote by $\mathscr{E}(N^{+})$ the Brasca--Buzzard eigencurve, whose $L$-rational points parametrize modular forms of tame level $N^+$. A general result by Chenevier, when adapted to our context, implies a $p$-adic Jacquet--Langlands correspondence between the elliptic and quaternionic eigencurves:

\begin{proposition}\label{prop.padicJacquetLanglands}
The new at $N^{-}$ locus of the Coleman--Mazur eigencurve, denoted $\mathcal{E}(N)^{N^{-}\text{-new}}$, and the Brasca--Buzzard eigencurve $\mathscr{E}(N^+)$ are isomorphic as rigid analytic varieties; see \cite[Théorème 3]{Chenevier}.
\end{proposition}

\subsection{Coleman families}\label{subsec.ColemanFamilies}
Coleman families of any positive slope, generalizing the ordinary (zero slope) families defined by Hida, were first defined in \cite{Coleman.Banach} for elliptic modular forms. We follow the exposition in \cite[\S 2.2]{NuccioOchiai} for the elliptic case.

Let $f$ be a classical level $N$ new normalized eigenform of even weight $k_0$, such that the roots of its Hecke polynomial, $\upalpha$ and $\upbeta$ are distinct and $v_p(\upalpha)<k_0-1$. These conditions are enough for us to define after $f_\upalpha$ a point in $\mathcal{E}(N)^{N^{-}\text{-new}}$ in which the projection to the weight space is smooth and étale (see \cite[\S 1.1]{Hansen} and \cite[\S 2.3]{Bellaiche}).

For an affinoid $\mathscr{U}\subseteq\mathscr{W}_N$, we write $\mathscr{A}_\mathscr{U}^\circ$ for the ring of power bounded rigid analytic functions defined over $\Q_p$ (see \cite[\S 2.1]{NuccioOchiai}).

\begin{proposition}\label{prop.ColemanFamilies}
Let $f$ be a $p$-stabilized modular form of slope $\lambda$ in the conditions above. Then there exists an affinoid $\mathscr{U}\subseteq\mathscr{W}_N$ and, for each $n\in\Z_{>0}$, $\mathbf{a}_n\in\mathscr{A}_{\mathscr{U}}^\circ\otimes_{\Q_p}L$ rigid analytic functions with coefficients in a finite extension $L$ of $\Q_p$ such that, for each $k\in\mathscr{U}\cap\Z$ with $\lambda<k-1$, the specialization at weight $k$ \[f_k\defeq\sum_{n=1}^\infty\mathbf{a}_n(k) q^n\in L[\![q]\!]\] is a classical normalized
cuspidal eigenform of level $Np$, weight $k$ and slope $\lambda$, and for $k=k_0$, $f_{k_0}=f$; see also \cite[Theorem 4.6.4]{LoefflerZerbes.Rankin}.
\end{proposition}

The family of functions $\{\mathbf{a}_n\}_n$ is the \textit{(elliptic) Coleman family} defined over $\mathscr{U}$ passing through $f$. In other words, a Coleman family is an affinoid $\{f_k\}_{k\in\mathscr{U}}$ in the Coleman--Mazur eigencurve $\mathcal{E}(N)$ defined over $\mathscr{U}$. Let $\Lambda_\mathscr{U}$ be the Iwasawa algebra associated to $\mathscr{U}$, over which $\mathscr{A}_\mathscr{U}^\circ$ is finite and flat module. The (at first) formal series \[\mathscr{f}\defeq\sum_{n=1}^\infty\mathbf{a}_n q^n\in \mathscr{A}_{\mathscr{U}}^\circ[\![q]\!]\] is a well defined element in $V_p(N^+,\mathscr{A}_{\mathscr{U}}^\circ)\defeq V_p(N^+,L)\widehat{\otimes}_{\Lambda_{\mathscr{U}}}\mathscr{A}_{\mathscr{U}}^\circ$.

Let $f$ be a classical $p$-stabilized eigenform as above, $\mathcal{F}\defeq\mathrm{JL}(f)$ be the correspondent $p$-stabilized quaternionic modular form, and $\mathscr{f}$ be the Coleman family passing through $f$ defined over some affinoid $\mathscr{U}\subseteq\mathscr{W}$ (after projection from $\mathscr{W}_N$). Via the $p$-adic Jacquet--Langlands correspondence (Proposition \ref{prop.padicJacquetLanglands}), $\mathscr{f}$, as an open neighborhood of $f$ in $\mathcal{E}(N)^{N^{-}\text{-new}}$ maps into an open neighborhood $\mathscr{F}$ of $\mathcal{F}$ in $\mathscr{E}(N^+)$, which we call the \textit{(quaternionic) Coleman family} passing through $\mathcal{F}$ defined over $\mathscr{U}$. As an abuse of notation, we write $\mathscr{F}=\mathrm{JL}(\mathscr{f})$ to emphasize the relation between the elliptic and quaternionic Coleman families. As a consequence of Propositions \ref{prop.padicJacquetLanglands} and \ref{prop.ColemanFamilies}:

\begin{corollary}\label{cor.QuaternionicColemanFamilies}
Let $\mathscr{F}=\mathrm{JL}(\mathscr{f})$ be the Coleman family passing through $\mathcal{F}$ over $\mathscr{U}\subseteq\mathscr{W}$. For each $k\in\mathscr{U}\cap\Z$ with $\lambda<k-1$, the specialization at weight $k$ denoted $\mathcal{F}_k$ is the quaternionic modular form $\mathrm{JL}(f_k)$, where $f_k$ is the specialization at weight $k$ of $\mathscr{f}$. At weight $k=k_0$, the specialization is $\mathcal{F}$ itself.
\end{corollary}

\section{Big generalized Heegner classes}\label{sec.BigHeegnerClasses}
We define big generalized Heegner classes associated to Coleman families of quaternionic modular forms through the $p$-adic interpolation of generalized Heegner classes associated to quaternionic modular forms, following the approach from \cite[\S 4]{JetchevLoefflerZerbes}. As seen in \S\ref{subsec.GysinMap}, the basis vectors we constructed after CM points correspond under the Gysin map to the generalized Heegner classes (Proposition \ref{prop.Gysin}), so the strategy is to interpolate $p$-adically the basis vectors and the Gysin map.

Let $k_0=2r_0+2\in\Z$, which can be found in $\mathscr{W}$ as the character $z\mapsto z^{2r_0}$. For an integer $s\ge 0$, denote by $\mathscr{W}^{(s)}$ the locus of the $s$-analytic characters, which are the $\kappa\in\mathscr{W}$ such that \[v_p(\kappa(1+p)-1)>(p^{s-1}(p-1))^{-1}.\]
Let $\mathscr{U}$ be an open neighborhood of $z_0$ in $\mathscr{W}^{(s)}$ defined over a $p$-adic field $L$ in which $K$ embeds by $\sigma\colon K\hookrightarrow L$, $\Lambda_\mathscr{U}\defeq \mathcal{O}_L[\![u]\!]$ be its Iwasawa algebra and let $\kappa_{\mathscr{U}}\colon\Gamma\hookrightarrow\Lambda_{\mathscr{U}}^\times$ be its universal character. For any character $\sigma\colon\mathcal{O}_L^\times\to\Gamma$ we write $\sigma^{\kappa_{\mathscr{U}}}\defeq\kappa_{\mathscr{U}}\circ\sigma$, $\sigma^{k}\defeq k\circ\sigma$ for any $k\in\Z$ and $\sigma^{\pm\kappa_\mathscr{U}+i_1}\bar{\sigma}^{i_2}\defeq (\sigma^{\kappa_{\mathscr{U}}})^{\pm 1}\sigma^{i_1}\bar{\sigma}^{i_2}$ for any two integers $i_1$ and $i_2$.

Recall the fixed new at $p$ elliptic eigenform $f_0^\sharp$ of weight $k_0$ and level $\Gamma_1(N)$ from \S\ref{subsec.GeneralizedHeegnerHypothesis}, let $f_0$ be the $p$-stabilization of $f_0^\sharp$ with respect to some root $\upalpha$ of its Hecke polynomial, and let $\mathcal{F}_0=\mathrm{JL}(f_0)$, a $p$-stabilized quaternionic modular form of signature $(k_0,\psi_0)$ and level $N^{+}p$. Let as in \S\ref{subsec.ColemanFamilies} $\mathscr{F}$ be the fixed Coleman family passing through $\mathcal{F}_0$ defined over $L$ and let $\mathscr{U}$ be an open neighborhood of $k_0$ small enough so that $\mathscr{F}$ specializes at every $k\in\mathscr{U}\cap\Z$ to a $p$-stabilized modular form.

\begin{remark}\label{rem.InvertPrimes}
In all étale cohomology groups in this section, we always invert (but don't write) all primes dividing $Np$ in the respective fields of definition.
\end{remark}

\subsection{Classes associated to quaternionic modular forms}\label{subsec.ClassesModularForms}

Let $\mathcal{F}\in M_k(N^+p^m,\mathcal{O}_L)$ be a quaternionic newform of signature $(k,\psi)$ over $\widetilde{X}_m$. From \S\ref{subsec.GaloisRepresentations}, the dual Galois representation $V_\mathcal{F}^*$ of $\mathcal{F}$ is a quotient of $H^1_{\et}({X}_m\otimes_\Q\overline{\Q},\mathscr{V}^{2r}_{\et}(1))$, so we have a projection
\begin{equation}\label{eq.ProjectionMF1}
\mathrm{pr}^{[\mathcal{F},j]}\colon H^1\left(F_{cp^n},H^1_{\et}({X}_m\otimes_\Q\overline{\Q},\mathscr{V}^{2r}_{\et}(1))\otimes\sigma^{2r-j}_{\et}\bar{\sigma}_{\et}^j\right)\longepi\\ H^1\left(F_{cp^n},V_\mathcal{F}^*\otimes\sigma_{\et}^{2r-j}\bar\sigma^j_{\et}\right).
\end{equation}

\begin{definition}
The class $z_{cp^n,m}^{[\mathcal{F},j]}\defeq\mathrm{pr}^{[\mathcal{F},j]}\left(z^{[k,j]}_{cp^n,m}\right)$ is the $j$-component of the \emph{generalized Heegner class} associated to $\mathcal{F}$ and with $0\leq j\leq 2r$.
\end{definition}

\begin{lemma}\label{lem.Shimura}
Let $\chi$ be a Hecke character of infinity type $(2r-j,j)$ and finite type $(cp^n,\mathfrak{N}^+,\psi)$. Then the generalized Heegner class $z_{cp^n,m}^{[\mathcal{F},j]}$ belongs to the $\Gal(F_{cp^n}/H_{cp^n})$-invariant subspace of \[H^1(F_{cp^n},V^\ast_\mathcal{F}\otimes\sigma_{\et}^{2r-j}\sigma^j_{\et})\cong H^1(F_{cp^n},V_\mathcal{F}^\ast\otimes\chi).\]
\end{lemma}
\begin{proof}
The proof is similar to \cite[Proposition 3.5.2]{JetchevLoefflerZerbes}. The group $\widehat{\mathcal{O}}_{cp^n}^\times/U_{\mathfrak{N}^+,cp^n}\cong(\Z/N^+\Z)^\times$ acts naturally on $S_{cp^n}$ and by diamond operators on $\widetilde{X}_m$, and both actions are compatible with the embedding $\delta_{cp^n}^*\colon S_{cp^n}\hookrightarrow \widetilde{X}_m\times S_{cp^n}$ induced by the embedding from \S\ref{subsec.GysinMap}. Since $\chi$ has conductor prime to $N^+$, $\chi$ restricts to the character $\psi\sigma^{2r-j}_{\et}\bar\sigma^{j}_{\et}$ on $\Gal(F_{cp^n}/H_{cp^n})$, thus extending $\sigma^{2r-j}\bar\sigma^{j}$ to $\Gal(K^\mathrm{ab}/ H_{cp^n})$ after twisting the Galois action by $\psi$. Since $z_{cp^n,m}^{[\mathcal{F},j]}$ lies in the 
finite dimensional subspace of classes which are unramified outside $Np$, the result then follows from Shimura reciprocity law (\cite[Theorem 9.6]{Shimura}).
\end{proof}
 
Via the inflation-restriction exact sequence and the irreducibility of $V_\mathcal{F}$, there is an isomorphism 
\begin{equation}\label{eq.InflationRestriction}
\left(H^1(F_{cp^n},V_\mathcal{F}^\dagger\otimes\chi)\right)^{\Gal(F_{cp^n}/H_{cp^n})}\cong H^1(H_{cp^n},V_\mathcal{F}^\dagger\otimes\chi).
\end{equation}
\begin{definition}
Let $\chi$ be a Hecke character of infinity type $(2r-j,j)$ with $0\le j\le 2r$. The \emph{$\chi$-component} $z_{cp^n,m}^{[\mathcal{F},\chi]}$
of the generalized Heegner class is the image of $z_{cp^n,m}^{[\mathcal{F},j]}$ via the isomorphism \eqref{eq.InflationRestriction}.
\end{definition}

Let $\mathcal{F}=\mathrm{JL}(f)$ be the quaternionic lift of a new at $p$ elliptic eigenform $f$ and $\upalpha, \upbeta$ be the roots of its Hecke polynomial; let $\mathcal{F}_\upalpha$ the $p$-stabilization of $\mathcal{F}$ with respect to $\upalpha$. Proposition \ref{prop.ProjStabIsom} gives an isomorphism $(\mathrm{pr}_\upalpha)_*\colon V_{\mathcal{F}_\upalpha}\stackrel{\sim}{\to}V_{\mathcal{F}}$.

\begin{proposition}\label{prop.ProjStab}
The map $(\mathrm{pr}_\upalpha)_*$ induces an isomorphism \[H^1\left(H_{cp^n},V_{\mathcal{F}_\upalpha}^\ast\otimes\chi\right)\stackrel{\sim}{\longrightarrow} H^1\left(H_{cp^n},V_\mathcal{F}^\ast\otimes\chi\right)\] under which $z_{cp^n,1}^{[\mathcal{F}_\upalpha,\chi]}$ maps into $E_p(\mathcal{F},\chi)\cdot z_{cp^n,0}^{[\mathcal{F},\chi]}$, where the \textit{Euler factor} $E_p(\mathcal{F},\chi)$ is defined by \[E_p(\mathcal{F},\chi)\defeq\begin{cases}
\left(1-\frac{\chi(\mathfrak{p})}{\upalpha} \right)\medskip\left(1-\frac{\chi(\bar{\mathfrak{p}})}{\upalpha}\right), & \text{ if $p$ is split;}\\\medskip
1-\frac{\chi(\mathfrak{p})}{\upalpha^2}, & \text{ if $p$ is inert;}\\
1-\frac{\chi(\mathfrak{p})}{\upalpha}, &\text{ if $p$ is ramified.}
\end{cases}\]
\end{proposition}
This is an explicit calculation as in \cite[Theorem 5.7.6]{KingsLoefflerZerbes.ERL}; see also \cite[Proposition 3.5.5]{JetchevLoefflerZerbes}.

\subsection{$p$-adic interpolation of basis vectors}
Let $A_{\mathscr{U},s}^\circ$ be the $\Lambda_\mathscr{U}$-module of continuous bounded $\Lambda_\mathscr{U}$-valued functions on $\mathcal{O}_L$ converging on $p^s\mathcal{O}_{\C_p}$, that is, the power series of the form $\sum_{t\ge 0}a_t\left(\tfrac{Z}{p^s}\right)^t$ with $a_t\in\Lambda_\mathscr{U}$. We call $s$ the \textit{slope} of these functions. Inverting $p$ we get the space $A_{\mathscr{U},s}\defeq A_{\mathscr{U},s}^\circ[1/p]$ of continuous functions with slope $s$.

The monoid $\Sigma(p^m)=\begin{pmatrix}\mathcal{O}_L & \mathcal{O}_L \\ p^m\mathcal{O}_L & \mathcal{O}_L^\times\end{pmatrix}$, generated by $\begin{pmatrix}p & 0 \\ 0 & 1\end{pmatrix}$ and elements of $\GL_2(\mathcal{O}_L)$ that are upper triangular modulo $p^m$, acts on $A_{\mathscr{U},s}$ from the left by \[\gamma\cdot \varphi(Z)=\kappa_\mathscr{U}(bZ+d)\varphi(Z\cdot \gamma),\] where $\gamma=\smallmat abcd$ acts on the variable $Z$ from the \textit{right} as $Z\cdot\gamma=\frac{aZ+c}{bZ+d}$. The space $\Sym^{2r}(L^2)$ of weight $2r$ homogeneous polynomials in two variables $X$ and $Y$ embeds into $A_{\mathscr{U},s}$ by setting $Z$ a fractional linear combination of $X$ and $Y$ and mapping $P(X,Y)\mapsto P(Z,1)$. In view of Lemma \ref{lem.vartheta}, we set $Z=u^{-1}\cdot(X,Y)^\mathsf{T}$, where $u=(u_{cp^n}^{-1})^\mathsf{T}$. Such embedding is equivariant with respect to the action above and its restriction to $\Sym^{2r}(L^2)$ given by \[\gamma\cdot \varphi(Z)=(bZ+d)^{2r}\varphi(Z\cdot \gamma).\] The dual of this embedding gives the \textit{moment map} 
\[\mom^{2r}\colon D_{\mathscr{U},s}^\circ\defeq\Hom_{\Lambda_{\mathscr{U}}}(A_{\mathscr{U},s}^\circ,\Lambda_\mathscr{U})\longrightarrow \TSym^{2r}((L^2)^\vee)\] (and similarly for $D_{\mathscr{U},s}\defeq D_{\mathscr{U},s}^\circ[1/p]$ over $\Lambda_\mathscr{U}[1/p]$) defined by the integration formula 
\begin{equation}
\left(\mathrm{mom}^{2r}(\mu)\right)(\varphi)=\int_{\mathcal{O}_L}\varphi(z,1)\ \mathrm{d}\mu(z),
\end{equation}
for each $\varphi\in\Sym^{2r}(L^2)$. The dual action of  $\Sigma'(p^m)\defeq \{\gamma^{-1}:\ \gamma\in\Sigma(p^m)\}$ is given by \[\gamma\cdot\mu(\varphi)=\mu(\gamma^{-1}\cdot\varphi).\]

\begin{definition}
We denote by $\boldsymbol{\mathrm{e}}_{cp^n}^{[\mathscr{U},0]}\in D_{\mathscr{U},m}$ the distribution given by the integration formula 
\[\boldsymbol{\mathrm{e}}_{cp^n}^{[\mathscr{U},0]}(\varphi)=\int_{\mathcal{O}_L}\varphi(z)\ \mathrm{d}\mathbf{e}_{cp^n}^{[\mathscr{U},0]}(z)\defeq\varphi(e_{cp^n})\] for any $\varphi\in A_{\mathscr{U},m}$, where, in light of Lemma \ref{lem.vartheta}, $\varphi(e_{cp^n})$ works out to be evaluation at $p^n\sigma_L(\delta)$ if $p$ is split, and at $p^n\sigma_L(\bar\vartheta)$ otherwise.
\end{definition}

\begin{remark}
The distribution $\boldsymbol{\mathrm{e}}_{cp^n}^{[\mathscr{U},0]}$ behaves independently of $m$ in the sense that the distributions of each $m$-layer are compatible with respect to the natural maps $D_{\mathscr{U},m}\to D_{\mathscr{U},m-1}$.
\end{remark}

\begin{proposition}\label{prop.BigBaseVectors} The distribution $\boldsymbol{\mathrm{e}}_{cp^n}^{[\mathscr{U},0]}$ plays the role of a $p$-adic family over $\mathscr{U}$ of base vectors having the following properties: 
\begin{itemize}
\item The action of $i_{cp^n}\left((\mathcal{O}_{cp^n}\otimes\Z_p)^\times\right)$ on $\mathbf{e}_{cp^n}^{[\mathscr{U},0]}$ is via $\sigma_L^{-\kappa_\mathscr{U}}$. 
\item For all integers $k\geq 0$ we have 
$\operatorname{mom}^{2r}(\boldsymbol{\mathrm{e}}_{cp^n}^{[\mathscr{U},0]})=e^{[k,0]}_{cp^n}.$
\end{itemize}
\end{proposition}
\begin{proof}
Let $\bigstar$ denote either $p^n\sigma_L(\delta)$ or $p^n\sigma_L(\bar\vartheta)$ depending on whether $p$ is split or not, respectively. For the first statement, recall that $i_{cp^n}$ is an optimal embedding 
of $\mathcal{O}_{cp^n}$ into the Eichler order $R_m=B^\times\cap U_m$. For $u$ in $(\mathcal{O}_{cp^n}\otimes\Z_p)^\times$, we have
$(i_{cp^n}(u)^{-1})^\mathsf{T}e_{cp^n}=\sigma_L^{-1}(u)e_{cp^n}$ (see \S\ref{subsec.BasisVectors}). Let $u$ act under $i_{cp^n}$ by $\smallmat{a}{b}{c}{d}$. As we just observed, 
\begin{equation}\label{eq.sigmaLu}
\binom{a\bigstar + c}{b\bigstar + d}=i_{cp^n}(u)^\mathsf{T}\binom{\bigstar}{1}=\sigma_L(u)\binom{\bigstar}{1}.
\end{equation}
For each $\varphi\in A_{\mathscr{U},m}$ we then have
\begin{align*}
i_{cp^n}(u)^{-1}\mathbf{e}_{cp^n}^{[\mathscr{U},0]}(\varphi)&=\mathbf{e}_{cp^n}^{[\mathscr{U},0]}(i_{cp^n}(u)\cdot \varphi)=\int_{\mathcal{O}_L}\kappa_\mathscr{U}(bx+d)\varphi\left(\frac{ax+c}{bx+d}\right)\ \mathrm{d}\mathbf{e}_{cp^n}^{[\mathscr{U},0]}(x)\\
&=\kappa_\mathscr{U}(b\bigstar+d)\varphi\left(\dfrac{a\bigstar+c}{b\bigstar+d}\right)\stackrel{\eqref{eq.sigmaLu}}{=}\sigma_L^{\kappa_\mathscr{U}}(u)\varphi(\bigstar)=\sigma_L^{\kappa_\mathscr{U}}(u)\mathbf{e}_{cp^n}^{[\mathscr{U},0]}(\varphi).
\end{align*}
Therefore, we conclude that the action of $i_{cp^n}(u)$ on the measure $\mathbf{e}_{cp^n}^{[\mathscr{U},0]}$ is just the product  
$\sigma_L^{-\kappa_\mathscr{U}(u)}\mathbf{e}_{cp^n}^{[\mathscr{U},0]}$, which gives the first statement. The second statement follows once noticing that over polynomials both $e^{[k,0]}_{cp^n}$ and $\boldsymbol{\mathrm{e}}_{cp^n}^{[\mathscr{U},0]}$ act as evaluation at $\bigstar$.
\end{proof}

To interpolate the remaining vectors, we introduce some twist. First, recall from \cite[\S 5.1]{LoefflerZerbes.Rankin} the operator $\nabla$, which on locally analytic $L$-valued functions defined on $\mathscr{U}$ by \[\nabla f(z)=\dfrac{\mathrm{d}}{\mathrm{d}t}f(tz)\Big|_{t=1}\] (\textit{ibid.} Proposition 5.1.2). For each character $\kappa\in\mathscr{W}$ defined over $\mathscr{U}$, $\nabla\circ\kappa|_{\mathscr{U}}$ acts as multiplication by $\kappa'(1)$ and, in particular, at weight $k\in\Z\cap\mathscr{U}$ corresponding to the character $x\mapsto x^{k-2}$, $\nabla$ specializes to multiplication by $k-2$ (\textit{ibid.} Proposition 5.2.5). Let \[\Pi_j\colon D_{\mathscr{U}-j,m}\otimes\TSym^j\left((L^2)^\vee\right)\longrightarrow D_{\mathscr{U},m}\]denote the \textit{overconvergent projector} from \cite[Corollary 5.2.1]{LoefflerZerbes.Rankin}, which essentially acts as the usual symmetrized tensor product up to denominators: there is a map acting on evaluations at $v\in (L^2)^\vee$ by \[\Pi_j^*\colon D_{\mathscr{U},m}\longrightarrow D_{\mathscr{U}-j,m}\otimes\TSym^j\left((L^2)^\vee\right)\colon\mathrm{ev}_v\mapsto\mathrm{ev}_v\circ v^{\odot j}\] such that $\Pi_j\circ\Pi_j^*$ acts as multiplication by $\binom{\nabla}{j}$. The denominators of $\Pi_j$ are explicitly bounded in terms of $j$ and $m$, see \textit{ibid.} Remark 5.2.2.

\begin{definition}
We denote by $\boldsymbol{\mathrm{e}}_{cp^n}^{[\mathscr{U},j]}$ the distribution given by $\Pi_j(\mathbf{e}_{cp^n}^{[\mathscr{U},0]}\odot (e_{cp^n}^{-j}\odot\bar{e}_{cp^n}^{j}))$.
\end{definition}

As an immediate consequence of Proposition \ref{prop.BigBaseVectors}:

\begin{corollary}\label{cor.BigBaseVectorsj}
The distribution $\boldsymbol{\mathrm{e}}_{cp^n}^{[\mathscr{U},j]}$ has the following properties: 
\begin{itemize}
\item The group $i_{cp^n}((\mathcal{O}_{cp^n}\otimes\Z_p)^\times)$ acts on 
$\boldsymbol{\mathrm{e}}_{cp^n}^{[\mathscr{U},j]}$ via the 
representation $\sigma^{-(\kappa_\mathscr{U}-j)}\bar{\sigma}^{j}$ 
\item For all integers $k\geq 0$ we have 
 $\operatorname{mom}^{2r}(\mathbf{e}^{[\mathscr{U},j]}_{cp^n})=e^{[k,j]}_{cp^n}.$
\end{itemize}
\end{corollary}

\subsection{$p$-adic interpolation of generalized Heegner classes}
As in \S\ref{subsec.AbelJacobi}, let $\sigma_{\et}^{\kappa_{\mathscr{U}}-j}\sigma_{\et}^j$ be the étale realization of $\sigma_L^{\kappa_{\mathscr{U}}-j}\sigma_L^j$. The ``big'' Gysin map
\begin{equation}\label{eq.BigGyin}
\xymatrix{
H^0_{\et}\left(S_{cp^n},\delta_{cp^n}^\ast(D_\mathscr{U}\otimes\sigma^{\kappa_\mathscr{U}-j}_{\et}\bar{\sigma}^j_{\et})\right)\ar[r]^-{\delta_{cp^n,\ast}}& H^2_{\et}\left({X}_m\otimes_K F_{cp^n},D_\mathscr{U}(1)\otimes\sigma^{\kappa_\mathscr{U}-j}_{\et}\bar{\sigma}_{\et}^j\right)}
\end{equation}
interpolates the Gysin maps from \S\ref{subsec.GysinMap}, in the sense of being compatible with the one in \eqref{eq.Gysin} via the moment maps: for each $k\in\Z\cap\mathscr{U}$ with $k\geq j$ we have, writing $\overline{X}_m\defeq X_m\otimes_\Q\overline\Q$ and $\mathscr{D}^{(2r-j,j)}_{\et}\defeq D_{\mathscr{U},m}\otimes\sigma_{\et}^{\kappa_\mathscr{U}-j}\bar{\sigma}_{\et}^j$ to simplify the notation in the diagram,
\begin{equation}\label{eq.DiagramInterpolationGHCs}
\begin{tikzcd}[column sep = large]
H^0_{\et}\left(S_{cp^n},\delta_{cp^n}^\ast(\mathscr{D}^{(2r-j,j)}_{\et})\right)\arrow{d}{\delta_{cp^n,\ast}}\arrow{r}{\operatorname{mom}^{2r}}& H^0_{\et}\left(S_{cp^n},\delta_{cp^n}^\ast(\mathscr{M}_{\et}^{(2r-j,j)})\right)\arrow{d}{\delta_{cp^n,\ast}} \\ 
H^2_{\et}\left({X}_m\otimes F_{cp^n},\mathscr{D}^{(2r-j,j)}_{\et}(1)\right)\arrow{d}{\sim}\arrow{r}{\mathrm{mom}^{2r}} & H^2_{\et}\left({X}_m\otimes F_{cp^n},\mathscr{M}_{\et}^{(2r-j,j)}(1)\right)\arrow{d}{\sim}\\
H^1\left(F_{cp^n},H_{\et}^1(\overline{X}_m,\mathscr{D}^{(2r-j,j)}_{\et}(1))\right)\arrow{r}{\mathrm{mom}^{2r}} & H^1\left(F_{cp^n},H^1_{\et}(\overline{X}_m,\mathscr{M}_{\et}^{(2r-j,j)}(1))\right)
\end{tikzcd}
\end{equation}

\begin{definition}\label{def.jBigHeegnerClass}
Let $j\geq 0$ and $n\geq m\geq 1$ be integers. 
The \emph{$j$-component of the big generalized Heegner class associated to $\mathscr{U}$} is $\mathbf{z}^{[\mathscr{U},j]}_{cp^n,m}\defeq\delta_{cp^n,\ast}(\boldsymbol{\mathrm{e}}^{[\mathscr{U},j]}_{cp^n}).$
\end{definition}

By Diagram \eqref{eq.DiagramInterpolationGHCs}, $\mathbf{z}^{[\mathscr{U},j]}_{cp^n,m}$ is an element of $H^1\big(F_{cp^n},H^1_{\et}(\overline{X}_m,\mathscr{D}_{\et}^{(2r-j,j)}(1))\big)$.

\begin{proposition}
For each $k\in\Z\cap\mathscr{U}$ with $k\geq j$ we have $\mathrm{mom}^{2r}(\mathbf{z}^{[\mathscr{U},j]}_{cp^n,m})=z_{cp^n,m}^{[k,j]}.$
\end{proposition}
\begin{proof}
Again by Diagram \eqref{eq.DiagramInterpolationGHCs} and Corollary \ref{cor.BigBaseVectorsj}, \[\mathrm{mom}^{2r}(\mathbf{z}^{[\mathscr{U},j]}_{cp^n,m})=\mathrm{mom}^{2r}(\delta_{cp^n,\ast}(\boldsymbol{\mathrm{e}}^{[\mathscr{U},j]}_{cp^n}))=\delta_{cp^n,\ast}(\mathrm{mom}^{2r}(\boldsymbol{\mathrm{e}}^{[\mathscr{U},j]}_{cp^n}))=\delta_{cp^n,\ast}(e_{cp^n,m}^{[k,j]}),\] and it follows then from Proposition \ref{prop.Gysin}.
\end{proof}

Next we derive the norm-compatibility relations satisfied by the set of big generalized Heegner classes or varying $n$, $m$ and $j$ that turns it into an Euler system, arguing similarly to the elliptic counterpart in \cite[Proposition 5.1.2]{JetchevLoefflerZerbes}. We begin with an alternative description of the induced action of the Hecke operators in the cohomology level.

For $m\ge 1$, let $\widetilde{U}^1_m$ be the subgroup of $\widetilde{U}_m$ whose elements have $p$-parts \textit{lower} triangular modulo $p$: \[\widetilde{U}^1_m\defeq\{g\in \widetilde{U}_m:\ g_p\equiv\left(\begin{smallmatrix}* & 0 \\ * & *\end{smallmatrix}\right)\mod p\}\] (notice that these matrices are already upper triangular modulo $p$ by being such modulo $p^m$, thus elements of $\widetilde{U}^1_m$ are diagonal modulo $p$). We define the correspondent Shimura curve as \[\widetilde{X}^1_m\defeq B^\times\backslash (\mathcal{H}^{\pm}\times\widehat{B}^\times)/ \widetilde{U}_m^1,\] which, as a moduli space, classifies triplets $(A,\iota,\alpha)$ consisting of a QM abelian surface $(A,\iota)$, a level $U$ structure $\alpha$ of full level $N^+p^{m+1}$, where, in the notation of \S\ref{subsec.NaiveLevelStructures} \[U=\{g\in\widehat{\mathcal{O}}^\times_B:\ \exists a,b,d\text{ such that }i_{N^{+}p^m}\circ\widehat{\pi}_{N^+p^m}(g)\equiv\left(\begin{smallmatrix}a & pb \\ 0 & d\end{smallmatrix}\right)\mod N^+p^m\},\] and to which we may also attach an arithmetic trivialization $\beta\colon\boldsymbol{\mu}_{p^{m+1}}\to e_pA[p^{m+1}]^0$. This kind of Shimura curve relates to the other kind $\widetilde{X}_m$ by the following maps:

\begin{center}
\begin{tikzcd}[column sep = large]
\widetilde{X}_m^1 \arrow[r,"\Phi_p"]\arrow[d, "\hat{\mathrm{pr}}", swap] & \widetilde{X}_{m+1} \arrow[d,"\mathrm{pr}"] \\ \widetilde{X}_m & \,\widetilde{X}_m,
\end{tikzcd}
\end{center}
where

\begin{itemize}
\item $\mathrm{pr}$ takes a quadruple $(A,\iota,\alpha,\beta)$ and maps into $(A,\iota,\alpha_p,\beta_p)$, where $\beta_p=\beta\circ\varphi_p$, being $\varphi_p\colon\boldsymbol{\mu}_{p^{m+1}}\to\boldsymbol{\mu}_{p^{m}}$ the map $x\mapsto x^p$, and $\alpha_p$ is (the equivalence class of) $\psi_p\circ\alpha$, where $\psi_p$ is the map $(x,y)\mapsto (px,py)$;

\item $\hat{\mathrm{pr}}$ acts similarly to $\mathrm{pr}$ with the correspondent changes in the level structure;

\item $\Phi_p$ acts as $\left(\begin{smallmatrix}p & 0 \\ 0 & 1\end{smallmatrix}\right)$ on the level structure.
\end{itemize}

These maps induce pullbacks and pushforwards (trace maps) in the cohomology of the Shimura curves above. Following \cite[\S 2.4]{KingsLoefflerZerbes.ERL} and \cite[\S 2.9]{Kato}, we define the Hecke operator \[U_p\defeq (\hat{\mathrm{pr}})_*\circ(\Phi_p)^*\circ(\mathrm{pr})^*,\] which coincides with the previously defined Hecke operator $U_p$ in the level of quaternionic modular forms as follows: in the notation of \S\ref{subsec.HeckeOperators}, $(\mathrm{pr})^*$ lifts a quadruplet $(A,\iota,\alpha,\beta)$ to $(A,\iota,\alpha_{p^{-1}},\beta_{p^{-1}})$ by composing $\alpha$ with $\psi_{p^{-1}}\colon(x,y)\mapsto (p^{-1}x,p^{-1}y)$ and precomposing $\beta$ with $\varphi_{p^{-1}}\colon x\mapsto x^{1/p}$; the isomorphism $(\Phi_p)^*$ introduces the averaging factor $\frac{1}{p}$; lastly, the pushforward (trace) $(\hat{\mathrm{pr}})_*$ projects the quadruplets back to level $U_1(N^+p^m)$ by taking the trace of the quotients running over all $p$-torsion subgroups of $A$, recovering the original definition of $U_p$.

Under the Poincaré duality (refer again to \cite[Corollary 11.2]{Milne.EC}), pushforwards and pullbacks are dual to each other, and $(\Phi_p)_*=(\Phi_{p}^{-1})^*$ as the induced actions of $\left(\begin{smallmatrix}p & 0 \\ 0 & 1\end{smallmatrix}\right)$ and $\left(\begin{smallmatrix}p^{-1} & 0 \\ 0 & 1\end{smallmatrix}\right)$ are dual to each other. Thus, the dual to the Hecke operator $U_p$ is \[U_p'\defeq ({\mathrm{pr}})_*\circ(\Phi_p^{-1})^*\circ(\hat{\mathrm{pr}})^*.\]

\begin{proposition}[Euler relations]\label{prop.EulerRelations}
$\cores_{F_{cp^{n+1}}/F_{cp^n}}(\mathbf{z}_{cp^{n+1},m}^{[\mathscr{U},j]})=U_p'\cdot\mathbf{z}_{cp^{n},m}^{[\mathscr{U},j]}$, for all $n\geq m\geq 1$ 
and all $j\geq 0$. 
\end{proposition}
\begin{proof}
Since $\Gal(H_{cp^{n+1}}/H_{cp^n})\cong \widetilde{U}_m/\widetilde{U}_m^1\cong \widetilde{U}_m/\widetilde{U}_{m+1}$, the last isomorphism being a result of the action of $\Phi_p$ as $\left(\begin{smallmatrix}p & 0 \\ 0 & 1\end{smallmatrix}\right)$, the preimage of the $(\widetilde{U}_m/\widetilde{U}_m^1)$-orbit of $x_{cp^n,m}(1)=[(\iota_K,\xi_{cp^n})]$ in $\widetilde{X}_m$ under $\hat{\mathrm{pr}}$ is the $\Gal(H_{cp^{n+1}}/H_{cp^n})$ orbit of the CM point $[(\iota_K,\xi_{cp^n})]$ in $\widetilde{X}_m^1$. Repeating the construction of the $j$-component of the big generalized Heegner class associated to $\mathscr{U}$ for the Shimura curve $\widetilde{X}_m^1$ gives a class \[\mathbf{z}^{1,[\mathscr{U},j]}_{cp^n,m}=\delta_{cp^n,*}(\mathbf{e}^{1,[\mathscr{U},j]}_{cp^n})\in H^1\left(F_{cp^{n+1}},H^1_{\et}(\widetilde{X}_m^1\otimes_{\Q}\overline{\Q}, D_{\mathscr{U},m}\otimes\sigma_{\et}^{\kappa_\mathscr{U}-j}\bar{\sigma}_{\et}^j)\right),\] defined over $F_{cp^{n+1}}$, and thus satisfying \[\cores_{F_{cp^{n+1}}/F_{cp^n}}(\mathbf{z}^{1,[\mathscr{U},j]}_{cp^n,m})=(\hat{\mathrm{pr}})^*(\mathbf{z}^{[\mathscr{U},j]}_{cp^n,m}).\] Therefore
\begin{align*}
U_p'\cdot\mathbf{z}_{cp^{n},m}^{[\mathscr{U},j]}&=({\mathrm{pr}})_*\circ(\Phi_p^{-1})^*\circ(\hat{\mathrm{pr}})^*(\mathbf{z}^{[\mathscr{U},j]}_{cp^n,m})\\
&=({\mathrm{pr}})_*\circ(\Phi_p^{-1})^*(\cores_{F_{cp^{n+1}}/F_{cp^n}}(\mathbf{z}^{1,[\mathscr{U},j]}_{cp^n,m}))\\
&=\cores_{F_{cp^{n+1}}/F_{cp^n}}\left(({\mathrm{pr}})_*\circ(\Phi_p^{-1})^*(\mathbf{z}^{1,[\mathscr{U},j]}_{cp^n,m})\right)
\end{align*}
The morphism $\Phi_p$ maps the CM point $[(\iota_K,\xi_{cp^n})]\in \widetilde{X}_m^1$ into $[(\iota_K,\xi_{cp^{n+1}})]\in \widetilde{X}_{m+1}$, so \[(\Phi_p^{-1})^*(\mathbf{z}^{1,[\mathscr{U},j]}_{cp^n,m})=(\Phi_p)_*(\mathbf{z}^{1,[\mathscr{U},j]}_{cp^n,m})=\mathbf{z}^{[\mathscr{U},j]}_{cp^{n+1},m+1},\] and finally $(\mathrm{pr})_*$ projects $[(\iota_K,\xi_{cp^{n+1}})]\in \widetilde{X}_{m+1}$ into the ``same point'' $[(\iota_K,\xi_{cp^{n+1}})]\in \widetilde{X}_{m}$ (this identity like behavior of the pushforward of the projection is akin to what happens in the elliptic case, as seen in \cite[\S 2.4]{KingsLoefflerZerbes.ERL}). Piecing it all together yields

\begin{align*}
U_p'\cdot\mathbf{z}_{cp^{n},m}^{[\mathscr{U},j]}&=\cores_{F_{cp^{n+1}}/F_{cp^n}}\left(({\mathrm{pr}})_*\circ(\Phi_p^{-1})^*(\mathbf{z}^{1,[\mathscr{U},j]}_{cp^n,m})\right)\\
&=\cores_{F_{cp^{n+1}}/F_{cp^n}}\left(({\mathrm{pr}})_*(\mathbf{z}^{[\mathscr{U},j]}_{cp^{n+1},m+1})\right)\\
&=\cores_{F_{cp^{n+1}}/F_{cp^n}}\left(\mathbf{z}^{[\mathscr{U},j]}_{cp^{n+1},m}\right).\qedhere
\end{align*}
\end{proof}

\subsection{Big Galois representations}\label{subsec.BigGaloisRepresentations}
Let $\mathscr{f}$ be the Coleman family defined over $\mathscr{U}\subseteq\mathscr{W}$ passing through $f$ corresponding via the $p$-adic Jacquet--Langlands lift to $\mathscr{F}$.

\begin{proposition}\label{prop.BigGaloisRep}
After shrinking $\mathscr{U}$, there is an open disc $\widetilde{\mathscr{U}}\supseteq\mathscr{U}$ in $\mathscr{W}^{(1)}$ so that the $\mathcal{O}(\mathscr{U})$-module \[H^1_{\et}\left(X_1\otimes_{\Q}\overline{\Q},D_{\widetilde{\mathscr{U}},1}(1)\right)\widehat{\otimes}_{\Lambda_{\widetilde{\mathscr{U}}}[1/p]}\mathcal{O}(\mathscr{U})\] has a rank 2 maximal quotient $\mathbf{V}_{\mathscr{F}}^*$ interpolating the dual $p$-adic Galois representations of the specializations at integral weights of $\mathscr{F}$ in the following sense: for each $k\in\mathscr{U}\cap\Z_{\ge 0}$, there is an isomorphism of $L$-representations \[\mathbf{V}_{\mathscr{F}}^*/k\mathbf{V}_{\mathscr{F}}^*\cong V_{\mathcal{F}}^*,\] where $k$ is seen as a character in $\mathscr{W}$.
\end{proposition}

\begin{proof}
Using the $p$-adic Jacquet--Langlands correspondence (Proposition \ref{prop.padicJacquetLanglands}), this reduces to the equivalent statement about $\mathscr{f}$, which is \cite[Theorem 4.6.6]{LoefflerZerbes.Rankin} (taking into account our conventions stablished in \S\ref{subsec.GaloisRepresentations}). 
\end{proof}

\begin{definition}\label{def.BigGaloisRep}
The $\mathcal{O}(\mathscr{U})$-module $\mathbf{V}_\mathscr{F}^*$ from Proposition \ref{prop.BigGaloisRep} is the \textit{dual big Galois representation} associated to $\mathscr{F}$, and its linear dual $\mathbf{V}_\mathscr{F}\cong\mathbf{V}_\mathscr{F}^*(1)$ is the \textit{big Galois representation} associated to $\mathscr{F}$.
\end{definition}

The constructions leading up to Definition \ref{def.jBigHeegnerClass} can be redone for $\widetilde{\mathscr{U}}$, yielding a class $\mathbf{z}^{[\widetilde{\mathscr{U}},j]}_{cp^n,1}$ for $0\le j\le 2r$. The natural inclusion $\mathscr{U}\subseteq\widetilde{\mathscr{U}}$ induces for all $s\in\Z_{\ge 0}$ a restriction map $A_{\widetilde{\mathscr{U}},s}\rightarrow A_{\mathscr{U},s}$ in the level of continuous slope $s$ functions, and a dual map $D_{\mathscr{U},s}\to D_{\widetilde{\mathscr{U}},s}$ in the level of distributions. Under the induced dual map in the level of cohomologies, the class $\mathbf{z}^{[{\mathscr{U}},j]}_{cp^n,1}$ maps to $\mathbf{z}^{[\widetilde{\mathscr{U}},j]}_{cp^n,1}$.

Similarly to \S\ref{subsec.ClassesModularForms}, the projection given in Proposition \ref{prop.BigGaloisRep} induces a map
\begin{equation}\label{eq.ProjectionBigGaloisRep}
\mathrm{pr}^{[\mathscr{F},j]}\colon H^1\left(F_{cp^n},D_{\widetilde{\mathscr{U}},1}(1)\otimes\sigma^{\kappa_{\widetilde{\mathscr{U}}}-j}_{\et}\bar{\sigma}^{j}_{\et}\right)\longrightarrow H^1\left(F_{cp^n},\mathbf{V}_{\mathscr{F}}^\ast\otimes\sigma_{\et}^{\kappa_{\widetilde{\mathscr{U}}}-j}\bar\sigma^j_{\et}\right).
\end{equation}
The reason one extends to $\widetilde{\mathscr{U}}$ in the first place is to get big Heegner classes that are divisible by the operator \[\binom{\nabla}{j}\defeq\dfrac{1}{j!}\prod_{i=0}^{j-1}(\nabla-i)\]   (\textit{cf.} \textit{ibid.} Proposition 5.2.5), which acts invertibly on $\Lambda_{\widetilde{\mathscr{U}}}[1/p]$ after shrinking $\widetilde{\mathscr{U}}$ enough so to avoid all integers $0,\dots,j-1$ (\textit{ibid.}, Remark 5.2.2) and specializes at weight $k$ to $\binom{2r}{j}$. This will be fundamental to control the growth of these classes when projected into the Coleman family $\mathscr{F}$ in the proof of Proposition \ref{prop.BigHeegnerClass} to follow; such growth is due to the slope of $\mathscr{F}$. This motivates adding a factor of $\binom{\nabla}{j}^{-1}$ to the map above, leading to the following definition:

\begin{definition}
The \emph{$j$-component} of the big generalized Heegner class is 
\[\mathbf{z}^{[\mathscr{F},j]}_{cp^n,1}\defeq\binom{\nabla}{j}^{-1}\mathrm{pr}^{[\mathscr{F},j]}\left(\mathbf{z}^{[\widetilde{\mathscr{U}},j]}_{cp^n,1}\right).\]
\end{definition}

These classes satisfy the following Euler relations:

\begin{proposition}\label{prop.EulerRelationsBig}
For all $0\le j\le 2r$ and $n\ge 1$, we have \[\norm_{F_{cp^{n+1}}/F_{cp^{n}}}\left(\mathbf{z}^{[\mathscr{F},j]}_{cp^{n+1},1}\right)=\mathbf{a}_p\cdot\mathbf{z}^{[\mathscr{F},j]}_{cp^n,1}.\]
\end{proposition}
\begin{proof}
This follows directly from Proposition \ref{prop.EulerRelations} and the fact that, by definition, $U_p'$ acts over the dual Galois representation by the eigenvalue $\mathbf{a}_p$ (\textit{cf.} \S\ref{subsec.GaloisRepresentations}).
\end{proof}

\begin{proposition}\label{prop.SpecializationNabla}
Let $k\in {\mathscr{U}}\cap\Z$ with $k-2\ge j$. The specialization at weight $k$ of $\mathbf{z}^{[\mathscr{F},j]}_{cp^n,1}$ is $\binom{2r}{j}^{-1}z^{[\mathcal{F}_k,j]}_{cp^n,1}\in H^1(F_{cp^n}, V_{\mathcal{F}_k}^\ast\otimes\sigma_{\et}^{2r-j}\bar{\sigma}_{\et}^j)$.
\end{proposition}
\begin{proof}
By construction, the specialization at weight $k$ of $\mathbf{z}^{[\mathscr{F},j]}_{cp^n,1}$ is the image of the specialization at weight $k$ of $\binom{\nabla}{j}^{-1}\mathbf{z}^{[\widetilde{\mathscr{U}},j]}_{cp^n,1}$ under $\mathrm{pr}^{[\mathcal{F}_k,j]}$, which is exactly $\binom{2r}{j}^{-1}z^{[\mathcal{F}_k,j]}_{cp^n,1}$.
\end{proof}

\subsection{Definition of big generalized Heegner classes}
The remaining step towards the general definition is the interpolation of the characters $\sigma^{-j}\bar\sigma^j$ for $j\geq 0$.

Let $F_{cp^\infty}=\bigcup_{n=1}^\infty F_{cp^n}$. Since $H_{cp^n}\cap F_{\mathfrak{N}^+}=K$ for all $n\ge 0$, $\Gamma_\infty\defeq\Gal(F_{cp^\infty}/F_{c})$ is isomorphic to $\Gal(H_{cp^\infty}/H_{c})$. 
For any integer $n\geq 1$, let $\Gamma_n$ denote the subgroup 
$\Gal(F_{cp^\infty}/F_{cp^n})\cong\Gal(H_{cp^\infty}/H_{cp^n})$ of $\Gamma_\infty$. Via the Artin reciprocity map, there is an isomorphism 
$\Gamma_{1}\cong (\mathcal{O}_{K}\otimes\Z_p)^\times/\Z_p^\times,$
so the character $\sigma/\bar\sigma\colon  (\mathcal{O}_{K}\otimes\Z_p)^\times/\Z_p^\times\rightarrow \Z_p^\times$ induces a Galois character 
$\sigma_{\et}/\bar\sigma_{\et}\colon \Gamma_{1}\rightarrow\Z_p^\times$.

The Iwasawa algebra $\Lambda_{\Gamma_1}\defeq\Z_p\pwseries{\Gamma_{1}}$ can be seen as a $\Gal(\overline{\Q}/F_{cp})$-module via the canonical projection into $\Gamma_1$, which embeds as group-like elements in $\Lambda_{\Gamma_1}$. Twisting $\Lambda_{\Gamma_1}$ by the inverse of 
$\sigma_{\et}/\bar\sigma_{\et}$ gives a Galois module $\Lambda_{\Gamma_1}(\bar\sigma_{\et}/\sigma_{\et})$, whose elements can be seen as  $\Gal(\overline{\Q}/F_{cp})$-representations. Define the specialization at $j\in\Z_{\ge 0}$ to be the map

\begin{equation}\label{eq.jSpecialization}
\Lambda_{\Gamma_1}(\bar\sigma_{\et}/\sigma_{\et})
\longrightarrow \overline{\Q}\sigma_{\et}^{-j}\bar\sigma_{\et}^{j}\colon
\mu\mapsto \left(\int_{\Gamma_{n}}\ \mathrm{d}\mu\right)(\bar\sigma_{\et}/\sigma_{\et})^j
\end{equation}

Define $\sigma^{\kappa_{{\mathscr{U}}}-\mathbf{j}}\bar\sigma^{\mathbf{j}}\defeq\sigma^{\kappa_{{\mathscr{U}}}}\widehat\otimes_{\Z_p}\Lambda_{\Gamma_1}(\bar\sigma_{\et}/\sigma_{\et})$. Map \eqref{eq.jSpecialization} induces correspondent specialization at $j$ maps for $\sigma^{\kappa_{{\mathscr{U}}}-\mathbf{j}}\bar\sigma^{\mathbf{j}}$ and therefore a map
\begin{equation}\label{momj}
\mathrm{mom}_{j,n}\colon H^1\left(F_{cp},\mathbf{V}_{\mathscr{F}}^\ast\widehat\otimes_{\Z_p}
\sigma_{\et}^{\kappa_{{\mathscr{U}}}-\mathbf{j}}\bar\sigma_{\et}^{\mathbf{j}}
\right)\longrightarrow H^1(F_{cp},\mathbf{V}_{\mathscr{F}}^\ast\otimes\sigma_{\et}^{\kappa_{{\mathscr{U}}}-j}\bar\sigma_{\et}^j).
\end{equation}
which can be composed with $\norm_{F_{cp^n}/F_{cp}}$ so to land in the same field of definition $F_{cp}$.

\begin{proposition}\label{prop.BigHeegnerClass}
There exists
$\mathbf{z}_{cp}^{[\mathscr{F},\mathbf{j}]}\in H^1\left(F_{cp},\mathbf{V}_{\mathscr{F}}^\ast\widehat\otimes_{\Z_p}
\sigma_{\et}^{\kappa_{{\mathscr{U}}}-\mathbf{j}}\bar\sigma_{\et}^{\mathbf{j}}\right)$ 
such that, for all $n\geq 1$ and all $j\geq 0$, $\mathrm{mom}_{j,n}(\mathbf{z}_{cp}^{[\mathscr{F},\mathbf{j}]})=\mathbf{a}_p^{-n}\mathbf{z}^{[\mathscr{F},j]}_{cp^n,1}$. 
\end{proposition}

\begin{proof}
The existence of the class will follow from a general construction of Loeffler--Zerbes, namely \cite[Proposition 2.3.3]{LoefflerZerbes.Rankin}. To check the conditions for it, endow $\mathcal{O}(\mathscr{U})$ with the supremum norm and take $\lambda$ to be the slope of $\mathscr{F}$, that is $\log_p(\|\mathbf{a}_p^{-1}\|)$, a positive real number which by hypothesis is strictly smaller than $k-1$, thus $2r\ge\lfloor\lambda\rfloor$. Define the classes \[c_{n,j}\defeq \mathbf{a}_p^{-n}\binom{\nabla}{2r}\mathbf{z}^{[\mathscr{F},j]}_{cp^n,1}\in H^1(F_{cp^\infty},\mathbf{V}_\mathscr{F}^\ast{\otimes}\sigma^{\kappa_{{\mathscr{U}}}}),\]
considered over $F_{cp^\infty}$ where the characters $\sigma_{\et}$ and $\bar{\sigma}_{\et}$ are trivial. For the first condition in \textit{loc. cit.}, it suffices to check that \[\norm_{F_{cp^{n+1}}/F_{cp^n}}(c_{n+1,j})=c_{n,j},\] which is immediate from Proposition \ref{prop.EulerRelationsBig}.

For the second condition, we wish to show that there exists a constant $C$ such that \[\left\|p^{-2rn}\sum_{j=0}^{2r}(-1)^j\binom{2r}{j}\Res_{F_{cp^\infty}/F_{cp^n}}(c_{n,j})\right\|\le Cp^{\lfloor\lambda n\rfloor},\] where $\|\,\cdot\,\|$ is the supremum seminorm on $H^1(F_{cp^\infty},\mathbf{V}_\mathscr{F}^\ast{\otimes}\sigma^{\kappa_{{\mathscr{U}}}})$ coming from a choice of norm on $\mathbf{V}_\mathscr{F}^\ast{\otimes}\sigma^{\kappa_{{\mathscr{U}}}}$ as a Banach $\mathcal{O}(\mathscr{U})$-module. By Lemma \ref{lem.vartheta}, $e_{cp^n}-\bar{e}_{cp^n}\in p^n(\mathcal{O}_L^2)^\vee$, therefore\[(e_{cp^n}-\bar{e}_{cp^n})^{\odot 2r}=\sum_{j=0}^{2r} (-1)^j(e_{cp^n}^{2r-j}\odot \bar{e}_{cp^n}^j)\in p^{2rn}\TSym^{2r}((\mathcal{O}_L^2)^\vee).\] Taking the product of that sum with $\mathbf{e}_{cp^n}^{[\mathscr{U}-2r,0]}$ and applying the overconvergent projector $\Pi_j$ gives
\begin{equation}\label{eq.SumBinom}
\sum_{j=0}^{2r} (-1)^j\Pi_j(\mathbf{e}_{cp^n}^{[\mathscr{U}-2r,0]}\odot(e_{cp^n}^{2r-j}\odot \bar{e}_{cp^n}^j))=\sum_{j=0}^{2r}(-1)^j\binom{\nabla-j}{2r-j}\mathbf{e}_{cp^n}^{[\mathscr{U},j]},
\end{equation}
where the equality (without the alternating sum) is \cite[Lemma 5.1.5]{LoefflerZerbes.Rankin}. Since $\Pi_j$ has denominators bounded in terms of $j$ and $m=1$, there is a constant $C_{j}$ such that \[\Pi_j(\mathbf{e}_{cp^n}^{[\mathscr{U}-2r,0]}\odot(e_{cp^n}^{2r-j}\odot \bar{e}_{cp^n}^j))\in C_{j}p^{2rn}D_{\mathscr{U},1}^\circ,\] where $D_{\mathscr{U},1}^\circ$ denotes the norm-1 distributions in $D_{\mathscr{U},1}$. Applying the big Gysin map \eqref{eq.BigGyin}, the dual map $D_{\mathscr{U},1}^\circ\to D_{\widetilde{\mathscr{U}},1}^\circ$ induced by the inclusion $\mathscr{U}\hookrightarrow\widetilde{\mathscr{U}}$, and $\mathrm{pr}^{[\mathscr{F},j]}$ to the right-hand side of \eqref{eq.SumBinom} and using the explicit calculation \[\binom{\nabla}{j}\binom{\nabla-j}{2r-j}=\dfrac{\nabla\cdot\ldots\cdot(\nabla-j+1)}{j!}\cdot\dfrac{(\nabla-j)\cdot\ldots\cdot(\nabla-2r+1)}{(2r-j)!}\cdot\dfrac{(2r)!}{(2r)!}=\binom{2r}{j}\binom{\nabla}{2r},\] we have \[\left\|\sum_{j=0}^{2r}(-1)^j\binom{2r}{j}\binom{\nabla}{2r}\mathbf{z}^{[\mathscr{F},j]}_{cp^n,1}\right\|\le C_{j}p^{2rn}.\] To conclude the calculation,
\begin{align*}
\left\|p^{-2rn}\sum_{j=0}^{2r}(-1)^j\binom{2r}{j}\Res_{F_{cp^\infty}/F_{cp^n}}(c_{m,j})\right\|&\le p^{-2rn}\left\|\sum_{j=0}^{2r}(-1)^j\binom{2r}{j}\binom{\nabla}{2r}\mathbf{z}^{[\mathscr{F},j]}_{cp^n,1}\mathbf{a}_p^{-n}\right\|\\
&\le C_{j}\|\mathbf{a}_p^{-n}\| = C_{j}p^{\lambda n}
\end{align*}
which implies the inequality we want. Thus we may apply \cite[Proposition 2.3.3]{LoefflerZerbes.Rankin} to obtain a unique class $\mathbf{z}_{cp}^{[\mathscr{F},\mathbf{j}]}\in H^1(F_{cp^\infty}, D_{\mathscr{W}_{\Gamma_1},\lambda}\widehat{\otimes}_{\Z_p}(\mathbf{V}_\mathscr{F}^\ast{\otimes}\sigma^{\kappa_{{\mathscr{U}}}}))$ interpolating the classes $\Res_{F_{cp^\infty}/F_{cp^n}}(c_{n,j})$, where $\mathscr{W}_{\Gamma_1}$ denotes the weight space associated to $\Lambda_{\Gamma_1}$. Since all quaternionic modular forms are cuspidal (as the underlying Shimura varieties are compact) we have $H^0(F_{cp^\infty},\mathbf{V}_\mathscr{F}^\ast{\otimes}\sigma^{\kappa_{{\mathscr{U}}}})=0$, which, via the inflation-restriction sequence, allows the class above to be defined over $F_{cp}$, which we can map to the larger submodule \[\mathbf{z}_{cp}^{[\mathscr{F},\mathbf{j}]}\in H^1\left(F_{cp},\mathbf{V}_{\mathscr{F}}^\ast\widehat\otimes_{\Z_p}
\sigma_{\et}^{\kappa_{{\mathscr{U}}}-\mathbf{j}}\bar\sigma_{\et}^{\mathbf{j}}\right).\]It then follows by the interpolation property in the definition of the class that \[\mathrm{mom}_{j,n}(\mathbf{z}_{cp}^{[\mathscr{F},\mathbf{j}]})=\mathbf{a}_p^{-n}\binom{\mathbf{k}-2}{2r}\mathbf{z}^{[\mathscr{F},j]}_{cp^n,1},\] where $\mathbf{k}$ is the weight of the Coleman family, which is constant and equal to $k=2r+2$, thus making the binomial term vanish, concluding the proof.
\end{proof}

\begin{definition}
The class $\mathbf{z}_{cp}^{[\mathscr{F},\mathbf{j}]}\in H^1\left(F_{cp},\mathbf{V}_{\mathscr{F}}^\ast\widehat\otimes_{\Z_p}
\sigma_{\et}^{\kappa_{{\mathscr{U}}}-\mathbf{j}}\bar\sigma_{\et}^{\mathbf{j}}\right)$ from Proposition \ref{prop.BigHeegnerClass} is called the \textit{big generalized Heegner class} associated to $\mathscr{F}$.
\end{definition}

\subsection{Specialization of big generalized Heegner classes}

For each $k\in\Z\cap\mathscr{U}$, composing $\mom_{j,n}$ with the weight $k=2r+2$ specialization, we also have a map 
\[\mathrm{mom}_{j,n}^{2r}\colon H^1\left(F_{cp},\mathbf{V}_{\mathscr{F}}^\ast\widehat\otimes_{\Z_p}
\sigma_{\et}^{\kappa_{{\mathscr{U}}}-\mathbf{j}}\bar\sigma_{\et}^{\mathbf{j}}
\right)\longrightarrow H^1(F_{cp^n},V_{\mathcal{F}_{k}}^*\otimes\sigma_{\et}^{2r-j}\bar\sigma_{\et}^j).\]

\begin{theorem}\label{teo.Specialization}
For each $k\in\Z\cap\mathscr{U}$ with $k\geq j$, we have \[\mathrm{mom}^{2r}_{j,n}\left(\mathbf{z}_{cp}^{[\mathscr{F},\mathbf{j}]}\right)=a_p(\mathcal{F}_k)^{-n}\binom{2r}{j}^{-1}z_{cp^n,1}^{[\mathcal{F}_{k},j]}\in H^1(F_{cp^n},V_{\mathscr{F}}^*\otimes\sigma^{2r-j}_{\et}\bar{\sigma}^j_{\et}).\] Furthermore, if $\chi$ is a Hecke character of infinity type $(2r-j,j)$ and finite type $(cp^n, \mathfrak{N}^+, \psi_0)$, the specialization at $\chi$ is \[a_p(\mathcal{F}_k)^{-n}\binom{2r}{j}^{-1}z_{cp^n,1}^{[\mathcal{F}_{k},\chi]}\in H^1(H_{cp^n},V_{\mathscr{F}}^\ast\otimes\chi).\]
\end{theorem}
\begin{proof}
By Proposition \ref{prop.BigHeegnerClass}, $\mathrm{mom}_{j,n}(\mathbf{z}_{cp}^{[\mathscr{F},\mathbf{j}]})=\mathbf{a}_p^{-n}\mathbf{z}^{[\mathscr{F},j]}_{cp^n,1}$. The specialization at weight $k$ of $\mathbf{a}_p$ is by definition $a_p(\mathcal{F}_k)$ and the specialization at weight $k$ of $\mathbf{z}^{[\mathscr{F},j]}_{cp^n,1}$ is $\binom{2r}{j}^{-1}z_{cp^n,1}^{[\mathcal{F}_{k},j]}$ by Proposition \ref{prop.SpecializationNabla}. The second statement follows immediately.
\end{proof}

\begin{corollary}
Let $\mathcal{F}^\sharp\in M_k(N^+,1,L)$ be such that $\mathcal{F}$ as above is its $p$-stabilization with respect to a root $\upalpha$ of its Hecke polynomial. For each $k\in\Z\cap\mathscr{U}$ with $k\geq j$ and $\chi$ a Hecke character of infinity type $(2r-j,j)$, the big generalized Heegner class $\mathbf{z}_{cp}^{[\mathscr{F},\mathbf{j}]}$ specializes at weight $k$ and character $\chi$ to \[E_p(\mathcal{F}^{\sharp},\chi)\cdot\upalpha^{-n}\binom{2r}{j}^{-1}z_{cp^n,0}^{[\mathcal{F}^\sharp,\chi]}\in H^1(H_{cp^n},V_{\mathscr{F}}^*\otimes\chi).\]
\end{corollary}
\begin{proof}
Follows directly from Theorem \ref{teo.Specialization} and Proposition \ref{prop.ProjStab}.
\end{proof}

\begin{remark}
The class $z_{cp^n,1}^{[\mathcal{F}_k,\chi]}$, originally defined over $H_{cp^n}$, can be defined over any intermediate extension $H_{cp^n}/K'/K$ (including $K$ itself) directly by applying $\norm_{H_{cp^n}/K'}$.
\end{remark}

\subsection{Big generalized Heegner classes defined over $K$}
To further lower the field of definition of the big generalized Heegner classes, it is necessary to remove dependence on $\kappa_{{\mathscr{U}}}$, which is not well-defined over $H_{cp}$. The workaround is to reparametrize the underlying affinoid of the space of distributions $\mathbf{V}_\mathscr{F}^*\widehat{\otimes}_{\Z_p}\sigma^{\kappa_{{\mathscr{U}}}-\mathbf{j}}\bar{\sigma}^{\mathbf{j}}$ by an appropriate affinoid over a suitable, bigger weight space. The $\mathbf{V}_\mathscr{F}^*\widehat{\otimes}_{\Z_p}\sigma^{\kappa_\mathscr{U}}$ part is parametrized by the affinoid $\mathscr{U}\subseteq\mathscr{W}$, while the $\sigma^{-\mathbf{j}}\bar{\sigma}^{\mathbf{j}}$ part is a $\Lambda_{\Gamma_1}$-module, thus corresponding to an affinoid $\mathscr{U}_1\subseteq\mathscr{W}_{\Gamma_1}$, which, for simplicity, we may take to be the whole weight space. The following is a summary of \cite[\S5.4]{JetchevLoefflerZerbes}:

\begin{proposition}
Let $\widetilde{\mathscr{W}}\to\mathscr{W}\times\mathscr{W}_{\Gamma_1}$ be the weight space parametrizing the continuous characters $K^\times\backslash\mathbb{A}_{K,\mathrm{fin}}^\times\to\mathbb{Q}_p^\times$ which restrict trivially to $(\widehat{\mathcal{O}_K})_{(p)}^\times$. Then the preimage of $\mathscr{U}\times\mathscr{W}_{\Gamma_1}$ in $\widetilde{\mathscr{W}}$ is an affinoid $\mathscr{Y}\subseteq \widetilde{\mathscr{W}}$ isomorphic to $\mathscr{U}\times\mathscr{W}_{\Gamma_1}$ and whose universal character $\kappa_{\mathscr{Y}}$ is well defined as a character $\Gal(K^\mathrm{ab}/K)\to\mathcal{O}(\mathscr{Y})^\times$, and therefore is defined over $K$.
\end{proposition}

Let $\widetilde{\mathbf{V}}_\mathscr{F}\defeq\mathbf{V}_{\mathscr{F}}^*\otimes_{\mathcal{O}(\mathscr{U})}\mathcal{O}(\mathscr{Y})(\kappa_{\mathscr{Y}})$ and consider the image of the big generalized Heegner class $\mathbf{z}^{[\mathscr{F},\mathbf{j}]}_{cp}$ in $H^1(F_{cp},\widetilde{\mathbf{V}}_\mathscr{F})$, denoted by the same symbol. By an argument similar to Lemma \ref{lem.Shimura}, $\mathbf{z}^{[\mathscr{F},\mathbf{j}]}_{cp}$ is found to be invariant with respect to $\Gal(F_{cp}/H_{cp})$, so it can be seen as class in $H^1(H_{cp},\widetilde{\mathbf{V}}_\mathscr{F})$. Finally, define \[\mathbf{z}^{[\mathscr{F},\mathbf{j}]}\defeq\cores_{H_{cp}/K}(\mathbf{z}^{[\mathscr{F},\mathbf{j}]}_{cp})\in H^1(K,\widetilde{\mathbf{V}}_\mathscr{F}),\] which, by abuse of nomenclature, we also call the \textit{big generalized Heegner class} associated to $\mathscr{F}$, enjoying the same specializations as the original, only being defined over $K$.

\bibliographystyle{amsalpha}
\bibliography{references}
\end{document}